\documentclass[12pt, leqno]{amsart}

\usepackage{shuffle} 

\usepackage{lscape}
\usepackage{graphicx}
\usepackage{amssymb,amsmath,amsthm,amscd}
\usepackage{mathrsfs}
\usepackage{enumerate}
\usepackage[usenames,dvipsnames]{color}
\usepackage[colorlinks=true, pdfstartview=FitV,
 linkcolor=blue,citecolor=blue,urlcolor=blue]{hyperref}

\usepackage[all]{xy}
\usepackage[normalem]{ulem}  

\usepackage{enumitem}

\allowdisplaybreaks[2]
\usepackage{oldgerm}
\usepackage{xspace}

\usepackage{marginnote}

\newcommand{\nc}{\newcommand}
\numberwithin{equation}{section}

\newenvironment{red}{\relax\color{red}}{\relax}
\newenvironment{blue}{\relax\color{Dandelion}}{\hspace*{.5ex}\relax}

\newcommand{\beb}{\begin{blue}}
\newcommand{\eb}{\end{blue}}

\newcommand{\ber}{\begin{red}}
\newcommand{\er}{\end{red}}

\newcommand{\berm}[1]{\begin{red}{}\marginnote{\fbox{\scshape\lowercase{M}}}%
#1}  

\newcommand{\berE}[1]{\begin{red}{}\marginnote{\fbox{\scshape\lowercase{E}}}%
#1}  

\renewcommand{\le}{\leqslant}
\renewcommand{\ge}{\geqslant}

\theoremstyle{plain}
\newtheorem{lemma}{Lemma}[section]
\newtheorem{prop}[lemma]{Proposition}
\newtheorem{theorem}[lemma]{Theorem}

\newcommand{\Prop}{\begin{prop}}
\newcommand{\enprop}{\end{prop}}
\newcommand{\Lemma}{\begin{lemma}}
\newcommand{\enlemma}{\end{lemma}}
\newcommand{\Th}{\begin{theorem}}
\newcommand{\enth}{\end{theorem}}
\newtheorem{corollary}[lemma]{Corollary}
\newcommand{\Cor}{\begin{corollary}}
\newcommand{\encor}{\end{corollary}}
\newtheorem{definition}[lemma]{Definition}
\newtheorem*{conjecture}{Conjecture}
\newcommand{\Def}{\begin{definition}}
\newcommand{\edf}{\end{definition}}
\newtheorem{sublemma}[lemma]{Sublemma}
\newcommand{\Sublemma}{\begin{sublemma}}
\newcommand{\ensub}{\end{sublemma}}
\newtheorem*{convention}{Convention}

\theoremstyle{definition}
\newtheorem{remark}[lemma]{Remark}
\newtheorem{example}[lemma]{Example}
\newtheorem{Convention}[lemma]{Convention}
\newcommand{\Conv}{\begin{Convention}}
\newcommand{\enconv}{\end{Convention}}
\nc{\Conj}{\begin{conjecture}}
\nc{\enconj}{\end{conjecture}}
\nc{\Rem}{\begin{remark}}
\nc{\enrem}{\end{remark}}
\newcommand{\C}{{\mathbb C}}
\newcommand{\Q}{\mathbb {Q}}

\newcommand{\Z}{{\mathbb Z}}
\newcommand{\B}{{\mathbf{B}}}

\newcommand{\D}{\mathscr{D}}

\newcommand{\one}{{\bf{1}}}

\newcommand{\seteq}{\mathbin{:=}}

\newcommand{\sh}{\operatorname{sh}}
\newcommand{\hd}{{\operatorname{hd}}}
\newcommand{\soc}{{\operatorname{soc}}}

\newcommand{\g}{{\mathfrak{g}}}

\newcommand{\Hom}{\operatorname{Hom}}

\newcommand{\isoto}[1][]{\mathop{\xrightarrow%
[{\raisebox{.3ex}[0ex][.3ex]{$\scriptstyle{#1}$}}]%
{{\raisebox{-.5ex}[0ex][-.5ex]{$\mspace{2mu}\sim\mspace{2mu}$}}}}}

\newcommand{\M}{{\mathscr M}}

\newcommand{\hs}{\hspace*}
\newcommand{\ms}{\mspace}

\newcommand{\To}[1][{\hs{2ex}}]{\xrightarrow{\,#1\,}}

\newenvironment{myequation}
{\relax\setlength{\arraycolsep}{1pt}\begin{eqnarray}}
{\end{eqnarray}}
\newenvironment{myequationn}
{\relax\setlength{\arraycolsep}{1pt}\begin{eqnarray*}}
{\end{eqnarray*}}

\nc{\eq}{\begin{myequation}}
\nc{\eneq}{\end{myequation}}
\nc{\eqn}{\begin{myequationn}}
\nc{\eneqn}{\end{myequationn}}

\newcommand{\on}{\operatorname}

\newcommand{\bni}{\be[label=\rm(\roman*)]}
\newcommand{\bnum}{\bni}
\newcommand{\bna}{\be[label=\rm(\alph*)]}

\newcommand{\soplus}{\mathop{\mbox{\normalsize$\bigoplus$}}\limits}

\newcommand{\ba}{\begin{array}}
\newcommand{\ea}{\end{array}}

\newcommand{\monoto}{\rightarrowtail}

\newcommand{\eqsub}{\begin{subequations}\begin{eqnarray}}
\newcommand{\eneqsub}{\end{eqnarray}\end{subequations}}

\newcommand{\ol}{\overline}

\newcommand{\ko}{{{\mathbf{k}}}}

\nc{\la}{\lambda}
\nc{\lam}{\lambda}
\nc{\U}[1][\g]{U_q(#1)}
\nc{\te}{\tilde{e}}
\nc{\tei}{\tilde{e}_i}
\nc{\tf}{\tilde{f}}
\nc{\tfi}{\tilde{f}_i}
\nc{\tU}{\widetilde U_q(\g)}
\nc{\tE}{\widetilde{E}}
\nc{\tF}{\widetilde{F}}
\nc{\tK}{\widetilde{K}}
\nc{\tEs}{\widetilde{E}^*}
\nc{\tFs}{\widetilde{F}^*}

\nc{\ttE}{\widetilde{\mathcal{E}}}
\nc{\ttF}{\widetilde{\mathcal{F}}}
\nc{\ttEs}{\ttE^*}
\nc{\ttFs}{\ttF^*}
\nc{\tfs}{\tf^*}
\nc{\tes}{\te^*}

\nc{\tk}{\tilde{k}}
\nc{\tkone}{\tk_{\ol{1}}}
\nc{\teone}{\tilde{e}_{\ol{1}}}
\nc{\tfone}{\tilde{f}_{\ol{1}}}

\nc{\teibar}{\tilde{e}_{\ol{i}}} \nc{\tfibar}{\tilde{f}_{\ol{i}}}
\nc{\tki}{{\tk}_{\ol {i}}}

\nc{\BZ}{{\mathbb{Z}}}
\nc{\al}{\alpha}
\nc{\qs}{{q}}
\nc{\lan}{\langle}
\nc{\ran}{\rangle}
\nc{\re}{{\mathrm{re}}}
\nc{\wt}{\operatorname{wt}}
\nc{\hwt}{\widehat{\wt}}
\nc{\Ht}{\mathrm{ht}}
\nc{\hHt}{\widehat{\Ht}}
\nc{\ch}{\operatorname{ch}}
\nc{\Um}[1][\g]{U^-_q(#1)}
\nc{\Ue}{U^+_q(\g)}
\nc{\ep}{\varepsilon}
\nc{\hep}{\widehat{\ep}}
\nc{\vphi}{\varphi}
\nc{\sphi}{\varphi^*}
\nc{\eps}{\ep^*}
\nc{\heps}{\hep^{ \hskip 0.2em  *}}

\nc{\nn}{\nonumber}
\def\max{{\mathop{\mathrm{max}}}}
\nc{\vp}{\varpi}
\nc{\cls}{{\operatorname{cl}}}
\nc{\Wt}{{\operatorname{Wt}}}
\nc{\Us}{U'_q(\g)}
\nc{\La}{\Lambda}
\nc{\tLa}{\widetilde\Lambda}
\nc{\ro}{{\rm(}}
\nc{\rf}{{\rm)}}
\nc{\norm}{{\mathrm{norm}}}
\nc{\qbox}{\quad\mbox}
\nc{\braid}{{\mathfrak{B}}}
\nc{\Ad}{\operatorname{Ad}}
\nc{\Aut}{\operatorname{Aut}}
\nc{\dt}[1]{\tilde{\tilde #1}}
\nc{\Sn}{S^{{\mathrm{norm}}}}
\nc{\aff}{{\mathrm{aff}}}
\nc{\rk}{{\mathrm{rk}}}
\nc{\tP}{\widetilde{P}}
\nc{\tW}{\widetilde{W}}
\nc{\Dyn}{\mathrm{Dyn}}
\nc{\tD}{\widetilde{\Delta}}
\nc{\height}[1]{{\operatorname{ht}}(#1)}
\nc{\bl}{\bigl(}
\nc{\br}{\bigr)}
\nc{\Hecke}{\mathrm{H}}
\nc{\HA}{\Hecke^{\mathrm{A}}}
\nc{\HB}{\Hecke^{\mathrm{B}}}
\newcommand{\scbul}{{\,\raise1pt\hbox{$\scriptscriptstyle\bullet$}\,}}
\nc{\vac}{{\phi}}
\nc{\Bt}{\B_\theta(\g)}
\nc{\be}{\begin{enumerate}}
\nc{\ee}{\end{enumerate}}
\nc{\low}{{\mathrm{low}}}
\nc{\upper}{{\mathrm{up}}}
\nc{\Zodd}{\Z_{\mathrm{odd}}}
\nc{\Ft}[1][n]{\mathbb{P}\mathrm{ol}_{#1}}
\nc{\Ftf}[1][n]{\widetilde{\mathbb{P}\mathrm{ol}}_{#1}}
\nc{\KA}{\on{K}^{\mathrm{A}}}
\nc{\KB}{\on{K}^{\mathrm{B}}}
\nc{\Res}{\on{Res}}
\nc{\Fc}[1][{n,m}]{\mathbf{F}_{#1}}
\nc{\tphi}{\tilde{\varphi}}
\nc{\CO}{\mathscr{O}}
\nc{\inte}{\mathrm{int}}
\nc{\Oint}{\mathcal{O}^{\ge0}_{\inte}}
\nc{\vs}{\vspace*}
\nc{\tLt}{\widetilde{L}}
\nc{\tL}{\widetilde{\Lambda}}
\nc{\tu}{\tilde{u}}
\nc{\noi}{\noindent}
\nc{\heigh}{\mathfrak{t}}
\nc{\lowest}{\mathfrak{l}}
\nc{\rootl}{\mathsf{Q}}
\nc{\rlQ}{\rootl}
\nc{\cl}{\colon}

\nc{\uqpg}{U'_q(\mathfrak g)}
\nc{\uq}{\uqpg}
\nc{\Oh}{\widehat{\mathcal{O}}}
\nc{\pn}{p_{\mathfrak{n}}}

\nc{\KLR}{KLR algebra}
\nc{\KLRs}{KLR algebras}
\nc{\cor}{\mathbf{k}}
\nc{\cora}{{\cor(A)}}
\nc{\haut}{\mathrm{ht}}
\nc{\tens}{\mathop\otimes}
\nc{\gmod}{\mbox{-$\mathrm{gmod}$}}
\nc{\gMod}{\mbox{-$\mathrm{gMod}$}}
\nc{\proj}{\mbox{-$\mathrm{proj}$}}
\nc{\gproj}{\mbox{-$\mathrm{gproj}$}}
\nc{\smod}{\mbox{-$\mathrm{mod}$}}
\nc{\Mod}{\mbox{-$\mathrm{Mod}$}}
\nc{\h}{\mathfrak h}
\nc{\Rnorm}{R^{\mathrm{norm}}}
\nc{\Runiv}{R^{\mathrm{univ}}}
\nc{\Rren}{R^{\mathrm{ren}}}

\nc{\Vhat}{\widehat{V}}
\nc{\F}{\mathcal{F}}

\def\T{{\mathcal T}}

\nc{\fd}[1][A]{\on{\mathrm{flat.dim}_{#1}}}
\nc{\bP}{{\mathbb{P}}}
\nc{\bPh}{\widehat{\mathbb{P}}}
\nc{\bK}[1][{n}]{\widehat{\mathbb{K}}_{#1}}
\nc{\bV}[1][{n}]{\widehat{V}^{\otimes{#1}}}
\nc{\bVK}[1][{n}]{\widehat{V}^{\otimes{#1}}_{\widehat{\mathbb{K}}}}
\nc{\hV}{\widehat{V}}
\nc{\opp}{\mathrm{opp}}
\nc{\col}{\colon}
\nc{\oep}{\epsilon}
\nc{\qtext}{\quad\text}
\nc{\qtextq}[1]{\quad\text{#1}\quad}
\nc{\longtwoheadrightarrow}[1][]{\xymatrix{\ar@{->>}[r]^-{{#1}}&}}
\nc{\epiTo}[1][]{\longtwoheadrightarrow[{#1}]}
\nc{\epito}{\twoheadrightarrow}
\nc{\monoTo}[1][]{\xymatrix{\ar@{>->}[r]^-{{#1}}&}}
\nc{\sym}{\mathfrak{S}}
\nc{\inp}[1]{{({#1})_{\mathrm{n}}}}
\nc{\rtl}{\rootl}
\nc{\wtd}{\widetilde}
\nc{\etens}{\boxtimes}
\nc{\ds}[1]{\mathrm{d}(#1)}
\nc{\rmat}[1]{{\mathbf{r}}_%
{\mspace{-2mu}\raisebox{-.6ex}{${\scriptstyle{#1}}$}}}
\nc{\rmats}[1]{{\mathbf{r}}_%
{\mspace{-2mu}\raisebox{-.6ex}{${\scriptscriptstyle{#1}}$}}}
\nc{\shc}{\mathcal{C}}
\nc{\shs}{\mathcal{S}}
\nc{\Fct}{{\on{Fct}}}
\nc{\tC}{\widetilde{\shc}}
\nc{\Zp}{\Z_{\ge0}}
\nc{\tPhi}{\widetilde{\Phi}}
\nc{\tT}{{\widetilde{\T}}}
\nc{\Ob}{\on{Ob}}
\nc{\bwr}{\mbox{\large$\wr$}}
\nc{\Img}{\on{Im}}
\nc{\Ab}{\mathcal{A}^{\mathrm{big}}}
\nc{\Sb}{\mathcal{S}^{\mathrm{big}}}
\nc{\As}{\mathcal{A}}
\nc{\Ss}{\mathcal{S}}
\nc{\ntens}{\widetilde{\otimes}}
\nc{\hR}{\widehat{R}}
\nc{\nconv}{\mathop{\mbox{\large $\odot$}}}
\nc{\snconv}{\mbox{\scriptsize$\odot$}}
\nc{\ts}{\tilde{s}}
\nc{\sho}{\mathcal{O}}
\nc{\bc}{\begin{cases}}
\nc{\ec}{\end{cases}}
\nc{\slnh}{{\widehat{\mathfrak{sl}}_N}}
\nc{\UA}{U_q'(\slnh)}
\nc{\KR}{R_K}
\nc{\cQ}{\mathcal{Q}}
\nc{\Irr}{\mathcal{I}rr}
\nc{\tQ}{\widetilde{\cQ}}
\nc{\bs}{\mathbf{s}}
\nc{\bL}{\mathbb{L}}
\nc{\tg}{\tilde{g}}

\nc{\conv}{\mathbin{\mbox{\large $\circ$}}}
\nc{\shconv}{\mathbin{\large\diamond}}
\nc{\sconv}{\mathbin{\large\Delta}}
\nc{\stens}{\mathbin{\large\Delta}}
\nc{\hconv}{\mathbin{\nabla}}
\nc{\htens}{\mathbin{\nabla}}

\nc{\Rm}{R^{\mathrm{ren}}}

\nc{\bQ}{\ol{Q}}
\renewcommand{\Im}{\on{Im}}

\nc{\de}{\on{\textfrak{d}\ms{1mu}}}

\nc{\xmono}{\ar@{>->}}
\nc{\xepi}{\ar@{->>}}
\nc{\db}[1]{\raisebox{-.5ex}[2ex][1.8ex]{$#1$}}
\nc{\wb}[1]{\mbox{$\rule[-1.1ex]{0ex}{2ex}#1$}}
\nc{\univ}{\mathrm{univ}}
\nc{\rM}{{}^*\mspace{-2mu}M}
\nc{\lM}{M^*}
\nc{\uqm}{\uq\smod}
\nc{\tR}{\widetilde{R}_{\gamma,\beta}}
\nc{\tx}{\tilde{x}}
\nc{\bi}{\mathbf{i}}
\nc{\ttau}{\widetilde{\tau}}

\nc{\tEnd}{\on{\widetilde{E}nd}}
\nc{\tHom}{\on{\widetilde{H}om}}

\nc{\K}{{J}}
\nc{\Kex}{{\K}_{\mathrm{ex}}}
\nc{\Kfr}{{\K}_{\mathrm{f\mspace{.01mu}r}}}
\nc{\coro}{\cor}
\nc{\tB}{\widetilde{B}}
\nc{\seed}{\mathscr{S}}

\nc{\up}{\mathrm{up}}
\nc{\bfa}{\mathbf{a}}
\nc{\bfb}{\mathbf{b}}
\nc{\bfc}{\mathbf{c}}

\newlength{\mylength}
\nc{\ov}[1]{\overline{#1}}
\nc{\Wlmj}[3]{\W_{#2,#3}^{(#1)}}
\nc{\Mkl}[2]{\M_\ttww(#1,#2)}
\nc{\mqs}{(-q^2)}
\nc{\Cquiver}{\upsigma}

\nc{\mut}[1]{{\mu}_{\mspace{-2mu}\raisebox{-.5ex}{${\scriptstyle{#1}}$}}}

\nc{\Kt}{\mathcal K_t}
\nc{\KT}{\mathbb{K}_t}
\nc{\yim}{y_{i,m}}
\nc{\yjm}{y_{j,m}}
\nc{\yjp}{y_{j,p}}
\nc{\yimp}{y_{i,m+1}}
\nc{\yjmp}{y_{j,m+1}}

\nc{\Refl}{\mathscr{S}}
\nc{\Reflinv}{{\Refl}^{-1}}

\nc{\catC}{\mathscr C}
\nc{\catA}{\mathcal A}
\nc{\shift}{{\mathrm T}}
\nc{\rE}{ \mathsf{E} }
\nc{\rW}{ \mathcal{W} }
\nc{\rES}{ \mathcal{E} }

\nc{\brd}{\sigma} 
\nc{\into}{\xymatrix@C=3ex{{}\ar@{^{(}->}[r]&{}}}
\nc{\dual}{\D\ms{1.5mu}}
\nc{\cdual}{D}
\nc{\cat}[1][{\g}]{\catC_{#1}^0}

\nc{\catCO}{{\catC_\g^0}}
\nc{\catCOD}{{\catC_\g^{0, \ddD}}}
\nc{\catCQ}{{\catC_{\qQ}}}
\nc{\catCQd}{{\catC_{\widetilde{\qQ}}}}
\nc{\catCD}{{\catC_{\ddD}}}
\nc{\catCDK}{{\catC_{\ddD, \iK}}}
\nc{\Li}{{\La^\infty}}
\nc{\sig}{{\sigma(\g)}}
\nc{\sigZ}{{\sigma_0(\g)}}
\nc{\sigQ}{{\sigma_\qQ(\g)}}
\nc{\sigQd}{{\sigma_{\widetilde{\qQ}}(\g)}}
\nc{\phiQd}{\phi_{\widetilde{\qQ}}}
\nc{\sigD}{{\sigma_\ddD(\g)}}

\nc{\ZZ}{{\mathbf{Z}}}
\nc{\sP}{{\mathsf{P}}}
\nc{\sV}{{\mathsf{V}}}
\nc{\rxw}{{\underline{w_0}}}
\nc{\boten}[1]{\overrightarrow{\bigotimes_{#1}}}
\nc{\cmA}{{\mathsf{A}}}
\nc{\cmC}{{\mathsf{C}}}
\nc{\ddD}{{\mathcal{D}}}
\nc{\ddDQ}{{\ddD_Q}}
\nc{\ddDQd}{{\ddD_{\widetilde{Q}}}}
\nc{\qQ}{{\mathcal{Q}}}
\nc{\gf}{{\g_{\mathrm{fin}}}}
\nc{\Df}{{\Delta_{\mathrm{fin}}}}
\nc{\If}{{I_{\mathrm{fin}}}}
\nc{\cmAf}{{\cmA_{\mathrm{fin}}}}
\nc{\weyl}{{\mathsf{W}}}
\nc{\weylf}{{\mathsf{W}_{\mathrm{fin}}}}
\nc{\sg}{{\mathfrak{S}}}
\nc{\weylA}{{\mathsf{W}_\cmA}}
\nc{\weylC}{{\mathsf{W}_\cmC}}
\nc{\Deg}{\mathrm{Deg}}
\nc{\Di}{\Deg^\infty}
\nc{\KRc}{{K_{q=1}(R_\cmC\gmod)}}
\nc{\prD}{{\Delta^+}}
\nc{\prDf}{{\Delta^+_{\mathrm{fin}}}}
\nc{\nrD}{{\Delta^-}}
\nc{\prDA}{{\Delta^+_\cmA}}
\nc{\prDC}{{\Delta^+_\cmC}}
\nc{\nrDC}{{\Delta^-_\cmC}}
\nc{\n}{{\mathfrak{n}}}
\nc{\Rt}{\mathsf{L}} 
\nc{\Cp}{\mathsf{V}} 
\nc{\cuspS}{{\mathsf{S}}}
\nc{\st}[1]{\{{#1}\}}
\nc{\bst}[1]{\bigl\{{#1}\bigr\}}
\nc{\WS}{Quantum affine Schur-Weyl duality\xspace}
\nc{\CWS}{Quantum affine Weyl-Schur duality}
\nc{\zz}{{{\mathbf{z}}}}
\nc{\wlP}{\mathsf{P}}
\nc{\wl}{\wlP}
\nc{\clp}{{\mathrm{cl}}}
\nc{\wlPc}{{\wlP_\clp}}
\nc{\awlP}{\widehat{\mathsf{P}}}
\nc{\dM}{\mathsf{M}}
\nc{\dC}{\mathsf{C}}
\nc{\cC}{\mathcal{C}}
\nc{\lR}{\widetilde{{R}}}
\nc{\zero}{\mathrm{zero}}
\nc{\PD}{principal }

\nc{\prtl}[1][J]{\rootl_{#1}^+}
\nc{\hL}{\widehat{\Rt}}
\nc{\hF}{\widehat{\F}}
\nc{\Proof}{\begin{proof}}
\nc{\QED}{\end{proof}}
\nc{\e}{\mathrm{e}}

\nc{\Aff}{\mathrm{Aff}}
\nc{\rT}{\mathcal{T}}		
\nc{\rr}{rationally renormalizable\xspace}
\nc{\RA}{{R_\cmA}}		
\nc{\RC}{{R_\cmC}}		
\nc{\proolim}[1][]{\mathop{``{\varprojlim}{\mbox{''}}}\limits_{#1}}
\nc{\qtq}[1][\text{and}]{\quad\text{#1}\quad}

\nc{\corh}{\widehat{\cor}}
\nc{\ang}[1]{\langle{#1}\rangle}
\nc{\rc}{renormalizing coefficient\xspace}
\nc{\cz}{{\cor[z^{\pm1}]}}
\nc{\tp}{\ms{1.5mu}{\widetilde{p}}\ms{2mu}}
\nc{\G}{\mathcal{G}}
\nc{\cc}{\mathfrak{c}}
\nc{\rsP}{{\Phi_\g}}
\nc{\rsX}{{X_\g}}
\nc{\rs}{ \mathsf{s} }
\nc{\Dynkin}{\mathsf{D}}
\nc{\Dat}{\sigma}
\nc{\hf}{\xi}

\nc{\cB}{\widehat{B}}
\nc{\iK}{\mathsf{K}}
\nc{\cBg}[1][\g]{\widehat{B}_{#1}(\infty)}
\nc{\cBsg}[1][\g]{\widehat{B}_{#1}(\infty)^*}
\nc{\cBgk}[1][\g]{\widehat{B}_{#1}^{\iK}(\infty)}
\nc{\cb}{\mathbf{b}}
\nc{\cI}{ \widehat{I} }
\nc{\cIf}{{\widehat{I}_{\mathrm{fin}}}}
\nc{\cIz}{{\widehat{I}_{0}}}
\nc{\cJ}{ \widehat{J} }
\nc{\sB}{\mathcal{B}}
\nc{\sBk}{\sB_{\iK}}
\nc{\sBdk}{\sB_{\ddD, \iK}}
\nc{\sBD}{\sB_{\ddD}}
\nc{\sBkg}{\sBk(\g)}
\nc{\cs}{\star}

\nc{\cd}{\mathrm{D}}
\nc{\cm}{\mathbf{m}}
\nc{\MS}{\mathsf{MS}}
\nc{\hMS}{\widehat{\mathsf{MS}}}
\nc{\rS}{\mathbf{S}}
\nc{\rSs}{\rS^*}
\nc{\brS}{\overline{\rS}}

\nc{\crI}{ \mathsf{I}}
\nc{\crhI}{\mathscr{S} }
\nc{\crD}{\mathsf{D}}
\nc{\crc}{\mathsf{c}}
\nc{\crB}{\mathscr{P}_n}
\nc{\cru}{\mathsf{u}}
\nc{\crsh}{\mathsf{sh}}
\nc{\bfs}{\mathbf{s}}
\nc{\crBB}[1]{\mathscr{P}_#1}

\nc{\gc}{{\g_{\cmC}}}
\nc{\cL}{\mathcal{L}}
\nc{\cLD}{{\cL_\ddD}}
\nc{\ccL}{\mathscr{L}}
\nc{\ccLD}{{\ccL_\ddD}}
\nc{\clen}{\mathsf{len}}

\nc{\hv}{\mathsf{1}}
\nc{\qt}[1]{\quad\text{#1}}
\nc{\snoi}{\smallskip\noindent}
\nc{\mnoi}{\medskip\noindent}

\nc{\ul}[1]{\underline{#1}}
\nc{\dul}[1]{\underline{\underline{#1}}}
\nc{\qh}{\qedhere}
\nc{\ca}{\mathsf{v}}
\nc{\sck}[1][k]{\ms{8mu}{\raisebox{-1.3ex}{$\scriptstyle{#1}$}\hs{-1.4ex}{\succ}}\ms{4mu}}
\nc{\scke}[1][k]{\ms{8mu}{\raisebox{-1.3ex}{$\scriptstyle{#1}$}\hs{-1.4ex}%
{\succcurlyeq}}\ms{4mu}}

\nc{\edot}{\emptyset}
\nc{\Nf}{N_{\gf}}
\nc{\fin}{\mathrm{fin}}
\nc{\trg}{\scalebox{.7}{$\triangle$}}
\nc{\ake}[1][1ex]{\rule[-#1]{0ex}{1ex}}
\nc{\akew}[1][1ex]{\rule[-1ex]{#1}{0ex}}
\nc{\akeu}[1][1ex]{\rule[#1]{0ex}{1ex}}

\title[Categorical crystals for quantum affine algebras]{Categorical crystals for quantum affine algebras}

\author[M. Kashiwara]{Masaki Kashiwara}
\thanks{The research of M.\ Kashiwara
was supported by Grant-in-Aid for Scientific Research (B)
20H01795, Japan Society for the Promotion of Science.}
\address[M. Kashiwara]{
Kyoto University Institute for Advanced Study,
Research Institute for Mathematical Sciences, Kyoto University,
Kyoto 606-8502, Japan \& Korea Institute for Advanced Study, Seoul 02455, Korea }
\email
{masaki@kurims.kyoto-u.ac.jp}

\author[E. Park]{Euiyong Park}
\thanks{The research of E. Park was supported by the National Research Foundation of Korea(NRF) Grant funded by the Korea Government(MSIT)(NRF-2020R1F1A1A01065992 and NRF-2020R1A5A1016126).}
\address[E. Park]{Department of Mathematics, University of Seoul, Seoul 02504, Korea}
\email
{epark@uos.ac.kr}

\keywords{Crystal, Hernandez-Leclerc category, PBW theory, Quantum affine algebra} 
\subjclass[2010]{17B37, 05E10, 18D10} %

\date{November 13, 2021}

\begin{document}

 \maketitle

\begin{abstract}

 A new categorical crystal structure for the quantum affine algebras 
is presented. We introduce the notion of extended crystals $\cBg$ for an arbitrary quantum group  $U_q(\g)$, which is the product of infinite copies of the crystal $B(\infty)$.
For a complete duality datum $ \ddD$ in the Hernandez-Leclerc category  $\catCO$ of a quantum affine algebra $U_q'(\g)$,
we prove that the set $\sBD(\g)$ of the isomorphism classes  of simple modules in $\catCO$ has an extended crystal structure 
isomorphic to $\cBg[\gf]$. 
An explicit combinatorial description of the extended crystal $\sBD(\g)$ for affine type $A_n^{(1)}$ is given  in terms of affine highest weights.

\end{abstract}

\tableofcontents

\section{Introduction}

Let $q$ be an indeterminate and let $\catC_\g$ be the category of finite-dimensional integrable modules over a quantum affine algebra $U_q'(\g)$. 
Because of its rich and complicated structure, the category $\catC_\g$ has been actively studied in various research areas (see \cite{AK97, CP94, FR99, HL10, Kas02, Nak01, Nak04} for example).
The simple modules in $\catC_\g$ are indexed by using \emph{Drinfeld polynomials} (\cite{CP91,CP94,CP95, CP98}) and they can be obtained as the head of
ordered tensor product of \emph{fundamental representations} (\cite{AK97, Kas02, Nak01, VV02}). Thus the simple modules in $\catC_\g$ are parameterized by the \emph{affine height weights} (see Theorem \ref{Thm: basic properties} \ref{Thm: bp5} for the definition). 

The category $\catC_\g$ has distinguished subcategories called \emph{Hernandez-Leclerc categories}. We assume that $U_q'(\g)$ is of untwisted affine ADE type.
 In \cite{HL10}, Hernandez and Leclerc introduced the monoidal full subcategory $\catCO$,
which consists of objects whose all simple subquotients are obtained by taking the heads of tensor products of certain fundamental modules.  
Since any simple module in $\catC_\g$ can be obtained as a tensor product of suitable parameter shifts of simple modules in $\catCO$, 
the subcategory $\catCO$ occupies an important place in $\catC_\g$. 
For a Dynkin quiver $Q$, Hernandez and Leclerc introduced the monoidal subcategory $\catC_Q$ of $\catCO$ defined by using the Auslander-Reiten quiver of $Q$ (\cite{HL15}).
It was proved in \cite{HL15} that the complexified Grothendieck ring $\C\otimes_\Z K( \catC_Q )$ is isomorphic to the coordinate ring $\C[N]$ of the unipotent group $N$ associated with 
a certain finite-dimensional simple Lie subalgebra $\g_0$ of $\g$ (see \S\,\ref{subsec:quantum affine}), 
and the set of the isomorphism classes of simple modules in $\catC_Q$ corresponds to the \emph{upper global basis} (or \emph{dual canonical basis}) of $\C[N]$ under the categorification.

The notion of the categories $\catCO$ and $\catC_Q$ is extended to all untwisted and twisted quantum affine algebras by using the poles of R-matrices among fundamental modules and a \emph{Q-datum} $\qQ$ (\cite{FO20, KKK15B, KKKO16D, KO18, OS19, OhSuh19}).
 The monoidal subcategory $\catCQ$ is defined by using fundamental modules determined by a Q-datum $\qQ$. 
Let $\gf$ be the simple Lie algebra of the type $X_\g$ defined in \eqref{Table: root system}. Note that $\gf$ is of simply-laced finite type,
 and it arises naturally from the categorical structure of $\catC_\g$ (\cite{KKOP20}).
It coincides with $\g_0$ when $\g$ is of untwisted affine $ADE$ type.
Similarly to the case of untwisted affine ADE type, 
the complexified Grothendieck ring $\C\otimes_\Z K(\catCQ)$ is isomorphic to the coordinate ring $\C[\Nf]$ of the unipotent group $\Nf$ associated with $\gf$, and the set of the isomorphism classes of simple modules corresponds to the upper global basis of $\C[\Nf]$.

\smallskip 

On the other hand, the notion of \emph{quiver Hecke algebras} (or \emph{Khovanov-Lauda-Rouquier algebras}) was introduced independently by Khovanov-Lauda \cite{KL09, KL11} and Rouquier \cite{R08} for a categorification of the half of a quantum group $U_q(\g)$.   
Let $R$ be the quiver Hecke algebra associated with the quantum group $ U_q(\g)$. The category  $R\gmod$ of finite-dimensional graded $R$-modules categorifies the unipotent quantum coordinate ring $A_q(\n)$, which can be understood as a $q$-deformation of the coordinate ring $\C[N]$ associated with $\g$. 
The quiver Hecke algebras therefore takes an important position in the study of categorification and various new features have been studied (see \cite{BK09, KK11,  KKK18A, KKKO18, KP18, LV09, R11, VV11} for example).

The categorical structure of $R\gmod$ was studied from the viewpoint of the \emph{crystal basis theory}.
The \emph{crystal graph} (or shortly \emph{crystal}) is one of the most powerful combinatorial tools to study quantum groups and their representations, and there are numerous connections and applications to various research areas including representation theory, combinatorics and geometry (see \cite{K91, K93, K94}, and see also \cite{BS17, HK02, Lu93} and the references therein).
The crystal arises naturally from the categorification using quiver Hecke algebras. In \cite{LV09}, Lauda and Vazirani proved  that $R\gmod$ has a \emph{categorical crystal structure} isomorphic to $B(\infty)$, i.e., the set of the isomorphism classes of simple $R$-modules has a $U_q(\g)$-crystal structure isomorphic to the crystal $B(\infty)$. 
The crystal operators $\tf_i$ and $\tfs_i$ are given 
by taking the head of the convolution product with the 1-dimensional $R(\alpha_i)$-module $L(i)$ (see \eqref{Eq: def of tf in QHA} for details).
This result can be understood as a quiver Hecke algebra analogue of the classical results on the categorical crystal structure for affine Hecke algebras and their variants (see \cite{Ariki96, BK01, Gro99, Kle95,Tsu10,V01}, and see also \cite{Ariki02,Kle05} and the references therein). 
Thus the set of the isomorphism classes of simple $R$-modules can be parameterized by the crystal $B(\infty)$ and the categorical relations among simple $R$-modules can be interpreted in terms of the crystal operators. 
This led us to a new combinatorial approach to study on $R\gmod$ from the viewpoint of the crystal basis theory.

\smallskip 

The \emph{PBW theory} for $\catCO$ developed in \cite{KKOP20A} reveals interesting connections between quantum affine algebras and quiver Hecke algebras,
which can be summarized as follows.

\snoi
\textbf{(I)} The notion of \emph{strong} (resp.\ \emph{complete}) duality datum for $\catCO$ was introduced by investigating \emph{root modules}. 
The complete duality datum generalizes the notion of Q-datum one step further. 
Let $ \ddD= \{ \Rt_i \}_{i\in J}$ be a strong duality datum 
associated with a finite $ADE$ type Cartan matrix $\cmC$, and let 
$$
\F_\ddD \col \RC \gmod \longrightarrow  \catC_\g^0
$$
 be the \emph{quantum affine Schur-Weyl duality} functor \cite{KKK18A}, where $\RC$ is the \emph{symmetric} quiver Hecke algebra associated with $\cmC$.
Let $\catCD$ be the image of $\F_\ddD$ (see Section \ref{Sec: PBW} for the precise definition).
It was proved in \cite{KKOP20A} that the duality functor $ \F_\ddD$ sends simple modules in $\RC\gmod$ to simple modules in $\catCD$ and it preserves the invariants $\La$, $\Li$ and $\de$ (see \cite{KKOP19C} for details on $\La$, $\Li$ and $\de$ for $\catCO$). 
The invariants for $\catCO$ take a crucial role in the recent developments about block decomposition, the PBW theory, and monoidal categorification for quantum affine algebras (see \cite{KKOP20, KKOP20A, KKOP20B}).
The duality functor $\F_\ddD$ tells us that $\catCD$ has the \emph{categorical crystal structure} induced from $\RC\gmod$ via $\F_\ddD$.

\snoi
\textbf{(II)} The notion of \emph{affine cuspidal modules} for $\catCO$ was introduced. 
Let $\ddD$ be a complete duality datum. 
Let $\rxw$ be a reduced expression of the longest element $w_0$ of the Weyl group of $\gf$, and let $\{ \Cp_k \}_{k=1, \ldots, \ell}$ be
 the set of the \emph{cuspidal modules} of the quiver Hecke algebra $\RC$ associated with $\rxw$.
The affine cuspidal modules $\{ \cuspS_{k} \}_{k\in \Z}$ for $\catCO$ are defined by using $\{ \Cp_k \}_{k=1, \ldots, \ell}$ via the duality functor $\F_\ddD$ (see Section \ref{Sec: PBW} for the precise definition).
When $\ddD$ arises from a $Q$-datum, the affine cuspidal modules coincide with the fundamental modules in $\catCO$.
It turned out that all simple modules in $\catCO$ can be obtained uniquely as the simple heads of the ordered tensor products $\sP_{\ddD, \rxw} (\bfa)$ of affine cuspidal modules. 
Moreover, the module  $\sP_{\ddD, \rxw} (\bfa)$, called \emph{standard module}, 
has the \emph{unitriangularity property}. 
This generalizes the classical simple module construction taking the head of ordered tensor products of fundamental representations (\cite{AK97, Kas02, Nak01, Nak04, VV02}).
From the viewpoint of the PBW theory, it is natural that the Grothendieck ring $K(\catCO)$ can be viewed as a product of infinite copies of $K(\catCD)$.
Note that, when $\ddD$ arises from a Dynkin quiver of a finite ADE type, it was already observed by Hernandez-Leclerc (see \cite{HL15}).

\medskip 

The purpose of this paper is to study a new categorical crystal structure of the category $\catCO$ by using the new invariants and the PBW theory.
Our main results can be summarized as follows: 
\bni
\item We introduce the notion of \emph{extended crystals} $\cBg$ for an arbitrary quantum group $U_q(\g)$. 
The extended crystal $\cBg$ can be understood as a product of infinite copies of $B(\infty)$ and its crystal operators are defined by using the crystal structure of $B(\infty)$.
Thus the extended crystal $\cBg$ may look like an \emph{affinization} of the crystal $B(\infty)$.
As a usual crystal, the extended crystal $\cBg$ has the colored graph structure induced from its crystal operators.

\item We prove that the category $\catCO$ has a new categorical crystal structure isomorphic to the extended crystal $\cBg[\gf]$.
Let $ \ddD= \{ \Rt_i \}_{i\in \If}$ be a complete duality datum of $\catCO$ and let $\sBD(\g)$ be the set of the isomorphism classes of simple modules in $\catCO$.
The operators $\ttF_{i,k}$, $\ttE_{i,k}$, $\ttFs_{i,k}$, and $\ttEs_{i,k}$ defined in \eqref{Def: ttF} gives a categorical crystal structure on $\sBD(\g)$.
This crystal structure turns out isomorphic to that of $\cBg[\gf]$. 
Thus the simple modules in $\catCO$ are parameterized by the extended crystal $\cBg[\gf]$ which leads us to a new combinatorial approach to study $\catCO$.

\item We provide explicit formulas to compute the invariants $\La$ and $\de$ between  $\dual^k \Rt_i$ ($k\in \Z$) and an arbitrary simple module in terms of the extended crystal.
\item We give an explicit combinatorial description of the extended crystal $\sBD(\g)$ for affine type $A_n^{(1)}$ in terms of affine highest weights.
\ee

Let us explain our results more precisely. 
Let $I$ be an index set and let $\g$ be the Kac-Moody algebra associated with a symmetrizable generalized Cartan matrix.
The extended crystal is defined as
\begin{align*}
	\cBg \seteq  \left\{  (b_k)_{k\in \Z } \in \prod_{k\in \Z} B(\infty) \biggm| b_k =\hv \text{ for all but finitely many $k$} \right\},
\end{align*}
where $B(\infty)$ is the crystal of $U_q^-(\g)$. Let $\cI \seteq I \times \Z$.
For $\cb = (b_k)_{k\in \Z } \in \cBg$ and $(i,k)\in \cI$, we define the extended crystal operator $\tF_{(i,k)}(\cb)$ (resp.\ $\tE_{(i,k)}(\cb)$) in terms of the usual crystal operators $\tf_i(b_k)$ and $ \tes_i(b_{k+1})$ (resp.\ $\te_i(b_k)$ and $ \tfs_i(b_{k+1})$) according to the values of $\ep_i(b_k)$ and $\eps_i(b_{k+1})$ (see \eqref{Eq: tE and tF}). 
We show that there exist natural injections $\iota_k\col B(\infty) \rightarrowtail \cBg$ ($k\in \Z$) and interesting bijections $\cs$ and $\cd$ on $\cBg$ 
compatible with the extended crystal structure (see Lemma \ref{Lem: basic for cBg} and Lemma \ref{Lem: cD}). 
Similarly to the usual crystals, the extended crystal $\cBg$ has a connected $\cI$-colored graph structure induced from the operators $\tF_{i,k}$ (see Lemma~\ref{Lem: connectedness}).

We next deal with the category $\catCO$ from the viewpoint of the extended crystal. 
Let $U_q'(\g)$ be a quantum affine algebra of \emph{arbitrary} type and let $ \ddD= \{ \Rt_i \}_{i\in \If}$ be a complete duality datum of $\catCO$. 
We define $\sB(\g)$ to be the set of the isomorphism classes of simple modules in $\catCO$.
For $\cb = (b_k)_{k\in \Z}  \in \cBg[\gf]$, we define 
\begin{align*}
	\ccLD(\cb) \seteq  \hd ( \cdots \tens \dual^2 (\cLD( b_{2})) \tens \dual (\cLD( b_{1})) \tens  \cLD( b_{0} )\tens \dual^{-1} (\cLD( b_{-1})) \tens \cdots ), 
\end{align*}
where 
$\hd(X)$ stands for the head of a module $X$, 
$\dual $ is the right dual functor and $\cLD(b_k)$ is the simple module in $\catCD$ corresponding to $b_k$ via the duality functor $\F_\ddD$ (see Lemma \ref{Lem: B and catCD} (i)). 
Proposition \ref{Prop: bijection BB} says that $\ccLD(\cb)$ is simple and the map 
 $$	
 \Phi_\ddD \col  \cBg[\gf] \longrightarrow \sB(\g), \qquad  \cb \mapsto  \ccLD(\cb)
 $$
 is bijective. Moreover the bijection $ \cd$ on $\cBg[\gf]$ is compatible with the functor $\dual$  under the map $\Phi_\ddD$.
For $ M \in \sB(\g)$ and $ (i, k)\in \cIf$, we define 
\begin{equation*}
	\begin{aligned}
		\ttF_{i,k} (M) &\seteq    (\dual^k \Rt_i) \htens M  \quad \text{ and }\quad \ttE_{i,k} (M) \seteq    M \htens (\dual^{k+1} \Rt_i),
	\end{aligned}
\end{equation*}
where $X\htens Y$ denotes the the head of $X \tens Y$. 
Since the operators $\ttF_{i,k}$ and $\ttE_{i,k}$ depend on the choice of $\ddD$,  
we write $\sBD(\g)$ instead of $\sB(\g)$ when we consider $\sB(\g)$ together with the operators $\ttF_{i,k}$ and $\ttE_{i,k}$.
We prove  that $\ttF_{i,k}$ and $\ttE_{i,k}$ are inverse to each other and 
the bijection $\Phi_\ddD $ has compatibility with the extended crystal operators, i.e.,  
\begin{align*}
\Phi_\ddD(  \tF_{i,k}  ( \cb) ) =  \ttF_{i,k} (\Phi_\ddD(\cb)), \qquad &\Phi_\ddD(  \tE_{i,k}  ( \cb) ) =  \ttE_{i,k} (\Phi_\ddD(\cb)) 
\end{align*}
(see Theorem \ref{Thm: main1}).
Therefore $\sBD(\g)$ has the same extended crystal structure as $\cBg[\gf]$ and, for any two quantum affine algebras $U_q'(\g)$ and $U_q'(\g')$, 
$ \sBD(\g) \simeq \sB_{\ddD'}(\g')$ as a colored graph if and only if $\gf \simeq  (\g')_{\mathrm{fin}} $ as a simple Lie algebra. 
In particular, the $\cIf $-colored graph structure of $ \sBD(\g)$ does not depend on the choice of  complete duality data $\ddD$ (see Corollary \ref{Cor: main1}).

Explicit formulas for computing the invariants $\La$ and $\de$ between $\dual^k \Rt_i$ ($k\in \Z$) and a simple module are given in Theorem \ref{Thm: La Li M}. Let $M$ be a simple module in $\catCO$ and let $\cb =  \cb_\ddD(M) \seteq   \Phi_\ddD^{-1}(M) \in \cBg[\gf]$. By investigating properties of root modules, we prove that 
\bna
\item $\La( \dual^k \Rt_i, M) = 2 \max\{ x, r \}  + \sum_{ t \in \Z} (-1)^{\delta(t > k)}  (\alpha_i, \wt_t(\cb))$,
\item $\La( M, \dual^k \Rt_i) = 2 \max\{ y, s \}  + \sum_{t \in \Z} (-1)^{\delta(t < k)}  (\alpha_i, \wt_t(\cb))$,
\item $\de( \dual^k \Rt_i, M ) = \max\{ x,r \} + \max\{ y,s \} + (\alpha_i, \wt_k(\cb))  $.
\ee
where 
$x \seteq  \eps_{(i,k+1)}(\cb)$,  $r \seteq  \ep_{(i,k)}(\cb)$,  $s \seteq  \eps_{(i,k)}(\cb)$, and $y \seteq  \ep_{(i,k-1)}(\cb)$.
We remark that 
these formulas can be understood as a quantum affine algebra analogue of the ones given in \cite[Corollary 3.8]{KKOP18}.

For affine type $A_n^{(1)}$, we provide an explicit combinatorial description of the extended crystal $\sBD(\g)$ in terms of affine highest weights.
Let $U_q'(\g)$ be the quantum affine algebra of type $A_n^{(1)}$ and let $\catCO$ be the Hernandez-Leclerc category corresponding to 
$\sigZ \seteq  \{ (i, (-q)^{a}) \in I_0 \times \cor \mid   a-i \equiv 1 \bmod 2  \}$.
Let  
\begin{align*}
	\crB \seteq (\Z_{\ge0})^{\oplus  \crhI_n},
\end{align*}
where $\crhI_n \seteq  \{ (i, a) \in I_0 \times \Z \mid   a-i \equiv 1 \bmod 2  \} $. Note that the set of fundamental modules in $\catCO$ is $\{ V(\varpi_i)_{(-q)^a} \mid (i,a) \in \crhI_n \}$.  
For any $\la = \sum_{(i,a) \in \crhI_n} c_{(i,a)} (i,a)\in \crB $, 
we denote by $V(\la)$ the simple module in $\catCO$ with the affine highest weight $ \sum_{ (i,a)\in \crhI_n } c_{(i,a)} (i,(-q)^a)$ (see Theorem \ref{Thm: basic properties} \ref{Thm: bp5}), which gives the bijection 
  \begin{align*}
  	\Psi_n\col   \crB  \buildrel \sim \over \longrightarrow   \sBD(\g), \qquad \la \mapsto V(\la) \quad \text{ for $\la \in \crB$}.
  \end{align*}
For $ (i,k) \in \cIz $,
we define the crystal operators $ \tF_{i,k} $ and $ \tE_{i,k} $ on $\crB $ in a combinatorial way (see Section \ref{Sec: crystal rule}). 
We briefly explain the combinatorial rule for $\tF_{i,k}$.
For a given $\la \in \crB$, we first choose a suitable subset $S_{i,k} = \{  \ca_{2n},\ca_{2n-1},\cdots,\ca_{1}  \}$ of $\crhI_n$ and make a sequence of $+$ and $-$ according to $S_{i,k}$ and the coefficients of $\la$. We then cancel out all possible $(+,-)$ pairs to obtains a sequence of $-$'s followed by $+$'s.  
Let $\ca_t$ be the element of $S_{i,k}$ corresponding to the leftmost $+$ in the resulting sequence of $\la$.
If such an $\ca_t$ exists, then we define 
$$
\tF_{i,k} (\la) \seteq  \la - \ca_t + \ca_{t+1}.
$$
Otherwise, we define 
$$
\tF_{i,k} (\la) \seteq  \la + \ca_1.
$$
We remark that this combinatorial rule is quite similar to the $(+,-)$-signature rule for the tensor product of crystals.
Let $\ddD \seteq  \{ V(\varpi_1)_{(-q)^{2i-2}} \}_{i\in I_0}$. Theorem \ref{Thm: CR} tells us that 
there is a bijection $ \Upsilon_n \col  \cBg[\g_0]  \buildrel \sim \over \longrightarrow \crB$ which is compatible with the crystal operators $\tF_{i,k}$ and $\tE_{i,k}$ and, 
for $(i,k)\in \cIz $ and $\la \in  \crB $, 
\begin{align*}
	\ttF_{i,k} (  \Psi_n(    \la )) = \Psi_n(  \tF_{i,k}  ( \la) ) , \qquad  \ttE_{i,k} (  \Psi_n(    \la )) = \Psi_n(  \tE_{i,k}  ( \la) ).
\end{align*}
Hence the following diagram 
$$
\xymatrix{
	\cBg[\g_0]  \ar[rr]^{\Phi_\ddD}  \ar[rd]_{\Upsilon_n} &  & \sBD(\g) \\
	& \crB  \ar[ur]_{\Psi_n} &
}
$$
commutes, and the arrows are $\cIz$-colored graph isomorphisms. 
In the course of proofs, the \emph{multisegment realization} for the crystal $B(\infty)$ is used crucially (see Section \ref{Sec: ms}).

\smallskip 

The paper is organized as follows. 
In Section \ref{Sec:Preliminaries}, we briefly review necessary background on crystals, quantum affine algebras, and the new invariants $\La$, $\Li$ and $\de$. 
In Section \ref{Sec: PBW}, we recall the PBW theory for $\catCO$ developed in \cite{KKOP20A}. 
In Section \ref{Sec: extended crystal}, we introduce the notion of the extended crystals.
In Section \ref{Sec: categorical str},  we prove that the category $\catCO$ has a categorical crystal structure isomorphic to the extended crystal $\cBg[\gf]$. 
In Section \ref{Sec: crystals and invariants}, we give formulas to compute the invariants  $\La$ and $\de$ in terms of the extended crystal.
In Section \ref{Sec: combi des}, we give a combinatorial description of the extended crystal $ \sBD(\g)$ in terms of affine highest weights.

\vskip 2em

\section{Preliminaries} \label{Sec:Preliminaries}

\begin{convention}
\ 
\bnum
\item
For a statement $P$, $\delta(P)$ is $1$ or $0$ according that
$P$ is true or not.
\item For a  \ro semi-\rf ring $R$ and a set $A$, we denote by $R^{\oplus A}$ the direct sum of copies of $R$ indexed by $A$.
\item For a module $X$ of finite length, $\hd(X)$ denotes the head of $X$
and $\soc(X)$ denotes the socle of $X$. 
\ee
\end{convention}

\subsection{Crystals} \

Let $I $ be an index set. A quintuple $ (\cmA,\wlP,\Pi,\wlP^\vee,\Pi^\vee) $ is called a  (symmetrizable) {\it Cartan datum} if it
consists of
\begin{enumerate}
	\item[(a)] a generalized \emph{Cartan matrix} $\cmA=(a_{ij})_{i,j\in I}$, 
	\item[(b)] a free abelian group $\wlP$, called the {\em weight lattice},
	\item[(c)] $\Pi = \{ \alpha_i \mid i\in I \} \subset \wlP$,
	called the set of {\em simple roots},
	\item[(d)] $\wlP^{\vee}=
	\Hom_{\Z}( \wlP, \Z )$, called the \emph{coweight lattice},
	\item[(e)] $\Pi^{\vee} =\{ h_i \in \wlP^\vee \mid i\in I\}$, called the set of {\em simple coroots}, 
\end{enumerate}
which satisfy the following properties: \bnum
\item $\lan h_i, \alpha_j \ran = a_{ij}$ for $i,j \in I$,
\item $\Pi$ is linearly independent over $\Q$,
\item for each $i\in I$, there exists $\Lambda_i \in \wlP$, called a \emph{fundamental weight}, such that $\lan h_j,\Lambda_i \ran =\delta_{j,i}$ for all $j \in I$.
\item there is a symmetric bilinear 
form $( \cdot \, , \cdot )$ on $\wlP$ satisfying 
$(\al_i,\al_i)\in\Q_{>0}$ and 
$ \lan h_i,  \lambda\ran = {2 (\alpha_i,\lambda)}/{(\alpha_i,\alpha_i)}$. 
\end{enumerate}

We set $ \rlQ \seteq \bigoplus_{i \in I} \Z \alpha_i$ and  $\rlQ^+ \seteq  \sum_{i\in I} \Z_{\ge 0} \alpha_i  \subset\rlQ $ and define $\height{\beta}=\sum_{i \in I} k_i $  for $\beta = \sum_{i \in I} k_i \alpha_i \in \rlQ^+$.
We write $\prD$ for the set of  positive roots associated with $\cmA$ and set $\nrD \seteq  - \prD$.
Denote by $\weyl$  the \emph{Weyl group}, which is the subgroup of $\mathrm{Aut}(\wlP)$ generated by  
$ s_i(\lambda) \seteq  \lambda - \langle h_i, \lambda \rangle \alpha_i  $ for $i\in I$. 
We denote  by $U_q(\g)$ the \emph{quantum group} associated with $(\cmA, \wlP,\wlP^\vee, \Pi, \Pi^{\vee})$, which
is a $\Q(q)$-algebra generated by $f_i$, $e_i$ $(i\in I)$ and $q^h$ $(h\in \wlP^\vee)$ with certain defining relations (see \cite[Chapter 3]{HK02} for details).
We denote by $U_q^-(\g)$ the subalgebra of $U_q(\g)$ generated by $f_i$ ($i\in I$). Let us recall the notion of a crystal.
 We refer the reader to \cite{K91, K93, K94} and \cite[Chapter 4]{HK02} 
for more details. 
\begin{definition}
	A \emph{crystal} is a set $B$ endowed with  maps $\wt\col  B \rightarrow \wlP$,  $\varphi_i$,  $\ep_i \col  B \rightarrow \Z \,\sqcup \{ \infty \}$
	and $\te_i$, $\tf_i \col  B \rightarrow B\,\sqcup\{0\}$ 
 for all $i\in I$ which satisfy the following axioms: 
	\bna
	\item $\varphi_i(b) = \ep_i(b) + \langle h_i, \wt(b) \rangle$,
	\item $\wt(\te_i b) = \wt(b) + \alpha_i$ if $\te_i b\in B$, and 
$\wt(\tf_i b) = \wt(b) - \alpha_i$ if $\tf_i b\in B$, 
	\item for $b,b' \in B$ and $i\in I$, $b' = \te_i b$ if and only if $b = \tf_i b'$,
	\item for $b \in B$, if $\varphi_i(b) = -\infty$, then $\te_i b = \tf_i b = 0$,
	\item if $b\in B$ and $\te_i b \in B$, then $\ep_i(\te_i b) = \ep_i(b) - 1$ and $ \varphi_i(\te_i b) = \varphi_i(b) + 1$,
	\item if $b\in B$ and $\tf_i b \in B$, then $\ep_i(\tf_i b) = \ep_i(b) + 1$ and $ \varphi_i(\tf_i b) = \varphi_i(b) - 1$.
\ee
\end{definition}
We denote by $B_\g(\infty)$ the crystal of $U_q^-(\g)$ and let $\hv$ be the highest vector of $ B_\g(\infty)$. 
We simply write $B(\infty)$ instead of $B_\g(\infty)$ if no confusion arises.
The $\Q(q)$-antiautomorphism $*$ of $U_q(\g)$ defined by
\begin{align} \label{Eq: anti *}
	(e_i)^* = e_i, \qquad (f_i)^* = f_i, \qquad  (q^h)^*  = q^{-h}, 
\end{align}
gives an involution on $ B(\infty)$. This provides another crystal structure with $ \te_i^*$, $ \tf_i^*$, $\ep_i^*$, $\varphi_i^*$  on $U_q^-(\g)$, which is denoted by $B(\infty)^*$.
We set 
\begin{align} \label{Eq: height of b}
\Ht(b) \seteq  \Ht( - \wt(b)) \qt{for $b \in B(\infty)$}.
\end{align}

\subsection{Quantum affine algebras} \label{subsec:quantum affine}

We assume that $\cmA=(a_{i,j})_{i,j\in I}$
is an affine Cartan matrix. Note that the rank of $\wlP$ is $|I|+1$.
We choose a  $\Q$-valued non-degenerate  symmetric bilinear form $( \ , \ )$ on $\wlP$ satisfying 
$$
\lan h_i,\la \ran= \dfrac{2(\al_i,\la)}{(\al_i,\al_i)} \quad \text{ and} \quad \lan c,\la \ran = (\delta,\la)
$$
for any $i \in I$ and $\la \in \wlP$, where $\delta$ is the \emph{imaginary root} in  $\wlP$ and $c $ is the  \emph{central element} in $\wlP^\vee$. 
We take $\rho \in \wlP$ (resp.\ $\rho^\vee \in \wlP^\vee$) such that $\lan h_i,\rho \ran=1$ (resp.\ $\lan \rho^\vee,\al_i\ran =1$) for any $i \in I$.
Let $\g$ be the \emph{affine Kac-Moody algebra} associated with $\cmA$ and set $ I_0 \seteq I \setminus \{ 0 \}$.
Here we refer the reader to \cite[Section 2.3]{KKOP20A} 
for a choice of $0\in I$. 
We denote by $\g_0$ the subalgebra of $\g$ generated by $e_i$, $f_i$ ($i\in I_0$).

Let $q$ be an indeterminate and $\ko$ the algebraic closure of the subfield $\C(q)$
in the algebraically closed field $\corh\seteq\bigcup_{m >0}\C((q^{1/m}))$. 
Let $U_q'(\g)$ be the $\cor$-subalgebra of the quantum group $U_q(\g)$ generated by $e_i,f_i,  K_i\seteq  q_i^{\pm h_i}$ $(i \in I)$, where $q_i \seteq  q^{(\alpha_i,\alpha_i)/2}$.
Let $\catC_\g$ be the category of finite-dimensional integrable
$\uqpg$-modules, 
i.e., finite-dimensional modules $M$ with a weight decomposition
$$
M=\soplus_{\la\in\wlPc}M_\la  \qquad \text{ where } M_\la=\st{u\in M\mid K_iu=q_i^{\langle h_i,\la \rangle}},
$$
where $\wlPc  \seteq  \wlP / (\wlP \cap \Q \delta)$. 
The tensor product $\otimes$ gives a monoidal category structure on $\catC_\g$.  
We set $M^{\otimes k} \seteq \underbrace{M \tens \cdots \tens M}_{\ake[.5ex] k\text{-times}}$
for $k\in\Z_{\ge0}$. 
For $M,N\in \catC_\g$, we denote by $M \htens N$ the head of $M\tens N$ and by $M\stens N$ the socle of $M \tens N$. 
We say that 
$M$ and $N$ \emph{commute} if $M  \tens  N\simeq N\tens M$. 
We say that $M$ and $N$ \emph{strongly commute} if $M  \tens  N$ is simple. 
A simple $U_q'(\g)$-module $L$ is \emph{real} if $L \tens L $ is simple.
For $M\in \catC_\g$, we denote by  $\dual M$ and $\dual^{-1} M$ the right and the left dual of $M$, respectively. 
We extend this to $\dual^k (M)$ for all $k \in \Z$.

A simple module $L$ in $\catC_\g$ contains a non-zero vector $u \in L$ of weight $\lambda\in \wlPc$ such that (i) $\langle h_i,\lambda \rangle \ge 0$ for all $i \in I_0$,
(ii) all the weight of $L$ are contained in $\lambda - \sum_{i \in I_0} \Z_{\ge 0} \clp (\alpha_i)$, where $\clp \colon \wl\to \wlPc$ is the canonical projection.
Such a $\la$ is unique and $u$ is unique up to a constant multiple. We call $\lambda$ the {\it dominant extremal weight} of $L$ and $u$ a {\it dominant extremal weight vector} of $L$.
For each $i \in I_0$, we set 
$$
\varpi_i \seteq {\mathrm{gcd}}(\mathsf{c}_0,\mathsf{c}_i)^{-1}\clp (\mathsf{c}_0\Lambda_i-\mathsf{c}_i \Lambda_0)\in\wlPc,
$$
where the central element $c$ is equal to $ \sum_{ i\in I} \mathsf{c}_i h_i$.
For any $i\in I_0$, we denote by $V(\varpi_i)$ the \emph{$i$-th fundamental representation}. 
Note that the dominant extremal weight of $V(\varpi_i)$ is $\varpi_i$.

For a module $M\in \catC_\g$, we denote by $M^\aff$ the {\it affinization} of $M$ and 
by $z_M \colon M^\aff \to M^\aff$ the $U_q'(\g)$-module automorphism of weight $\delta$. Note that $M^\aff \simeq \cor[z^{\pm 1}]\otimes_\Z M $ 
with the action
$$
e_i(a\tens v)=z^{\delta_{i,0}}a\tens e_iv\quad \text{for $a\in\cor[z^{\pm1}]$ and $v\in M$.}
$$ 
We sometimes write $M_z$ instead of $M^\aff$ to emphasize the endomorphism $z$. 
For $x \in \ko^\times$, we define
$$M_x \seteq M^\aff / (z_M -x)M^\aff.$$
We call $x$ a {\it spectral parameter} (see \cite[Section 4.2]{Kas02} for details).

For $i\in I_0$, let $m_i$ be a positive integer such that
$$
\weyl\pi_i\cap\bigl(\pi_i+\Z\delta\bigr)=\pi_i+\Z m_i\delta,
$$
where $\pi_i$ is an element of $\wl$ such that $\clp(\pi_i)=\varpi_i$.
Then,
$V(\varpi_i)_x \simeq V(\varpi_i)_y  $ if and only if $x^{m_i}=y^{m_i}$ for $x,y\in \ko^\times$ (see \cite[Section 1.3]{AK97}).
We define
\begin{align*}
\sig \seteq I_0 \times \cor^\times / \sim, 
\end{align*}
where the equivalence relation $\sim$ is given by 
\begin{align*} 
	(i,x) \sim (j,y) \Longleftrightarrow
	V(\varpi_i)_x \simeq V(\varpi_j)_y
	\Longleftrightarrow\text{$i=j$ and $x^{m_i}=y^{m_j}$.}
\end{align*}
We denote by $[(i,a)]$ the equivalence class of $(i,a)$ in $ \sig$.  When no confusion arises, we simply write $(i,a)$ for the equivalence class $[(i,a)]$.
For $(i,x)$ and $ (j,y) \in \sig$, we put $d$ many arrows  from $(i,x)$ to $(j,y)$, where $d$ is the order of zeros of 
the denominator $d_{ V(\varpi_i), V(\varpi_j) } \allowbreak ( z_{V(\varpi_j)} / z_{V(\varpi_i)}  )$
at $  z_{V(\varpi_j)} / z_{V(\varpi_i)}  = y/x$ (see \S\,\ref{sec:rmat}). 
Since $\sig$ has a quiver structure,
we choose a connected component $\sigZ$ of $\sig$. Since 
a connected component of $\sig$ is unique up to a spectral parameter shift, 
$\sigZ$ is uniquely determined up to a quiver isomorphism.
The Hernandez-Leclerc category $\catCO$ is the smallest full subcategory of $\catC_\g$ such that 
\bna
\item $\catCO$ contains $V(\varpi_i)_x$ for all $(i,x) \in \sigZ$,
\item $\catCO$ is stable by taking subquotients, extensions and tensor products.
\ee

\subsection{R-matrices and related invariants} \label{sec:rmat}

In this subsection, we briefly recall the new invariants $\La$, $\Li$ and $\de$ introduced in \cite{KKOP19C} and several results of \cite{  KKOP19C, KKOP20A}.
For the notion of (\emph{universal\/}) \emph{R-matrices} and related properties, we refer the reader to \cite{D86}, \cite[Appendices A and B]{AK97}, \cite[Section 8]{Kas02} and \cite{H19}.

For  non-zero modules $M, N \in \catC_\g$, we denote by $\Runiv_{M,N_z}
\col \cor((z))\tens_{\cor[z^{\pm1}]}(M\tens N_z)\to
\cor((z))\tens_{\cor[z^{\pm1}]}(N_z\tens M)$ the \emph{universal R-matrix}. 
We say that $\Runiv_{M,N_z}$ is \emph{rationally renormalizable}
if there exists $f(z) \in \ko((z))^\times$ 
such that
$$
f(z) \Runiv_{M,N_z}\bl M\tens N_z\br\subset N_z\tens M. 
$$
If $\Runiv_{M,N_z}$ is rationally renormalizable, then we can choose
$c_{M,N}(z) \in \ko((z))^\times$ such that
$\Rren_{M,N_z} \seteq c_{M,N}(z)\Runiv_{M,N_z}$ sends 
$M \otimes N_z$ to $N_z \otimes M$ 
and its specialization 
$$  \Rren_{M,N_z}\big\vert_{z=x} \colon M \otimes N_x  \to N_x \otimes M$$
does not vanish at any $z=x\in\cor^\times$. 
 Note that $\Rren_{M,N_z}$ and $c_{M,N}(z)$ are unique up to a multiple of $\cz^\times = \bigsqcup_{\,n \in \Z}\ko^\times z^n$. 
We call $\Rren_{M,N_z}$ {\em the renormalized R-matrix} and 
$c_{M,N}(z)$  the \emph{renormalizing coefficient}.
We denote by $\rmat{M,N}$ the specialization at $z=1$ 
\eq
\rmat{M,N} \seteq \Rren_{M,N_z}\vert_{z=1} \colon M \otimes N \to N \otimes M,
\label{eq:rmat}
\eneq
and call it the \emph{R-matrix}. The  R-matrix $\rmat{M,N}$ is well-defined up to a constant multiple whenever $\Runiv_{M,N_z}$ is \rr. By the definition, $\rmat{M,N}$ never vanishes.

Let $M$ and $N$ be simple modules in $\catC_\g$ and let $u$ and $v$ be dominant extremal weight vectors of $M$ and $N$,
respectively. 
Then there exists $a_{M,N}(z)\in\ko[[z]]^\times$
such that
$$
\Runiv_{M,N_z}\big( u \tens v_z\big)= a_{M,N}(z)\big( v_z\tens u \big).
$$
Thus we have  a unique $\ko(z)\tens\uqpg$-module isomorphism
\begin{align*}
	\Rnorm_{M,N_z}\seteq a_{M,N}(z)^{-1} &  \Runiv_{M,N_z}\big\vert_{\;\ko(z)\otimes_{\ko[z^{\pm1}]} ( M \otimes N_z) } 
\end{align*}
from $\ko(z)\otimes_{\ko[z^{\pm1}]} \big( M \otimes N_z\big)$ to $\ko(z)\otimes_{\ko[z^{\pm1}]}  \big( N_z \otimes M \big)$, which
satisfies
$	\Rnorm_{M, N_z}\big( u  \otimes v_z\big) = v_z\otimes u .$
We call $a_{M,N}(z)$ the {\it universal coefficient} of $M$ and $N$, and $\Rnorm_{M,N_z}$ the {\em normalized $R$-matrix}.

Let $d_{M,N}(z) \in \ko[z]$ be a monic polynomial of the smallest degree such that the image of $d_{M,N}(z)
\Rnorm_{M,N_z}(M\tens N_z)$ is contained in $N_z \otimes M$, which is called the {\em denominator of $\Rnorm_{M,N_z}$}. 
Then we have
$	\Rren_{M,N_z}  =  d_{M,N}(z)\Rnorm_{M,N_z} \col M \otimes N_z \To N_z \otimes M$ up to a multiple of $\cz^\times$.
Thus
$	\Rren_{M,N_z} =a_{M,N}(z)^{-1}d_{M,N}(z)\Runiv_{M,N_z} $ and $ c_{M,N}(z)= \dfrac{d_{M,N}(z)}{a_{M,N}(z)}$
up to a multiple of $\ko[z^{\pm1}]^\times$.
In particular, $\Runiv_{M,N_z}$ is \rr whenever $M$ and $N$ are simple. 

In the following theorem, we refer \cite{Kas02} for the notion of good modules.
Note that any good module is real simple and every fundamental module $V(\varpi_i)$ is a good module. 
 
\begin{theorem}[{\cite{AK97,Chari02,Kas02,KKKO15}}]  \label{Thm: basic properties}
	\hfill
	\bnum
	\item \label{it:comm} For simple modules $M$ and $N$ such that one of them is real, $M_x$ and $N_y$ strongly commute to each other if and only if $d_{M,N}(z)d_{N,M}(1/z)$ does not vanish at $z=y/x$.
	\item For good modules $M$ and $N$, the zeroes of $d_{M,N}(z)$ belong to
	$\C[[q^{1/m}]]q^{1/m}$ for some $m\in\Z_{>0}$.
	\item  Let $M_k$ be a good module
	with a dominant extremal vector $u_k$ of weight $\lambda_k$, and
	$a_k\in\ko^\times$ for $k=1,\ldots, t$.
	Assume that $a_j/a_i$ is not a zero of $d_{M_i, M_j}(z) $ for any
	$1\le i<j\le t$. Then the following statements hold.
	\bna
	\item  $(M_1)_{a_1}\otimes\cdots\otimes (M_t)_{a_t}$ is generated by $u_1\otimes\cdots \otimes u_t$.
	\item The head of $(M_1)_{a_1}\otimes\cdots\otimes (M_t)_{a_t}$ is simple.
	\item Any non-zero submodule of $(M_t)_{a_t}\otimes\cdots\otimes (M_1)_{a_1}$ contains the vector $u_t\otimes\cdots\otimes u_1$.
	\item The socle of $(M_t)_{a_t}\otimes\cdots\otimes (M_1)_{a_1}$ is simple.
	\item  Let $\rmat{}\col (M_1)_{a_1}\otimes\cdots\otimes (M_t)_{a_t} \to (M_t)_{a_t}\otimes\cdots\otimes (M_1)_{a_1}$  be 
	$\rmat{ (M_1)_{a_1},\ldots, (M_t)_{a_t} }\seteq\prod\limits_{1\le j<k\le t}
	\rmat{(M_j)_{a_j},\,(M_k)_{a_k}}$. 
	Then the image of $\rmat{}$ is simple and it coincides with the head of $(M_1)_{a_1}\otimes\cdots\otimes (M_t)_{a_t}$
	and also with the socle of $(M_t)_{a_t}\otimes\cdots\otimes (M_1)_{a_1}$.
\end{enumerate}
\item\label{Thm: bp5}
For any simple module $M\in\catC_\g$, there exists
a finite sequence $\st{(i_k,a_k)}_{1\le k\le t}$ 
in  $\sig$
such that
$M$ 
has $\sum_{k=1}^t \vp_{i_k}$ as a dominant extremal weight
and it is isomorphic to a simple subquotient of
$V(\vp_{i_1})_{a_1}\tens\cdots V(\vp_{i_t})_{a_t}$.
Moreover, such a sequence $\st{(i_k,a_k)}_{1\le k\le t}$
is unique up to a permutation.

We call
$\sum_{k=1}^t(i_k,a_k)\in  \awlP^+\seteq\Z_{\ge0}^{\oplus  \sig}$, the {\em affine highest weight}
of $M$. 
\end{enumerate}
\end{theorem}

We define
\begin{align*} 
	\varphi(z) \seteq \prod_{s=0}^\infty (1-\tp^{s}z)
	=\sum_{n=0}^\infty\hs{.3ex}\dfrac{(-1)^n\tp^{n(n-1)/2}}{\prod_{k=1}^n(1-\tp^k)}\;z^n
	\in\ko[[z]],
\end{align*}
where $p^* \seteq (-1)^{ \langle \rho^\vee, \delta \rangle} q^{\langle c, \rho \rangle}$  and  
$\tp \seteq (p^*)^2 = q^{2 \langle c, \rho \rangle}$.
We consider the subgroup  
\begin{align*}
	\G \seteq \left\{ cz^m \prod_{a \in \ko^\times} \varphi(az)^{\eta_a} \ \left|  \
	\begin{matrix} \ c \in \ko^\times, \ m \in \Z , \\
		\eta_a \in \Z \text{ vanishes except finitely many $a$'s. } \end{matrix} \right. \right\} \subset \cor((z))^\times.
\end{align*}
For a subset $S$ of $\Z$, let $\tp^{S} \seteq \{ \tp^k\mid k \in S\} $, and define  group homomorphisms
$\Deg \col    \G \to  \Z $ and $ \Di \col    \G \to  \Z $ by
$$
\Deg\bl f(z)\br = \sum_{a \in \tp^{\,\Z_{\le 0}} }\eta_a - 
\sum_{a \in \tp^{\,\Z_{> 0}} } \eta_a \quad \text{ and } \quad \Di\bl f(z)\br = \sum_{a \in \tp^{\,\Z}} \eta_a
$$
for $f(z)=cz^m \prod_{  a\in\cor^\times } \varphi(az)^{\eta_a} \in \G$.

For  non-zero modules $M$ and $N$ in $\catC_\g$ such that $\Runiv_{M,N_z}$ is \rr, 
we define 
\begin{align*}
	\Lambda(M,N) & \seteq\Deg\bl c_{M,N}(z)\br,\\
	\Lambda^\infty(M,N) & \seteq\Deg^\infty\bl c_{M,N}(z)\br, \\
	\de(M,N) &\seteq  \frac{1}{2} \bl\La(M,N) + \La(N,M)\br.
\end{align*}
Note that $\La(M,N)\equiv \Li(M,N)\bmod 2$. 

Let  $M$ and  $N$ be simple modules in $\catC_\g$.  By 
\cite[Proposition 3.22, Proposition 3.18]{KKOP19C} (see also \cite[Proposition 2.16]{KKOP20A}), we have 
\begin{equation}\label{Eq: La d}
\begin{aligned}
\Lambda(M,N) &=		\sum_{k \in \Z} (-1)^{k+\delta(k<0)} \de(M,\D^{k}N) = \sum_{k \in \Z} (-1)^{k+\delta(k>0)} \de(\D^{k}M,N), \\
\Lambda^\infty(M,N) &= \sum_{k \in \Z} (-1)^{k} \de(M,\D^{k}N),\\
\La(M,N)&=\La(N,\dual M).
\end{aligned}
\end{equation}
If $L$ is a real simple module, then 
by \cite[Lemma 4.3 and Corollary 4.4]{KKOP19C} we have
\begin{itemize}
\item $\La(M\hconv N,L)=\La(M,L)+\La(N,L)$ and $\La(L,N\hconv M)=\La(L,N)+\La(L,M) $ if $L$ and $N$ strongly commute, 
\item  $\La(L, M\hconv N)=\La(L,M)+\La(L,N)$ if $\dual L$ and $ N$ strongly commute,
\item $\La(M\hconv N,L)=\La(M,L)+\La(N,L)$ if $\dual^{-1} L$ and $M$ strongly commute.
\end{itemize}

\begin{lemma} [{\cite[Corollary 3.13]{KKKO15}}] \label{Lem: MNDM}
	Let $L$ be a real simple module.
	Then for any simple module $X$, we have
\begin{align*}
	(L \htens X) \htens \dual L \simeq X,& \qquad \dual^{-1} L \htens (X \htens L)  \simeq X, \\
	L \htens (X \htens \dual L) \simeq X,& \qquad (\dual^{-1} L \htens X) \htens L  \simeq X.
\end{align*}
\end{lemma}

\Lemma[{\cite[Lemma 2.28, Lemma 2.29]{KKOP20A}}]   \label{lem:simplylinked}
Let $M$ and $N$ be real simple modules such that $\de(M,N)=1$.
Then, we have
\bnum
\item  $M \htens N$ is a real simple module, and it commutes with $M$ and $N$,
\item for any $m,n\in \Z_{\ge0}$, we have 
$$M^{\tens m}\hconv N^{\tens n}
\simeq
\bc (M\hconv N)^{\tens m}\tens N^{\tens(n-m)}&\text{if $m\le n$,}\\
M^{\tens (m-n)}\tens(M\hconv N)^{\tens n} &\text{if $m\ge n$.}
\ec
$$
In particular, $M^{\tens m}\hconv N^{\tens n}$ is real simple. 
\ee
\enlemma

We now recall the notion of \emph{normal sequences}. Let $L_1, L_2, \ldots, L_r$ be simple modules.
The sequence $(L_1, \ldots, L_r)$ is called a \emph{normal sequence} if the composition of the R-matrices
\begin{align*}
	\rmat{L_1, \ldots, L_r} \seteq&  \prod_{1 \le i < j \le r} \rmat{L_i, L_j} \\
	= & (\rmat{L_{r-1, L_r}}) \circ \cdots \circ ( \rmat{L_2, L_r}  \circ \cdots \circ \rmat{L_2, L_3} )  \circ ( \rmat{L_1, L_r}  \circ \cdots \circ \rmat{L_1, L_2} ) \\
	&\hs{3ex} \cl L_1 \otimes L_2 \otimes  \cdots \otimes L_r \longrightarrow L_r \otimes \cdots \otimes L_2 \otimes L_1
\end{align*}
does not vanishes. 

\Lemma[{\cite[Lemma 2.19]{KKOP20A}, \cite[Lemma 4.15]{KKOP19C}}]
\label{lem:normal}
If $(L_1, \ldots, L_r)$ is normal and they are real except for at most one, then $\Im (\rmat{L_1,\ldots, L_r})$ is simple
and it coincides with the head of $L_1\tens\cdots \tens L_r$ 
and also with the socle of $L_r\tens\cdots \tens L_1$.
\enlemma

Thus we have the following lemma. 

\begin{lemma} [\protect{\cite[Lemma 2.23]{KKOP20A}}]  \label{Lem: dual head}
Let $(L_1, L_2, \ldots, L_r)$ be a normal sequence of
simple modules such that they are real except for at most one. 
Then, for any $m\in \Z$, we have 
	$$
	\dual^{m}(  \hd ( L_1 \tens L_2 \tens  \cdots \tens L_r ) ) \simeq \hd ( \dual^{m} L_1 \tens \dual^{m} L_2 \tens  \cdots \tens \dual^{m} L_r ).
	$$ 
\end{lemma}

\begin{lemma}  [\protect{\cite[Lemma 2.21]{KKOP20A}, \cite[Lemma 4.3 and Lemma 4.17]{KKOP19C}}]  \label{Lem: normal for 3}
Let $L,M,N$ be three simple modules such that they are real except for at most one. 
If one of the following conditions
\bna
\item $\de(L,M)=0$ and $L$ is real, 
\item $\de(M,N)=0$ and $N$ is real,
\item $ \de(L, \dual^{-1}N) = \de(\dual L, N)=0$  and $L$ or $N$ is real 
\ee
holds,  then $(L,M,N)$ is a normal sequence, i.e., 
\begin{align*}
	\La(L,M\htens N) = \La(L,M ) + \La(L,N), \quad \La(L\htens M, N) = \La(L,N ) + \La(M,N).
\end{align*}
\end{lemma}

\begin{lemma}  [\protect{\cite[Lemma 2.24]{KKOP20A}}]  \label{Lem: normal for triple}
	Let $L, M,N$ be simple modules.
	Assume that $L$ is real and one of $M$ and $N$ is real. 
	Then $\de(L, M\hconv N) = \de(L, M) + \de(L,N)$ if and only if
	$(L,M,N)$ and $(M,N,L) $ are normal sequences.  
\end{lemma}

\smallskip

\Def Let $(M,N)$ be an ordered pair of simple modules in $\catC_\g$. We call it \emph{unmixed} if 
$
\de(\dual M, N)=0,
$
and \emph{strongly unmixed} if 
$$
\de(\dual^k M, N)=0 \qquad \text{ for any } k\in \Z_{\ge 1}.
$$
\edf
By \cite[Lemma 5.2]{KKOP20A},  we have
\begin{align} \label{Eq: Li=La}
\Li(M,N) = \La(M,N) \qt{as soon as the pair $(M,N)$ is strongly unmixed.}
\end{align}
 Let $ L_1, L_{2}, \ldots, L_r $  be simple modules. We say that the sequence  $(L_1, L_{2}, \ldots, L_r )$  is unmixed {\rm (}resp.\ strongly unmixed{\rm)} if 
$ ( L_a, L_b)$ is unmixed {\rm (}resp.\ strongly unmixed{\rm)} for any pair $( a,b) $ such that $1\le a<b\le r$. 
 If a sequence $( L_1, \ldots, L_r)$ of real simple modules is unmixed, then 
$(L_1, \ldots, L_r)$ is normal (\cite[Lemma 5.3]{KKOP20A}).

\vskip 2em

\section{PBW theory for $\catCO$} \label{Sec: PBW} 

In this section, we review the PBW theory for $\catCO$ developed in \cite{KKOP20A}. 

A module $L\in \catC_\g$ is called a \emph{root module} if $L$ is a real simple module such that
\eq\label{eq:root}
&\phantom{aaaa}&\de\bl L, \dual^kL\br\ms{3mu} = \ms{4mu}
\delta(k=\pm 1)
\qtext{for any $k\in\Z$.}
\eneq

Let $J$ be an index set and let $\cmC = (c_{i,j})_{i,j\in J}$ be a simply-laced finite Cartan matrix.
We denote by $\gc$ the simple Lie algebra associated with $\cmC$.
A family $\ddD \seteq \{ \Rt_i \}_{i\in J} $
 of simple modules in $\catC_\g$
is called  a \emph{duality datum} associated with $\cmC$  if it satisfies the following conditions: 
\bna
\item $\Rt_i$ is a real simple module for any $i\in J$, 
\item $\de(\Rt_i, \Rt_j) = -c_{i,j}$ for any $i,j\in J$ such that $i\ne j$.
\ee
The duality datum $\ddD $ is {\em strong}  if it satisfies the following conditions: 
\bna
\item $\Rt_i$ is a  root module for any $i\in J$, 
\item $\de(\Rt_i, \dual^k(\Rt_j)) = - \delta(k= 0)\, c_{i,j}$
for any $k\in \Z$ and $i,j\in J$ with $i\not=j$. 
\ee

Let  $\ddD \seteq  \{ \Rt_i \}_{i\in J} \subset \catCO$  be a duality datum associated with $\cmC$. 
One can construct a \emph{quantum affine Schur-Weyl duality functor} (shortly a \emph{duality functor})
$$
\F_\ddD \col \RC\gmod \longrightarrow \catC_\g
$$
using the duality datum $\ddD$, where $\RC$ is the \emph{symmetric quiver Hecke algebra} associated with $\cmC$ (see \cite{KKK18A, KKOP20A}).
Note that $ \F_\ddD (M \conv N) \simeq \F_\ddD( M) \tens \F_\ddD( N)$ and 
$\F_\ddD(L(i)) = \Rt_i$ for $i\in J$,
where $\conv$ denotes the convolution product in $\RC\gmod$ and $L(i)$ is the 1-dimensional simple $ \RC(\al_i)$-module.

\begin{theorem} [{\cite[Section 4.4]{KKOP20A}}]  \label{Thm: sdd}
	Let $ \ddD= \{ \Rt_i \}_{i\in J}$ be a strong duality datum associated with $\cmC$.
	\bni
\item The duality functor $\F_\ddD$ is exact. 
	\item The duality functor $\F_\ddD$ sends simple modules to simple modules.
	\item  The duality functor $\F_\ddD$ induces an injective ring homomorphism 
	$$\KRc\monoto K(\catC_\g  ),$$
	where $ \KRc$ is the specialization of the 
Grothendieck ring  $K_{q}( \RC\gmod)$ at $q=1$.
	\item 
	For any simple modules $M$, $N$ in $ \RC\gmod$, we have 
	\bna
	\item $\La(M,N) = \La\bl \F_\ddD(M), \F_\ddD(N)\br$,
	\item $\de(M,N) = \de\bl \F_\ddD(M), \F_\ddD(N) \br$,
	\item $(\wt M,\wt N) =-\Li\bl \F_\ddD(M), \F_\ddD(N)\br$,\label{it:wt}
\item $\tL(M,N)=\de\bl \dual \F_\ddD(M), \F_\ddD(N)\br$, 
	\item $\de\bl \dual^k\F_\ddD(M), \F_\ddD(N)\br=0$ for any $k\not=0,\pm1$.
	\ee
	\ee
\end{theorem}

The category $\catCD$ is defined to be  the smallest full subcategory of $\catCO$ such that
\bna
\item   it contains $\F_\ddD( L )$ for any simple  $R_{\cmC}$-module $L$, 
\item  it is stable by taking subquotients, extensions, and tensor products. 
\ee

It was proved in \cite{LV09} that the set of the 
isomorphism classes of simple modules in $\RC\gmod$ has a crystal structure, which is isomorphic to the crystal $B_{\gc}(\infty)$ of $U_q^-(\gc)$.
The corresponding crystal operators $\tf_i$ and $\tfs_i$ are given by 
\begin{align} \label{Eq: def of tf in QHA}
	\tf_i (M) =  \hd\bl L(i) \conv M\br \quad \text{ and } \quad \tfs_i (M)= \hd\bl M \conv L(i)\br 
\end{align} 
for a simple $\RC$-module $M$ and $i\in J$.
For $b\in B_{\gc}(\infty)$, we denote by $L(b)$ the corresponding self-dual simple module in $\RC\gmod$.

\begin{lemma} \label{Lem: B and catCD}
Let $ \ddD= \{ \Rt_i \}_{i\in J}$ be a strong duality datum associated with $\cmC$, and 
let $ B_\ddD$ be the set of the isomorphism classes of simple modules in $\catCD$.  
\bni
\item The duality functor $\F_\ddD$ gives a bijection $	\cL_\ddD\col  B_{\g_\cmC}(\infty) \buildrel \sim \over \longrightarrow B_\ddD $
sending $b $ to \allowbreak $ [\F_\ddD(L(b))]$ for $b\in B_{\gc}(\infty)$.

\item For any $i\in J$ and $b \in B_{\g_\cmC}(\infty)$, we have 
\bna
\item 
	$\cL_\ddD \bl \tf_i(b) \br \simeq \Rt_i \htens \cL_\ddD ( b ) $ and $  \cL_\ddD \bl \tfs_i(b) \br \simeq  \cL_\ddD ( b ) \htens \Rt_i$, 
\item 
$\cL_\ddD \bl\te_i(b) \br \simeq  \cL_\ddD ( b )\htens \dual \Rt_i  $  if $\te_i(b)\ne 0$, 
\item 
$\cL_\ddD \bl \tes_i(b) \br \simeq  \dual^{-1}\Rt_i  \htens \cL_\ddD ( b ) $ if $\tes_i(b)\ne 0$,

\item 	
$\ep_i (b) = \de\bl\dual\, \Rt_i, \cL_\ddD(b)\br $ and $ \eps_i (b) = \de\bl\dual^{-1} \Rt_i, \cL_\ddD(b) \br.$ 
\ee
\ee
\end{lemma}
\begin{proof}
(i) follows from  Theorem \ref{Thm: sdd}.

\snoi
(ii) Since $\F_\ddD$ is exact and monoidal, (a) follows  
from \eqref{Eq: def of tf in QHA} by applying 
the functor $\F_\ddD$. 
We have (b), (c) by Lemma \ref{Lem: MNDM} and 
(d) by {\cite[Corollary 4.13]{KKOP20A}}.	
\end{proof}

It is easy to see that,
for any duality datum $\ddD= \{ \Rt_i \}_{i\in J}$ and $k\in\Z$, the family
$\dual^k\ddD\seteq\st{\dual^k\Rt_i}_{i\in J}$ is also a duality datum.
If $\ddD$ is strong, then so is $\dual^k\ddD$.

\Lemma\label{lem:dual}
Let $k\in \Z$.
\bnum
\item The right dual functor $\dual$ induces a ring automorphism of 
$K(\catC_\g)$.
\item
The following diagram commutes:
$$\xymatrix{
\KRc\ar[r]^-{\F_\ddD}\ar[rd]_-{\F_{\dual^k\ddD}}& K(\catC_\g  )\ar[d]^{\dual^k}\\
&K(\catC_\g  ).
}$$
In particular, we have $\cL_{\dual^k\ddD}(b)=\dual^k\bl\cL_{\ddD}(b)\br$
for any $b\in B(\infty)$.
\ee
\enlemma
\Proof
(i) follows from the fact that $\dual$ induces an anti-automorphism of 
$K(\catC_\g  )$ which is a commutative ring.
(ii) immediately follows from (i)
and Lemma~\ref{Lem: dual head}.
\QED

\begin{definition} \label{Def: complete}
	A duality datum $\ddD$ is called \emph{complete} if it is strong and, for any simple module $M \in \catCO$, there exists simple modules $M_k \in \catCD$ $(k\in \Z)$ such that 
	\bna
	\item $M_k \simeq \one $ for all but finitely many $k$,
	\item $M \simeq \hd ( \cdots \tens \dual^2 M_2 \tens \dual M_1 \tens \ M_0 \tens \dual^{-1} M_{-1} \tens \cdots  ).$
	\ee
\end{definition}
When $\ddD$ is complete, the simple Lie algebra $\gc$ is of the type $X_\g$ given in \eqref{Table: root system} (\cite[Proposition 6.2]{KKOP20A}).
In this case, we write $\gf$, $\If$, etc.\ instead of $\gc$, $J$, etc. 
\renewcommand{\arraystretch}{1.5}
\begin{align} \label{Table: root system} \small
	\begin{array}{|c||c|c|c|c|c|c|c|} 
		\hline
		\text{Type of $\g$} & A_n^{(1)}  & B_n^{(1)} & C_n^{(1)} & D_n^{(1)} & A_{2n}^{(2)} & A_{2n-1}^{(2)} & D_{n+1}^{(2)}  \\
		&(n\ge1)&(n\ge2)&(n\ge3)&(n\ge4)&(n\ge1)&(n\ge2)&(n\ge3)\\
		\hline
		\text{Type $X_\g$} & A_n & A_{2n-1}    & D_{n+1}   &  D_n & A_{2n} & A_{2n-1} & D_{n+1}  \\
		\hline
		\hline
		\text{Type of $\g$} & E_6^{(1)}  & E_7^{(1)} & E_8^{(1)} & F_4^{(1)} & G_{2}^{(1)} & E_{6}^{(2)} & D_{4}^{(3)}  \\
		\hline
		\text{Type $X_\g$} & E_6 & E_{7}    & E_{8}   & E_6 & D_{4} & E_{6} & D_{4}  \\
		\hline
	\end{array}
\end{align}

\smallskip 

From now on, we assume that $\ddD = \{ \Rt_i \}_{i\in \If}$ is a complete duality datum of $\catCO$.  

\medskip
Let $\prDf$ be the set of positive roots of $\gf$ and let $\weylf $ be the Weyl group associated with $\gf$. 
Let $w_0$ be the longest element of $\weylf$, and $\ell$ denotes the length of $w_0$.
We choose an arbitrary reduced expression $\rxw = s_{i_1} s_{i_2} \cdots s_{i_\ell}$ of $w_0$.
We extend $\st{i_k}_{1\le k\le \ell}$ to $\st{i_k}_{k\in\Z}$ by 
$i_{k+\ell}=(i_k)^*$ for any $k\in\Z$,
where $i^*$ is a unique element of $\If$ such that $\al_{i^*}=-w_0\al_i$ for $i\in \If$.
Let 
$$
\{ \Cp_k \}_{k=1, \ldots, \ell} \subset R_{\gf}\gmod
$$ 
be the \emph{cuspidal modules} associated with the reduced expression $\rxw$.
Note that $\Cp_k$ corresponds to the dual PBW vector corresponding to $\beta_k\seteq  s_{i_1} \cdots s_{i_{k-1}} (\alpha_{i_k}) \in \prDf $ for $k=1, \ldots, \ell$ under the categorification.
We define a sequence of simple modules $\{ \cuspS_k \}_{ k\in \Z }$ in $\catC_\g $ as follows:
\bna
\item $\cuspS_k = \F_\ddD(\Cp_k)$ for any $k=1, \ldots, \ell$, 
and we extend its definition to all $k\in\Z$ by 
\item $\cuspS_{k+\ell} = \dual( \cuspS_k )$ for any $k\in \Z$.
\ee
The modules $\cuspS_k$ are called the \emph{affine cuspidal modules} corresponding to $\ddD$ and $\rxw$.

\begin{prop} [{\cite[Proposition 5.7]{KKOP20A}}]  \label{Prop: cusp} 
	The affine cuspidal modules satisfy the following properties.
	\bnum
	\item $\cuspS_a$ is a root module for any $a\in\Z$.
	\item
	For any $a,b \in \Z $ with $a > b$, the pair $( \cuspS_a, \cuspS_b )$ is strongly unmixed.
	\item Let $ k_1 > \cdots > k_t $ be decreasing integers and $ (a_1, \ldots, a_t) \in \Z_{ \ge0 }^t  $. Then 
	\bna
	\item the sequence  
	$ ( \cuspS_{k_1}^{\tens a_1}, \ldots, \cuspS_{k_t}^{\tens a_t} ) $ is normal,
	\item the head of the tensor product $ \cuspS_{k_1}^{\tens a_1} \tens \cdots \tens \cuspS_{k_t}^{\tens a_t} $  is simple.
	\ee
	
	\ee
\end{prop}

We define 
$
	\ZZ \seteq \Z_{\ge0}^{\oplus \Z}
	= \bigl\{ (a_k)_{k\in \Z} \in   \Z_{\ge0}^{\Z} \mid
	\text{$a_k=0$ except finitely many $k$'s} 
	\bigr\}.
$
We denote by $\prec$ the bi-lexicographic order on $\ZZ$, i.e.,   
for any $ \bfa = (a_k)_{k\in \Z}$ and $\bfa' = (a_k')_{k\in \Z}$  in $\ZZ$, $ \bfa \prec \bfa' $ if and only if the following conditions hold:
\begin{align*}
&\phantom{aa}&\left\{\parbox{68ex}{\bna
	\item there exists $r \in \Z$ such that $ a_k = a_k' $ for any $k < r$ and $ a_r < a_r'$, 
	\item   there exists $s \in \Z$ such that $ a_k = a_k' $ for any $k > s$ and $ a_s < a_s'$.
	\ee
}\right.
\end{align*}

For $ \bfa = (a_k)_{k\in \Z} \in \ZZ$, we define
\eqn
&&\sP_{\ddD, \rxw} (\bfa)\seteq
\bigotimes_{k =+\infty}^{-\infty}  \cuspS_k^{\otimes a_k}
=\cdots\tens\cuspS_2^{\otimes a_2}\tens\cuspS_1^{\otimes a_1} \otimes  \cuspS_{0}^{\otimes a_{0}}
\otimes \cuspS_{-1}^{\otimes a_{-1}}
\otimes\cuspS_{-2}^{\otimes a_{-2}}\otimes \cdots.
\eneqn
Here, $\sP_{\ddD, \rxw} (0)$ should be understood as 
the trivial module $\one$.
We call the modules $\sP_{\ddD, \rxw} (\bfa)$ \emph{standard modules} with respect to the cuspidal modules $\{ \cuspS_k \}_{k\in \Z}$.

\begin{theorem} [{\cite[Theorem 6.10]{KKOP20A}}] \label{Thm: PBW1} \
	
	\bnum
	\item For any $\bfa \in \ZZ$, the head of \/ $\sP_{\ddD, \rxw} (\bfa)$ is simple. We denote the head by 
	$
	\sV _{\ddD, \rxw} (\bfa) \seteq  \hd \bl\sP_{\ddD, \rxw} (\bfa)\br.
	$
	
	\item For any simple module $M \in \catCO$, there exists a unique $\bfa \in \ZZ$ such that 
	$
	M \simeq  \sV _{\ddD, \rxw} (\bfa).
	$
	\ee
	Therefore, the set $\{   \sV _{\ddD, \rxw} (\bfa) \mid \bfa \in \ZZ \}$ is a complete and irredundant set of simple modules of $\catCO$ up to isomorphisms.
\end{theorem}

The element $\bfa \in \ZZ$ associated with a simple module $M$
in Theorem~\ref{Thm: PBW1}~(ii) is called the \emph{cuspidal decomposition} of $M$ with respect to the cuspidal modules $\{ \cuspS_k \}_{k\in \Z}$,
and it is denoted by $\bfa_{\ddD, \rxw}(M) $.

\begin{theorem}[{\cite[Theorem 6.12]{KKOP20A}}] \
	Let $\bfa$ be an element of $ \ZZ$. 
	\bnum
	\item The simple module $ \sV _{\ddD, \rxw} (\bfa) $ appears only once in $\sP_{\ddD, \rxw} (\bfa)$.
	\item If $V$ is a simple subquotient of $\sP_{\ddD, \rxw} (\bfa)$ which is not isomorphic to $ \sV _{\ddD, \rxw} (\bfa) $, then we have 
	$
	\bfa_{\ddD, \rxw}(V) \prec \bfa.
	$
	\item In the Grothendieck ring, we have 
	$$
	[\sP_{\ddD, \rxw} (\bfa)] = [\sV_{\ddD, \rxw} (\bfa)] + \sum_{\bfa' \prec \bfa} c(\bfa') [\sV_{\ddD, \rxw} (\bfa')]
	$$
	for some $ c(\bfa') \in \Z_{\ge 0}$.
	\ee
\end{theorem}

\vskip 2em

\section{Extended crystals} \label{Sec: extended crystal}

In this section, we introduce the notion of the extended crystal of $B(\infty)$.
Let $I$ be an index set and let $\g$ be the Kac-Moody algebra associated  
with a symmetrizable generalized Cartan matrix $\cmA = (a_{i,j})_{i,j\in I}$.  
We set 
\begin{align*}
\cBg \seteq  \Bigl\{  (b_k)_{k\in \Z } \in \prod_{k\in \Z} B(\infty) \bigm| b_k =\hv \text{ for all but finitely many $k$} \Bigr\} ,
\end{align*}
where $B(\infty)$ is the crystal of the negative half $U_q^-(\g)$.
We set 
$
\one \seteq  (\hv)_{k\in \Z} \in \cBg.
$

For any integer $k\in \Z$, let $\pi_k \col  \cBg \rightarrow B(\infty)$ be  the $k$-th projection 
defined by $ \pi_k(\cb) = b_k$ for any $  \cb =   (b_k)_{k\in \Z} \in \cBg$. We denote by $\iota_k \col  B(\infty) \rightarrow \cBg$ the section of $\pi_k$
which is defined by 
$$
 \pi_{k'}\circ \iota_{k} (b) = 
\begin{cases}
	b & \text{ if } k = k', \\
	\hv & \text{ if } k \ne k'
\end{cases} 
$$
for any $k' \in \Z$ and $b\in B(\infty)$.
We set 
$$
\cI \seteq  I \times \Z.
$$ 
Let  $(i,k)\in \cI$. The maps 
$\wt_k \col  	\cBg \longrightarrow \wlP$ and   $\ep_{(i,k)}$, $\eps_{(i,k)} \col   \cBg \longrightarrow \Z$
 are defined by
\begin{align*}
	\wt_k(\cb) \seteq  (-1)^k \wt(b_k),   \qquad \ep_{(i,k)}(\cb) \seteq   \ep_i(b_k),   \qquad   \eps_{(i,k)}(\cb) \seteq   \eps_i(b_k)
\end{align*}
for any $\cb = (b_k)_{k\in \Z} \in \cBg $, and we define 
\begin{align*}
	\hwt(\cb)  \seteq  \sum_{k \in \Z}  \wt_k(\cb),   \qquad \hep_{(i,k)}(\cb) \seteq   \ep_{(i, k)}(\cb) - \eps_{(i,k+1)}(\cb).
\end{align*}

We now define the \emph{extended crystal operators}
\begin{align*}
\tF_{(i,k)} \col  \cBg \longrightarrow \cBg\qtq	\tE_{(i,k)} \col  \cBg \longrightarrow \cBg,
\end{align*}
by
\begin{equation} \label{Eq: tE and tF}
\begin{aligned}
	\tF_{(i,k)}(\cb) & \seteq  
	\begin{cases}
		(\cdots , b_{k+2},  \  b_{k+1} , \ \tf_i( b_k), \ b_{k-1}, \cdots ) & \text{ if } \hep_{( i,k)} (\cb) \ge 0,\\
		(\cdots , b_{k+2},  \ \tes_i (b_{k+1} ), \ b_k, \ b_{k-1}, \cdots ) & \text{ if }  \hep_{( i,k)} (\cb) < 0 ,
	\end{cases}
	\\
	\tE_{(i,k)}(\cb) & \seteq  
	\begin{cases}
		(\cdots , b_{k+2},  \  b_{k+1} , \ \te_i( b_k), \ b_{k-1}, \cdots ) & \text{ if } \hep_{(i,k)} (\cb) >  0,\\
		(\cdots , b_{k+2},  \ \tfs_i (b_{k+1} ), \ b_k, \ b_{k-1}, \cdots ) & \text{ if }  \hep_{(i,k)} (\cb) \le  0,
	\end{cases}
\end{aligned}
\end{equation}
for any $(i, k) \in  \cI$ and $\cb =   (b_k)_{k\in \Z} \in \cBg$.
Note that $\tE_{(i,k)}(\cb) $ is non-zero for any $\cb \in \cBg$.
When no confusion arises, we simply write $\tF_{i,k}$, $\tE_{i,k}$, $\ep_{i,k}$, etc.\  for $\tF_{(i,k)}$, $\tE_{(i,k)}$ $\ep_{(i,k)}$, etc.

For any $\cb = (b_k)_{k\in \Z} \in \cBg$, we define $\cs(\cb) \seteq  (b_k')_{k \in \Z} $ by 
$$
b_k' = *(b_{-k})\qquad \text{ for any } k\in \Z,
$$
where $*$ is the anti-involution defined in \eqref{Eq: anti *}. Then it gives an involution
\begin{align*}
	\cs \col  \cBg \longrightarrow  \cBg
\end{align*}	
such that $\cs(\one)  = \one$ and $\hwt(\cb) = \hwt( \cs(\cb))$ for $\cb  \in \cBg$.

\begin{lemma} \label{Lem: star}
For $\cb \in \cBg$ and $(i,k) \in \cI$, we define 
\begin{align} \label{Eq: def of EFstar}
	\heps_{i,k}(\cb) \seteq  \hep_{i,- k}(\cs(\cb)), \quad \tFs_{i,k}  ( \cb) \seteq  \cs( \tF_{i, -k}( \cs(\cb) ) ), \quad \tEs_{i,k}  ( \cb) \seteq  \cs( \tE_{i, -k}( \cs(\cb) ) ) .
\end{align}	
Then we have the following.
\bni
\item $ \heps_{i,k}(\cb) = - \hep_{i,k-1}(\cb)  $.
\item $\tFs_{i,k}  ( \cb) = \tE_{i,k-1}  ( \cb) $.
\item $\tEs_{i,k}  ( \cb) = \tF_{i,k-1}  ( \cb) $.
\ee
\end{lemma}	
\begin{proof}
Let $\cb = (b_k)_{k\in \Z} \in \cBg$ and $(i,k) \in \cI$.
Since $ \eps_i(b) = \ep_i(* b)$ for $b\in B(\infty)$, we have 
$$
\ep_{i,- k}(\cs(\cb)) = \ep_{i} ( *b_k) = \eps_i(b_k) = \eps_{i,k}(\cb) 
$$
which implies that 
$$
\heps_{i,k}(\cb) =  \eps_{ i, k }(\cb) - \ep_{ i,k-1 }(\cb) =  - \hep_{i,k-1}(\cb).
$$
Thus we have (i).
Using the definitions of $\tFs_{(i,k)}$ and $\tEs_{(i,k)}$, one can obtain the following:
\begin{equation} \label{Eq: tEs and tFs}
	\begin{aligned}
		\tFs_{(i,k)}(\cb) & =  
		\begin{cases}
			(\cdots , \  \  b_{k+1} , \ \tfs_i( b_k), \ b_{k-1}, \ b_{k-2}, \cdots ) & \text{ if } \heps_{(i,k)} (\cb) \ge 0,\\
			(\cdots , \  \ b_{k+1}, \ b_k, \ \te_i ( b_{k-1}), \ b_{k-2}, \cdots ) & \text{ if }  \heps_{(i,k)} (\cb) < 0 ,
		\end{cases}
		\\
		\tEs_{(i,k)}(\cb) & =  
		\begin{cases}
			(\cdots ,  \  b_{k+1} , \ \tes_i( b_k), \ b_{k-1}, \ b_{k-2}, \cdots ) & \text{ if } \heps_{(i,k)} (\cb) >  0,\\
			(\cdots , \  b_{k+1}, \ b_k, \  \tf_i (b_{k-1}), \ b_{k-2}, \cdots ) & \text{ if }  \heps_{(i,k)} (\cb) \le  0.
		\end{cases}
	\end{aligned}
\end{equation}
Therefore, (ii) and (iii) follows from (i) and \eqref{Eq: tE and tF}. 
\end{proof}

\smallskip

The lemma below follows from the definitions.
\begin{lemma} \label{Lem: basic for cBg}
	For any $ \cb  \in \cBg$ and $(i,k) \in \cI$, we have the following.
	\bni
	\item  $ \hwt ( \tF_{i,k}  (\cb) ) =  \hwt ( \cb ) + (-1)^{k+1}\al_i $ and $ \hwt ( \tE_{i,k}  (\cb) ) =  \hwt ( \cb ) + (-1)^{k}\al_i $.
	\item $ \hep_{i,k} ( \tF_{i,k}  (\cb) ) =  \hep_{i,k} (\cb ) + 1 $ and $ \hep_{i,k} ( \tE_{i,k}  (\cb) ) =  \hep_{i,k} (\cb ) - 1 $.
	\item $ \heps_{i,k} ( \tFs_{i,k}  (\cb) ) =  \heps_{i,k} (\cb ) + 1 $ and $ \heps_{i,k} ( \tEs_{i,k}  (\cb) ) =  \heps_{i,k} (\cb ) - 1 $.
	
\item 
 $\tF_{ i,k }$ and $\tE_{ i,k }$ are inverse to each other. 
	\item 
$\tFs_{ i,k }$ and $\tEs_{ i,k }$ are inverse to each other. 
	\item For $b \in B(\infty)$, we have 
\begin{align*}
	\tF_{i,k} (  \iota_k(  b ) ) &=  \iota_k ( \tf_{i} ( b)), \\   \tE_{i,k} (  \iota_k(  b ) ) &=  \iota_k ( \te_{i} ( b)) \quad \text{if $ \te_i(b) \ne 0$,} \\
	\tFs_{i,k} (  \iota_k(  b ) ) &=  \iota_k ( \tfs_{i} ( b)), \\   \tEs_{i,k} (  \iota_k(  b ) ) &=  \iota_k ( \tes_{i} ( b)) \quad \text{if $ \tes_i(b) \ne 0$.}
\end{align*}	
	\ee 
\end{lemma}

\begin{lemma} \label{Lem: cD}
	Let $t\in \Z$. 
	For $\cb = (b_k)_{k\in \Z} \in \cBg$, we define $\cd^t(\cb) \in \cBg$ by
\eq
	\pi_k( \cd^t(\cb) )  = b_{k-t} \qquad \text{ for any $k\in \Z$.}
\label{def:dual}
\eneq
	Then it gives a bijection 
	\begin{align*}
		\cd^t \col  \cBg \longrightarrow  \cBg
	\end{align*}
	such that 
	$$
	\cd^t ( \tF_{i,k} (\cb) ) = \tF_{i,k+t} (   \cd^t  (\cb) ) \quad \text{ and } \quad \cd^t ( \tFs_{i,k} (\cb) ) = \tFs_{i,k+t} (   \cd^t  (\cb) )
	$$
	for any $ (i,k) \in \cI $.
\end{lemma}
\begin{proof}
	It follows from \eqref{Eq: tE and tF}, \eqref{Eq: tEs and tFs} and  the definition of $ \cd^t$.
\end{proof}

As a usual crystal, the set $\cBg $ has the $\cI $-colored graph structure induced by the operators $\tF_{i,k}$ for $(i,k) \in \cI$.
We take $\cBg $ as the set of vertices and define the $\cI$-colored arrows on $\cBg$ by
$$
\cb \To[{(i,k)}]\cb' \quad  \text{ if and only if} \quad  \cb' = \tF_{i,k} \cb=  \tE^*_{i, k+1 }\cb  \qquad (\;(i,k)\in \cI).
$$
We call $\cBg $ the \emph{extended crystal} of $B(\infty)$.

\Lemma \label{Lem: connectedness}
As an $\cI$-colored graph, $\cBg $ is connected.
\enlemma
\begin{proof}
For $\cb = (b_k)_{k\in \Z} \in \cBg$,
set $\Ht(\cb) \seteq  \sum_{k\in\Z} \Ht(b_k)$,
where $\Ht(b)$ is defined in \eqref{Eq: height of b}.
When $\cb \ne \one$, let
$
l(\cb) \seteq  \max\{ k\in \Z \mid b_k \ne \hv \}$.

Let $\cb = (b_k)_{k\in \Z} \in \cBg$. We shall prove that $\cb$ is connected to $\one$ in the graph $ \cBg$ by induction on $\Ht(\cb)$.

If $\Ht ( \cb) = 0$, then $\cb = \one$. Thus it is trivial that $\cb$ is connected to $\one$.

Suppose that $\Ht(\cb) \ne 0$. Let $l = l(\cb)$. Since $b_l \ne \hv$, there exists $i\in I$ such that $\te_i ( b_l)\ne 0$.
Since $\hep_{i,l} (\cb) > 0$, we have 
$$
\tE_{i,l} (\cb) = ( \cdots, \hv, \te_{i}(b_l), b_{l-1}, b_{l-2} \cdots ),
$$
which says that $ \Ht( \tE_{i,l} (\cb) ) = \Ht (\cb)-1 $. Since $\tE_{i,l} (\cb)$ is connected to $\one$ by the induction hypothesis, $\cb$ is also connected to $\one$.  
\end{proof}

\begin{example} \label{Ex: sl2}
Let $I= \{ 1\}$ and let $B(\infty)$ be the crystal of $U_q^-(\mathfrak{sl}_2)$. 
We identify $B(\infty)$ with $\Z_{\ge0}$ and simply write $\tF_k$ instead of $\tF_{1,k}$ for $k\in \Z$. 
Then the extended crystal
$\cB(\infty) $ is equal to $ (\Z_{\ge 0})^{\oplus \Z}$,
and for any $k\in \Z$ and $ \cb =   (b_k)_{k\in \Z} \in \cB(\infty) $, we have 
\begin{align*}
	\tF_{k}(\cb) & = 
	\begin{cases}
		(\cdots , b_{k+2},  \ b_{k+1} -1, \ b_k, \ b_{k-1}, \cdots ) & \text{ if }  b_{k+1}   >  b_{k}, \\
		(\cdots , b_{k+2},  \  b_{k+1} , \  b_k+1 , \ b_{k-1}, \cdots ) & \text{ if } b_{k+1}  \le b_{k}.
	\end{cases}
\end{align*}
\end{example}

\vskip 2em

\section{Categorical crystals for quantum affine algebras} \label{Sec: categorical str}

Let $U_q'(\g)$ be a quantum affine algebra of arbitrary type.
In this section, we will prove that the set of the isomorphism classes of simple modules in $\catCO$ has an extended crystal structure isomorphic to
 $\cBg[\gf]$.

\subsection{Root modules}\ 

In this subsection, we shall prove several lemmas for categorical crystals.

\begin{lemma} \label{Lem: LDLaLb}
	Let $L$ be a root module. For $a,b \in \Z_{\ge 0}$, we have 
	\begin{align*}
		L \htens \left( (\dual L)^{ \tens a} \htens L^{ \tens b} \right) \simeq 
		\begin{cases}
			(\dual L)^{ \tens (a-1)} \htens L^{ \tens b} & \text{ if } a  > b ,\\
			(\dual L)^{\tens a} \htens L^{ \tens ( b+1)} & \text{ if } a  \le b,
		\end{cases}\\[.3ex]
		\left( (\dual L)^{ \tens a} \htens L^{ \tens b} \right)\tens \dual L \simeq 
		\begin{cases}
			(\dual L)^{ \tens (a+1)} \htens L^{ \tens b} & \text{ if } a  \ge b ,\\
			(\dual L)^{\tens a} \htens L^{ \tens ( b-1)} & \text{ if } a  < b.
		\end{cases}
	\end{align*} 
\end{lemma}
\begin{proof}
Since the second isomorphism can be proved similarly, 
we show only the first isomorphism.
	By Lemma \ref{lem:simplylinked}, we have 
	\begin{align}
	\de(\dual L, \dual L \htens L) &=0, \label{Eq: dL dLL}\\
	\de(L, \dual L \htens L)&=0, \label{Eq: L dLL}
	\end{align}
	which implies that, by Lemma~\ref{lem:normal} and Lemma~\ref{Lem: normal for 3},  
	\begin{align} \label{Eq: LdL normal}
	\parbox[t]{72ex}{$
\hd\bl ( L)^{ \tens x}\tens ( \dual L )^{\tens y}\tens( \dual L \htens L)^{\tens z}\br$ and $\hd\bl (L)^{\tens x}\tens ( \dual L \htens L)^{\tens y}\tens ( L)^{\tens z}\br$ 
are simple modules
for any $x,y,z \in \Z_{\ge0}$.}
	\end{align}

\noi 
(a)\ We shall first treat the case $a > b$. 	
It follows from Lemma \ref{lem:simplylinked}, \eqref{Eq: dL dLL} and \eqref{Eq: LdL normal} that
	\begin{align*}
	L \htens \bl (\dual L)^{ \tens a} \htens L^{ \tens b} \br  
	& \simeq L \htens \bl (\dual L)^{\tens (a-b)} \tens  (\dual L \htens L)^{ \tens b}  \br \\ 
			& \simeq \bl L \htens (\dual L)^{\tens ( a-b)}\br \htens  (\dual L \htens L)^{ \tens b}.
\end{align*}
Since 
$ 
L \htens (\dual L)^{\tens (a-b)} \simeq ( L \htens \dual L )\htens (\dual L)^{\tens ( a-b-1)} \simeq (\dual L)^{\tens ( a-b-1)},  
$
Lemma \ref{lem:simplylinked} and \eqref{Eq: dL dLL} tell us that
\begin{align*}
	L \htens \bl (\dual L)^{\tens a} \htens L^{\tens b} \br 
	& \simeq  (\dual L)^{\tens ( a-b-1)} \tens  (\dual L \htens L)^{\tens b} \\
	& \simeq (\dual L)^{\tens ( a-1)} \htens L^{\tens b}.
\end{align*}

\noi
(b)\ Next we shall treat the case $a \le b$. 	
By  Lemma \ref{lem:simplylinked}, \eqref{Eq: L dLL} and \eqref{Eq: LdL normal}, 
we have
	\begin{align*}
		L \htens \bl (\dual L)^{\tens a} \htens L^{\tens b} \br  & \simeq L \htens \bl   (\dual L \htens L)^{\tens a} \tens L^{\tens ( b-a)}   \br \\ 
		& \simeq  L \tens    (\dual L \htens L)^{\tens a} \tens L^{\tens ( b-a)}         \\ 
		& \simeq(\dual L \htens L)^{ \tens a} \tens L^{\tens ( b-a+1)}     \\ 
		& \simeq (\dual L)^{\tens a} \htens L^{\tens ( b+1)}. \qedhere
	\end{align*}
\end{proof}

\begin{lemma} \label{Lem: sh LMN}
	
	Let $L,M,N$ be simple modules in $\catC_\g$. Suppose that 
	\bna
	\item $L$ is real,
	\item $\de(L,M)=0$,
	\item $M \tens N$ has a simple head.
	\ee
	Then we have 
	$$
	L \htens (M \htens N) \simeq M \htens (L \htens N).
	$$
\end{lemma}
\begin{proof}
	By the assumptions (a), (b)  and \cite[Lemma 4.3]{KKOP19C}, the diagram 
	$$
	\xymatrix{
		L \tens M \tens N \ar@{->>}[d] \ar[rr]^{\rmat{L, M\tens N}} &&M \tens N \tens L  \ar@{->>}[d] \\
		L \tens( M \htens N) \ar[rr]^{\rmat{L, M\htens N}} && (M \htens N) \tens L
	}
	$$
	commutes. Then it follows from the assumptions (a), (c) and  \cite[Proposition 4.5 (ii)]{KKOP19C} that $L \tens M \tens N$ has a simple head.
	Thus we have 
	$$
	L \htens (M \htens N) \simeq \hd (L \tens M \tens N)  \simeq \hd (M \tens L \tens N) \simeq  M \htens (L \htens N).  
	$$
\end{proof}

\begin{lemma} \label{Lem: MdLLN}
Let $L$, $M$ and $N$ be simple modules in $\catC_\g$. Suppose that 
\bna
\item $L$ is a root module,
\item $\de(L,M)=\de(\dual L,N)=0$,
\item \label{it:cst}$\hd ( M \tens (\dual L)^{\tens s} \tens L^{ \tens t} \tens N )$ is simple
for any $s,t\in\Z_{\ge0}$.
\ee
Then, for any $a,b \in \Z_{\ge0}$, we have  
\begin{equation*} 
	\begin{aligned}
		L \htens \hd \bl M \tens (\dual L)^{\tens a} \tens L^{ \tens b} \tens N \br 
		\simeq  
		\begin{cases}
			\hd \bl M \tens (\dual L)^{ \tens ( a-1)} \tens L^{ \tens b } \tens N\br & \text{ if } a  > b ,\\
			\hd \bl M \tens (\dual L)^{ \tens a } \tens L^{ \tens ( b+1)} \tens N\br & \text{ if } a  \le b.
		\end{cases}
	\end{aligned}
\end{equation*}
\end{lemma}
\begin{proof}
We set $N' \seteq  \hd( (\dual L)^{ \tens a} \tens L^{\tens  b} \tens N  )$. 
Since we have
\begin{itemize}	
\item $N'$ is simple by (c),
\item $ \de(   M, L ) = 0 $, 
\end{itemize}	
we obtain 
	\begin{equation} \label{Eq: MDLLN1}
		\begin{aligned}
			L \htens \hd\bl M \tens (\dual L)^{\tens a} \tens L^{ \tens b} \tens N\br & \simeq L \htens (M \htens N') \\
			&\simeq M \htens (L \htens N')\quad\qt{by Lemma \ref{Lem: sh LMN}}\\
			& \simeq M \htens\biggl( L \htens \Bigl(\bl (\dual L)^{ \tens a} \htens L^{ \tens b}   \br \htens N  \Bigr) \biggr).
		\end{aligned}
	\end{equation}
	On the other hand, since
	\begin{itemize}
		\item $(\dual L)^{\tens a} \htens L^{ \tens b}$ is real by Lemma \ref{lem:simplylinked},
		\item $\de( \dual L, N )=0$, 
	\end{itemize}
we have 
\begin{equation}\label{Eq: MDLLN2}
\ba{ll}
		L \htens  \Bigl(\bl (\dual L)^{\tens a} \htens L^{\tens b}  \br \htens N\Bigr)  & \simeq \Bigl( L \htens \bl (\dual L)^{\tens a} \htens L^{\tens b}\br \Bigr) \htens N 
\\[.5ex]
		&\simeq \begin{cases}
			\bl (\dual L)^{\tens ( a-1)} \htens L^{\tens b} \br  \htens N & \text{ if } a  > b ,\\
			\bl (\dual L)^{\tens a} \htens L^{\tens ( b+1)} \br  \htens N & \text{ if } a  \le b.
		\end{cases}
\ea
\end{equation}
Here, the first isomorphism follows from
Lemma \ref{Lem: normal for 3} and the second from Lemma \ref{Lem: LDLaLb}.
Therefore the assertion follows from $\eqref{Eq: MDLLN1}$, $\eqref{Eq: MDLLN2}$  and \ref{it:cst}.	
\end{proof}

\begin{lemma} \label{Lem: M'M''}
Let $L$ be a root module and let $M$ be a simple module. We set $m \seteq  \de( \dual L, M )$ and $n \seteq  \de( \dual^{-1} L, M)$.
Then, we have 
	\bni
	\item $\de(\dual L, M' )=0$ and $M \simeq L^{\tens m} \htens M'$
for some simple module $M'$,
	\item $\de(\dual^{-1} L, M'' )=0$ and $M \simeq  M'' \htens L^{ \tens n}$
for some simple module $M''$.
	\ee
\end{lemma}
\begin{proof}
Define 
\begin{align*}
	M' \seteq  M \htens ( \D L)^{\tens m} \quad \text{and} \quad M'' \seteq  ( \D^{-1}L)^{\tens n} \htens M.
\end{align*}
By Lemma \ref{Lem: MNDM}, we have 
$$
M \simeq L^{\tens m} \htens M'  \simeq  M'' \htens L^{ \tens n}.
$$
Moreover, by \cite[Lemma 3.4]{KKOP20A}, we have 
\begin{align*}
\de(\dual L, M' ) &= \de( \D L, M \htens (\D L)^{\tens m}  ) =  \de(\D L, M   )-m = 0, \\
\de(\dual^{-1} L, M'' ) &= \de(\D^{-1} L, (\D^{-1} L)^{\tens n} \htens  M   ) =  \de( L'', M   )-n = 0
\end{align*}
by a standard induction argument.
 \end{proof}

\subsection{Categorical crystals} \

Throughout this subsection, let $\ddD \seteq  \{ \Rt_i \}_{i\in \If} \subset \catCO$ be a complete duality datum of $\catCO$.

Let $\cBg[\gf]$ be the extended crystal of $U_q^-(\gf)$.
For $\cb = (b_k)_{k\in \Z}  \in \cBg[\gf]$, we define 
\begin{align*}
	\ccLD(\cb) \seteq  \hd ( \cdots \tens \dual^2 L_2 \tens \dual L_1 \tens  L_0 \tens \dual^{-1} L_{-1} \tens \cdots ), 
\end{align*}
where $ L_k =    \cLD( b_{k} ) $ for $k\in \Z$. Here, $\cLD$ is given in Lemma \ref{Lem: B and catCD} (i).

Let $\sB(\g)$ be the set of the isomorphism classes of simple modules in $\catCO$.
For a simple module $M $ in $\catCO$, we denote by $[M]$ the corresponding element in $\sB(\g)$. 
When no confusion arises, we simply write $M$ instead of $[M]$.

\begin{prop} \label{Prop: bijection BB} \
\bni

\item Let $a,b\in \Z$ with $a \le b$ and  let $L_k$ be simple modules in $\catCD$ for $k=a, a+1, \ldots, b$. Then we have
\bna
\item the sequence  $( \dual^b L_b , \dual^{ b-1} L_{ b-1},   \ldots,  \dual^a L_{a}  )$ is strongly unmixed, 
\item 
$
\hd(  \dual^b L_b \tens   \dual^{ b-1} L_{ b-1}  \tens  \cdots \tens   \dual^a L_{a} )
$
is a simple module,   
\item for $m \in \Z$, 
$$ 
\dual^{m} \left( \hd(  \dual^b L_b   \tens \cdots \tens   \dual^a L_{a} ) \right)
\simeq  \hd(  \dual^{ b+m} L_b  \tens \cdots \tens   \dual^{a+m} L_{a} ).
$$
\ee

\item For any $\cb = (b_k)_{k\in \Z}   \in \cBg[\gf] $, the module $\ccLD(\cb)$ is simple.

\item We define a map $	\Phi_\ddD \col  \cBg[\gf] \longrightarrow \sB(\g)$ by 
$$ 
\Phi_\ddD ( \cb) \seteq  \ccLD(\cb) \qquad \text{ for any $\cb \in  \cBg[\gf] $.}
$$
Then the map $\Phi_\ddD $ is bijective. 
\item For any $t\in \Z$ and $ \cb \in\cBg[\gf] $, we have 
$$
\Phi_\ddD( \cd^t(\cb) ) = \dual^t ( \Phi_\ddD(\cb)),
$$
where $ \cd $ is the bijection of $\cBg[\gf]$ defined in Lemma \ref{Lem: cD}. 
\ee	
\end{prop}
\begin{proof}
Let $\rxw$ be a reduced expression of the longest element $w_0$ of $\weylf$ and set $\ell\seteq\ell(w_0)$. Let $\cuspS_k$ $(k\in \Z)$ be the  affine cuspidal modules corresponding to $\ddD$ and $\rxw$.

\snoi
(i) 
Since $L_k \in \catCD$ ($a\le k\le b$), 
the sequence 
$( \dual^b L_b , \dual^{ b-1} L_{ b-1}   \ldots,  \dual^a L_{a}  )$ is strongly unmixed
by Theorem \ref{Thm: sdd} and 
there exist $b_{1}^k, \ldots, b_{\ell}^k \in \Z_{\ge 0}$ such that 
\begin{align} \label{Eq: Lk eq1}
L_k \simeq \hd( \cuspS_{\ell}^{\tens b_\ell^k} \tens \cdots \tens \cuspS_{1}^{\tens b_1^k} ).
\end{align}
We set $b_r^q \seteq 0$ 
unless $a\le q\le b$.
For $k\in \Z$, let $a_{k} \seteq  b_{r}^q $ where $ k = q\ell + r $ for some $1 \le r \le \ell$, and set $\bfa \seteq  (a_k)_{k\in \Z}$. By Proposition \ref{Prop: cusp} and \eqref{Eq: Lk eq1}, we have
\begin{equation} \label{Eq: Lk eq2}
\begin{aligned} 
	\sV _{\ddD, \rxw} (\bfa) &= \hd (  \cdots\tens\cuspS_2^{\otimes a_2}\tens\cuspS_1^{\otimes a_1} \otimes  \cuspS_{0}^{\otimes a_{0}} \tens \cuspS_{-1}^{\otimes a_{-1}} \tens \cdots) \\
	& \simeq \hd(  \dual^b L_b \tens   \dual^{ b-1} L_{ b-1}  \tens \ldots \tens   \dual^a L_{a} ).
\end{aligned}
\end{equation}
The statement (c) follows from Lemma \ref{Lem: dual head},  
\eqref{Eq: Lk eq1} and \eqref{Eq: Lk eq2}.
Indeed, we have
\eqn
&&\dual^m\bl\hd(  \dual^b L_b \tens   \dual^{ b-1} L_{ b-1}  \tens \ldots \tens   \dual^a L_{a} )\br\\
&&\hs{8ex}\simeq
\dual^m
\hd (  \cdots \tens\cuspS_2^{\otimes a_2}\tens\cuspS_1^{\otimes a_1} \otimes  \cuspS_{0}^{\otimes a_{0}} \tens \cuspS_{-1}^{\otimes a_{-1}} \tens \cdots) \\
&&\hs{8ex}\simeq\hd (  \cdots \tens\dual^m
\cuspS_2^{\otimes a_2}\tens\dual^m\cuspS_1^{\otimes a_1} 
\otimes \dual^m \cuspS_{0}^{\otimes a_{0}} 
\tens \dual^m\cuspS_{-1}^{\otimes a_{-1}} \tens \cdots) \\
&&\hs{8ex}\simeq
\hd(  \dual^{b+m} L_b \tens   \dual^{ b-1+m} L_{ b-1}  \tens \ldots \tens   \dual^{a+m} L_{a} ).
\eneqn

\mnoi
(ii) follows from (i).

\mnoi
(iii)  Let $M$ be a simple module in $\sB(\g)$. 
By Theorem \ref{Thm: PBW1}, there exists a unique $\bfa = (a_k)_{k\in \Z} \in \ZZ$ such that 
$$
M \simeq  \sV _{\ddD, \rxw} (\bfa).
$$
For each $k \in \Z$, we set 
$$
L_k \seteq  \hd( \cuspS_\ell^{  \tens a_{ (k+1) \ell }} \tens \cdots \tens \cuspS_2^{ \tens a_{k\ell + 2}} \tens \cuspS_1^{ \tens a_{k\ell + 1}}  ).
$$
By Proposition \ref{Prop: cusp},  $L_k$ is a simple module in $\catCD$ and 
$$
M \simeq \hd ( \cdots \tens \dual^{2} L_2 \tens \dual L_1 \tens L_0 \tens \dual^{-1} L_{-1} \tens \cdots ).
$$ 
For each $k \in \Z$, 
Lemma \ref{Lem: B and catCD} says that there exists a unique $b_k \in B_{\gf}(\infty) $ such that $\cLD (b_k) = L_k$.
Thus the correspondence $ M \mapsto \cb \seteq  (b_k)_{k\in \Z}$ gives a map 
$$
 \Theta_{\ddD} \col  \sB(\g) \longrightarrow  \cBg[\gf].
 $$ 
It is obvious that $\Theta_\ddD$ is  the inverse of $\Phi_{\ddD}$.	
 
\mnoi
 (iv) 
 Let $\cb = (b_k)_{k\in \Z}  \in \cBg[\gf]$ and let $ L_k \seteq      \cLD( b_{k} ) $ for $k\in \Z$. 
 Then, we have 
 \begin{align*}
 	\dual^t ( \Phi_\ddD(\cb) ) & = \dual^t ( \hd ( \cdots \tens \dual L_1 \tens  L_0 \tens \dual^{-1} L_{-1} \tens \dual^{-2} L_{-2} \tens \cdots ) ) \\
 	& =   \hd ( \cdots \tens  \dual^{t+1} L_1 \tens  \dual^t L_0 \tens \dual^{t-1} L_{-1} \tens    \dual^{t-2} L_{-2} \tens \cdots ) ) \\
 	& = \Phi_\ddD( \cd^t( \cb)).\qh
 \end{align*}
\end{proof}	
Note that $\Phi_\ddD (\one) = \one$ by the definition.
For a simple module $M \in \catCO$, we denote by 
\begin{align*}
	\cb_\ddD(M) \seteq   \Phi_\ddD^{-1}(M) \in \cBg[\gf].	
\end{align*} 

\smallskip

Let us recall the duality functor 
$
\F_\ddD \col R_{\gf}\gmod \To \catCO.
$
Let $w_0$ be the longest element  of $\weylf$ and let $\ell \seteq  \ell(w_0)$. For $a,b \in \Z$ with $a \le b$, we set $ [a,b] \seteq  \{ a,a+1, \ldots, b\}$.

\begin{lemma} \label{Lem: many cusps}
	 Let $a,b \in \Z$ with $a \le b$. 
	For each $p \in [ a,b]$, we  choose a reduced expression $\underline{w}_{ 0, p}$ of $w_0$, 
	and let $ \{ \Cp_{ p, k} \}_{k=1}^\ell$ be the cuspidal modules in $R_{\gf}\gmod$ associated with $\underline{w}_{0,p}$. We set 
	$\cuspS_{ p, k} \seteq  \F_\ddD(\Cp_{p,k})$, and define
	$$ 
	\mathbf{S}_p (\bfc_p) \seteq  ( \dual^{p} (  \cuspS_{ p, \ell}^{\otimes c_{p,\ell}}), \ldots  , \dual^{p} ( \cuspS_{ p, 2}^{\otimes c_{p,2}}),  \dual^{p} ( \cuspS_{ p, 1}^{\otimes c_{p,1}}))
	$$
	 for $ \bfc_p = (c_{p,\ell}, \ldots, c_{p, 2}, c_{p, 1}  ) \in \Z_{\ge 0}^{\oplus \ell} $. Then their concatenation 
	$$
	\mathbf{S}_b (\bfc_b) * \mathbf{S}_{b-1} (\bfc_{b-1}) * \cdots * \mathbf{S}_a(\bfc_a) \seteq \bl \dual^{b} (  \cuspS_{ b, \ell}^{\otimes c_{b,\ell}}), 
\dual^{b} (\cuspS_{ b, \ell-1}^{\otimes c_{b,\ell-1}}),
\ldots\ldots,\dual^{a} (  \cuspS_{ a, 2}^{\otimes c_{a,2}}),  \dual^{a} (  \cuspS_{ a, 1}^{\otimes c_{a,1}})\br
	$$ is strongly unmixed and normal.
\end{lemma}
\begin{proof}
Consider the pair $\bl   \dual^{p} (  \cuspS_{ p, k}^{\otimes c_{p,k}} ),
\dual^{p'} (  \cuspS_{ p', k'}^{\otimes c_{p',k'}})\br$. 
	If $p=p'$ and $ k > k'$, 
then it is strongly unmixed by Proposition \ref{Prop: cusp}. 
If $p  >  p'$, then it is also strongly unmixed by Theorem \ref{Thm: sdd} (iv).
\end{proof}

\begin{lemma} \label{Lem: simple head for Ns}
 Let $i \in \If$,  $a,b \in \Z$ with $a \le b$.  Let $M_k$ be a simple module in $\catCD$ for $k\in [a,b]$ and $N_k$ a module in $\catC_\g$.  
For each  $k \in [a,b] $, 
we assume one of the the following conditions: 
 \bna
 \item $N_k \simeq \dual^{k} M_k$, 
 \item $ N_k \simeq \dual^{k} (  \Rt_i^{\otimes n_k }) \otimes  \dual^{k} (  M_k')$ for some $n_k\in \Z_{\ge0}$, where $M_k'$ is a simple module in $\catCD$ such that $ M_k \simeq  \Rt_i^{\otimes n_k } \htens  M_k' $ and $ \de( \dual\Rt_i, M_k')=0$,
  \item $ N_k \simeq \dual^{k} (   M_k'') \otimes \dual^{k} ( \Rt_i^{\otimes n_k })$ for some $n_k\in \Z_{\ge0}$, where $M_k''$ is a simple module in $\catCD$ such that $ M_k \simeq   M_k'' \htens \Rt_i^{\otimes n_k } $ and  $ \de( \dual^{-1}\Rt_i, M_k')=0$.
 \ee
Then, the head 
$
 \hd(   N_b \otimes \cdots \otimes  N_{a+1} \otimes  N_{a} )
 $
  is simple and it is isomorphic to $ \hd( \dual^b M_b \otimes \cdots \otimes \dual^{a+1} M_{a+1} \otimes \dual^a M_{a} )$.
\end{lemma}
\begin{proof}

We take a reduced expression $\rxw = s_{i_1}s_{i_2} \cdots s_{i_\ell}$ with $i_1 = i$ and set $ \rxw' \seteq   s_{i_2} \cdots s_{i_\ell} s_{i_1^*} $, where $\al_{i^*} = -w_0(\al_i)$.
Let $\{ \cuspS_k \}_{k\in \Z}$ and $\{ \cuspS'_k \}_{k\in \Z}$ be the affine cuspidal modules corresponding to $(\ddD,\rxw)$ and $(\ddD, \rxw')$ respectively. 
By the construction, we have $\cuspS_1 = \cuspS_\ell' = \Rt_i $, and
\begin{equation} \label{Eq: de for S}
	\begin{aligned}
		 \de (\cuspS_k, \dual^{-1} \cuspS_1) &= 0 \qquad \text{ for $ 2 \le k \le \ell$}, \\
	 \de (\cuspS_k', \dual \cuspS_\ell') &= 0 \qquad \text{ for $1 \le k \le \ell-1$}
	\end{aligned} 
\end{equation}
by Proposition \ref{Prop: cusp}. Since $ M_k \in \catCD $ for all $k$, Theorem \ref{Thm: PBW1} says that there exists non-negative integers $a_{t}^k$ such that 
$$
M_k \simeq 
\begin{cases}
	 \hd ( ( \cuspS_\ell')^{\tens a_{\ell}^k} \tens \cdots \tens ( \cuspS_1')^{\tens a_{1}^k} ) & \text{ in case (b),}\\
	 \hd (  \cuspS_\ell^{\tens a_{\ell}^k} \tens \cdots \tens \cuspS_1^{\tens a_{1}^k}  ) & \text{otherwise}.
\end{cases}
$$
Hence, by Lemma \ref{Lem: MNDM} and \eqref{Eq: de for S},  
\begin{itemize}
\item in case (b), 
  $M_k' \simeq 	\hd( ( \cuspS_{ \ell-1}')^{\tens b_{\ell-1}^k} \tens \cdots \tens ( \cuspS_1')^{\tens b_{1}^k}  )  $ and $ n_k = a_\ell^k$,
\item in case (c),  $M_k'' \simeq 	\hd( \cuspS_\ell^{\tens a_{\ell}^k} \tens \cdots \tens \cuspS_2^{\tens a_{2}^k}  )   $ and $ n_k = a_1^k$.
\end{itemize}
For $k\in [a,b]$, let 
$$
P_k \seteq  
\begin{cases}
	\dual^k (  ( \cuspS_\ell')^{\tens a_{\ell}^k} ) \tens \cdots \tens  \dual^k ( ( \cuspS_1')^{\tens a_{1}^k} ) & \text{ in case (b)},\\
	 \dual^k (  \cuspS_\ell^{\tens a_{\ell}^k}) \tens \cdots \tens \dual^k (  \cuspS_1^{\tens a_{1}^k} )  & \text{otherwise}.
\end{cases}
$$
It follows from \cite[Lemma 4.15]{KKOP19C},  \cite[Lemma 5.3]{KKOP20A} and Lemma \ref{Lem: many cusps} that  $ \hd( P_b \otimes P_{b-1} \otimes \cdots \otimes P_a) $ is simple and 
\begin{align*}
	 \hd( P_b \otimes P_{b-1} \otimes \cdots \otimes P_a) & \simeq \hd( N_b \otimes N_{b-1} \otimes \cdots \otimes N_a) \\	
	&\simeq  \hd( \dual^b M_b \otimes \cdots \otimes \dual^{a+1} M_{a+1} \otimes \dual^a M_{a} ).\qh
\end{align*} 	
\end{proof}

\smallskip

For $ M \in \sB(\g)$ and $ (i, k)\in \cIf$, we define 
\begin{equation} \label{Def: ttF}
\begin{aligned}
\ttF_{i,k} (M) &\seteq    (\dual^k \Rt_i) \htens M  \quad \text{ and }\quad \ttE_{i,k} (M) \seteq    M \htens (\dual^{k+1} \Rt_i), \\
\ttFs_{i,k} (M)&\seteq      M \htens (\dual^k \Rt_i)  \quad \text{ and }\quad \ttEs_{i,k} (M) \seteq    (\dual^{k-1} \Rt_i) \htens M.
\end{aligned}
\end{equation}

\begin{lemma} \label{Lem: basic for ttF}
	Let $(i,k)\in \cIf$. 
	\bni
	\item 
	$\ttFs_{i,k} \simeq \ttE_{i,k-1}$ and $  \ttEs_{i,k} \simeq \ttF_{i,k-1}$.
	\item 
	For simple modules $M$ and $N$,
	$\ttF_{i,k} (M)\simeq N$   if and only if  $ M \simeq \ttE_{i,k} ( N) $.
	\item 	For a simple module $M$ and $t\in \Z$, we have 
	\begin{align*}
	\dual^t \ttF_{i,k} (M) &\simeq 	 \ttF_{i,k+t} ( \dual^t M), \qquad 	\dual^t \ttE_{i,k} (M) \simeq 	 \ttE_{i,k+t} ( \dual^t M),\\ 
	\dual^t \ttFs_{i,k} (M) &\simeq 	 \ttFs_{i,k+t} ( \dual^t M), \qquad 	\dual^t \ttEs_{i,k} (M)\simeq 	 \ttEs_{i,k+t} ( \dual^t M).
	\end{align*}
	\ee	
\end{lemma}
\begin{proof}
(i) and (iii) follow from the definition \eqref{Def: ttF}.
(ii) follows from  Lemma \ref{Lem: MNDM}.
\end{proof}

Since the operators $\ttF_{i,k}$, $\ttE_{i,k}$ $\ttFs_{i,k}$, and $\ttEs_{i,k}$ depend on the choice of $\ddD$,  
we write $\sBD(\g)$ instead of $\sB(\g)$ when we consider $\sB(\g)$ together with the operators $\ttF_{i,k}$, $\ttE_{i,k}$ $\ttFs_{i,k}$, and $\ttEs_{i,k}$.   
The set $\sBD(\g)$ has the $\cIf$-colored graph structure induced from $\ttF_{i, k}$ as follows.
We take  $\sBD(\g)$ as the set of vertices and define the $\cIf$-colored arrows on $\sBD(\g)$ by 
$$ 
[M] \To [\;(i,k)\;] [M'] \quad \text{ if and only if } \quad M' \simeq \ttF_{i,k} M \qquad (\;(i,k) \in \cIf).
$$
Then we have the main theorem. 

\begin{theorem} \label{Thm: main1} 
Let  $\ddD \seteq  \{ \Rt_i \}_{i\in \If}$  be a complete duality datum of $\catCO$.
\bni
\item For $(i,k)\in \cIf $ and $\cb \in  \cBg[\gf] $, we have 
\begin{align*}
&\Phi_\ddD(  \tF_{i,k}  ( \cb) ) =  \ttF_{i,k} (  \Phi_\ddD(    \cb )), \qquad \Phi_\ddD(  \tE_{i,k}  ( \cb) ) =  \ttE_{i,k} (  \Phi_\ddD(    \cb )), \\
&\Phi_\ddD(  \tFs_{i,k}  ( \cb) ) =  \ttFs_{i,k} (  \Phi_\ddD(    \cb )), \qquad \Phi_\ddD(  \tEs_{i,k}  ( \cb) ) =  \ttEs_{i,k} (  \Phi_\ddD(    \cb )),
\end{align*}
where $\Phi_\ddD$ is given in Proposition \ref{Prop: bijection BB}.
\item The map $\Phi_\ddD$ induces an $\cIf$-colored graph isomorphism 
$$
\cBg[\gf] \simeq \sBD(\g)
$$ 
sending $\one$ to $[\one]$.
\ee
\end{theorem}
\begin{proof}
(i) Thanks to Lemma \ref{Lem: basic for ttF},  it suffices to show that $\Phi_\ddD$ commutes with $\tF_{i,k} $ for $(i,k)\in \cIf$.
By Proposition \ref{Prop: bijection BB} (iv) and Lemma \ref{Lem: basic for ttF} (iii), we may assume that $k=0$ from the beginning. 

Let $\cb = (b_k)_{k\in \Z} \in \cBg[\gf] $. We set $M \seteq  \Phi_\ddD(\cb)$ and $M_k \seteq  \cLD( b_{k} ) $ for $k\in \Z$, where $\cLD (b)$ is given in Lemma \ref{Lem: B and catCD}. 
By the construction, we have 
$$
M \simeq \hd ( \cdots \tens \dual^{2} M_2\tens  \dual M_1 \tens M_0 \tens  \dual^{-1}M_{-1} \tens \cdots  ).
$$
Let $a \seteq  \de( \dual^{-1} \Rt_i, M_1 ) $ and $b \seteq  \de( \dual \Rt_i, M_0 )$.
By Lemma \ref{Lem: M'M''}, there exist simple modules $M_1''$ and $M_0'$ such that 
\begin{itemize}
	\item $ M_1 \simeq M_1'' \htens \Rt_i^{\tens a}$ and $\de( \dual M_1'', \Rt_i ) = \de( \dual^{-1} \Rt_i, M_1'' )=0$,
	\item $M_0\simeq \Rt_i^{\tens b} \htens M_0'$ and $\de(\dual \Rt_i, M_0')=0$.
\end{itemize} 

We now set $K \seteq  \hd(\cdots \tens \dual^2 M_2 \tens  \dual M_1'')   $ and $N \seteq  \hd( M_0' \tens  \dual^{-1} M_{-1} \tens \cdots  )$.
By the construction, $(K, \dual \Rt_i, \Rt_i, N)$ is strongly unmixed, and 
Lemma \ref{Lem: simple head for Ns} implies that
\begin{align*}
	M & \simeq  \hd ( \cdots \tens \dual^{2} M_2 \tens \dual M_1 \tens M_ 0\tens \dual^{-1} M_{-1}\tens \cdots  )\\ 
	& \simeq 
	\hd ( \cdots \tens \dual^2 M_2 \tens  \dual M_1'' \tens (\dual \Rt_i)^{\tens a} \tens \Rt_i^{\tens b} \tens M_0' \tens  \dual^{-1} M_{-1} \tens \cdots   ) \\
	& \simeq \hd(K \tens (\dual \Rt_i)^{\tens a} \tens \Rt_i^{\tens b} \tens N).
\end{align*}
Since 
\begin{align*}
 \dual^{-1} (\Rt_i) \htens M_1 & \simeq  M_1'' \htens \Rt_i^{\tens ( a-1)} \quad \text{if } a > 0, \\
 \Rt_i \htens M_0 & \simeq  \Rt_i^{\tens ( b+1)} \htens  M_{0 }',
\end{align*}  
 Lemma \ref{Lem: B and catCD} and Lemma \ref{Lem: MdLLN} imply the following.
\bna
\item 
If $a > b$, then  
\begin{align*}
	\ttF_{i,0} ( M) &\simeq \Rt_i \htens \hd\bl K \tens (\dual \Rt_i)^{\tens a} \tens \Rt_i^{ \tens b} \tens N\br \\
	& \simeq  \hd\bl K \tens (\dual \Rt_i)^{\tens ( a-1)} \tens \Rt_i^{\tens b} \tens N\br \\
	& \simeq \hd \bl \cdots\tens \dual^2 M_2 \tens ( \dual M_1'' \htens (\dual \Rt_i)^{\tens (a-1)}) \tens (\Rt_i^{\tens b} \htens M_ 0') \tens \dual^{-1} M_{-1}\tens \cdots  \br\\ 
	& \simeq \hd \bl \cdots\tens \dual^2 M_2 \tens \dual (  \dual^{-1} \Rt_i \htens  M_1) \tens  M_ 0 \tens \dual^{-1} M_{-1}\tens \cdots  \br,
\end{align*}
which implies $  \Phi_\ddD ( \tF_{i,0} (\cb)) = \ttF_{i,0} (  \Phi_\ddD(  \cb ) )  $.
\item 
If $a \le b$, then  
\begin{align*}
	\ttF_{i,0} ( M) &\simeq \Rt_i \htens \hd(K \tens (\dual \Rt_i)^{ \tens a} \tens \Rt_i^{ \tens b} \tens N) \\
	& \simeq  \hd(K \tens (\dual \Rt_i)^{\tens a} \tens \Rt_i^{\tens (b+1)} \tens N) \\
	& \simeq \hd ( \cdots\tens \dual^2 M_2 \tens ( \dual M_1'' \htens (\dual \Rt_i)^{\tens a}) \tens (\Rt_i^{\tens (b+1)} \htens M_ 0') \tens \dual^{-1} M_{-1}\tens \cdots  )\\ 
	& \simeq \hd ( \cdots\tens \dual^2 M_2 \tens \dual  M_1 \tens  \Rt_i \htens M_ 0 \tens \dual^{-1} M_{-1}\tens \cdots  ),
\end{align*}
which implies $  \Phi_\ddD ( \tF_{i,0} (\cb)) = \ttF_{i,0} (  \Phi_\ddD(  \cb ) )  $.
\ee

\snoi
(ii) follows  directly from (i).
\end{proof}

The corollary below follows from Theorem \ref{Thm: main1}.

\begin{corollary} \label{Cor: main1}\
	\bni
\item 	Let $U_q'(\g)$ and $U_q'(\g')$ be quantum affine algebras, and  let $\ddD $ and $\ddD'$ be complete duality data of $\catC_{\g}^0$ and $\catC_{\g'}^0$, respectively. 
	Then  the following are equivalent.
	\bna
	\item $\sB_{\ddD}(\g) \simeq \sB_{\ddD'}(\g')$ as an $\cIf $-colored graph.
	\item $\gf  \simeq (\g')_\fin$ as a simple Lie algebra. 
	\ee
	\item In particular, the $\cIf $-colored graph structure of $ \sBD(\g)$ does not depend on the choice of  complete duality data $\ddD$.
	\ee
	
\end{corollary}

\vskip 2em

\section{Crystals and invariants } \label{Sec: crystals and invariants}
 Throughout this section, we fix a complete duality datum  $\ddD \seteq  \{ \Rt_i \}_{i\in \If}$ of $\catCO$.
We will provide formulas to compute the invariants  $\La$ and $\de$ between $ \dual^k \Rt_i$ ($k\in \Z$) and an arbitrary simple module in terms of the extended crystal.

\begin{lemma} \label{Lem: root La}
Let $L$ be a root module.
	\bni
	\item $\La(L, \dual L )=0$ and $\La( \dual L, L ) =\La(L, \dual^{-1}L) =2$.
	\item $\La(L, \dual L \hconv L )=-2$ and $\La( \dual L \hconv L , L)=2$.
	\item $\La(L,  L \hconv \dual^{-1} L )=2$ and $\La(  L \hconv \dual^{-1} L, L)=-2$.
	\item For any $a,b\in  \Z_{\ge 0}$, we have 
	\bna
	\item $\La(L,  \dual L^{\tens a} \hconv L^{\tens b} )= -2 \min\{ a,b \}$,
	\item  $\La(  \dual L^{\tens a} \hconv L^{\tens b}, L )= 2 a$,
	\item $\La( L,   L^{\tens a} \hconv  \dual^{-1} L^{\tens b} )= 2 b$,
	\item  $\La(   L^{\tens a} \hconv \dual^{-1} L^{\tens b}, L )= -2 \min\{ a,b \} $.
	\ee
	\ee
\end{lemma}
\begin{proof} 
(i)\ By \eqref{Eq: La d}, we have
$$\La(L,\dual L)=\La(L,L)=0.$$
Since $\La(L,\dual L)+\La(\dual L,L)=2\de(L, \dual L )=2$ ,
we have
$\La(\dual L,L)=2 $.

\mnoi
(ii)\ 
By Lemma~\ref{Lem: normal for 3}, we have
$$\La(\dual L\hconv L,L)=\La(\dual L, L)=2.$$
Since $\de(L, \dual L \hconv L )=0$, we obtain 
$\La(L, \dual L\hconv L)=-\La(\dual L\hconv L,L)=-2$.

\mnoi
(iii) \
We have
$\La(L,  L \hconv \dual^{-1} L )=
\La(L,  \dual^{-1} L )=\La(\dual L,L)=2$.
As $\de(L,  L \hconv \dual^{-1} L )=0$, we obtain
$\La(L \hconv \dual^{-1} L,L)=-2$.

\mnoi
(iv) follows from (i), (ii), (iii), \cite[Lemma 4.3]{KKOP19C} and Lemma \ref{lem:simplylinked}.
\end{proof}

\begin{lemma} \label{Lem: N}
 Let $i\in I$ and $b \in B_{\gf}(\infty)$. We write 
$ M \seteq  \cL_\ddD(b) \in \catCD$ and  
$$
r \seteq   \ep_i(b) \quad \text{ and } \quad  s \seteq   \eps_i(b).
$$ 
Let $x, y \in \Z_{\ge 0}$ and set
$$
N \seteq  \hd ( ( \dual \Rt_i)^{\tens x} \tens M \tens ( \dual^{-1} \Rt_i)^{\tens y}  ).
$$
Then $N$ is simple and we have 
\bni
\item $\La(\Rt_i, N) = 2 \max\{  x,r\}  +  (\al_i, \wt(b)) - 2x + 2y$,
\item $\La( N, \Rt_i) = 2 \max\{  y,s\}  +  (\al_i, \wt(b)) + 2x - 2y $,
\item $\de( \Rt_i, N ) = \max\{ x,r \} + \max\{ y,s \} +  (\al_i, \wt(b))  $.
\ee
\end{lemma}
\begin{proof}
Note that $N$ is simple because the triple $(\dual \Rt_i, M, \dual^{-1} \Rt_i  )$ is strongly unmixed by Theorem \ref{Thm: sdd}.

\snoi
	(i) 
	Let $M' \seteq   \cL_\ddD (  \te_i^{ r} (b) )$. 
	Lemma \ref{Lem: B and catCD} says that $ M \simeq \Rt_i^{\tens r} \htens M'$ and $(\Rt_i, M')$ is strongly unmixed because $ \de( \dual^k \Rt_i, M' )=0 $ for $k > 1$ and $\de( \dual \Rt_i, M' ) = \ep_i( \te_i^{ r} (b) )=0$.
	For $k,l \in \Z_{\ge 0}$, we have the following by Lemma \ref{Lem: normal for 3}:
	\begin{itemize}
		\item $\bl ( \dual \Rt_i )^{\tens k}, ( \Rt_i)^{\tens l}, M' \br$ is normal, 
		\item for any simple module $X$,
$( ( \Rt_i)^{\tens k}, X, ( \dual^{-1} \Rt_i)^{\tens l} )$ is normal
 because $(( \Rt_i)^{\tens k}, \  \allowbreak  ( \dual^{-1} \Rt_i)^{\tens l} )$ is unmixed, 
		\item  for a real simple module $Y$, $( ( \Rt_i)^{\tens k}, Y, M' )$ is normal because $( ( \Rt_i)^{\tens k}, M' )$ is unmixed, 
		\item $\La(\Rt_i, M') = \Li(\Rt_i, M') = - (-\al_i, \wt(b) + r \alpha_i)$ by \eqref{Eq: Li=La}.
	\end{itemize}
	Using the above observations with Lemme \ref{Lem: root La}, we have 
	\begin{align*}
		\La(\Rt_i, N) &= \La\bl \Rt_i, ( ( \dual \Rt_i)^{\tens x} \htens M) \htens ( \dual^{-1} \Rt_i)^{\tens y} \br \\ 
		&= \La\bl \Rt_i, ( \dual \Rt_i)^{\tens x} \htens M \br +\La\bl \Rt_i, (\dual^{-1} \Rt_i)^{\tens y} \br \\
		&= \La\bl \Rt_i,  ( ( \dual \Rt_i)^{\tens x} \htens \Rt_i^{\tens r} )  \htens M'  \br + 2y \\
		&= \La( \Rt_i,  ( \dual \Rt_i)^{\tens x} \htens \Rt_i^{\tens r}  ) + \La( \Rt_i,  M'  ) + 2y \\
		&= -2 \min\{  x,r \} - ( -\al_i, \wt(b) + r\al_i) + 2y \\ 
		&= 2x + 2r -2 \min\{  x,r\} + (\al_i, \wt(b) ) -2x + 2y \\ 
		&=  2 \max\{  x,r\}  + (\al_i, \wt(b) ) - 2x + 2y.
	\end{align*}
	
\snoi
(ii) can be proved in the same manner as above.
	
\snoi
(iii) follows from (i) and (ii).
\end{proof}

\begin{theorem} \label{Thm: La Li M} 
Let $(i,k) \in \cIf$ and let $M$ be a simple module in $\catCO$. 
We set $\cb \seteq  \cb_\ddD(M)$ and
$$
x \seteq  \eps_{(i,k+1)}(\cb), \quad r \seteq  \ep_{(i,k)}(\cb),  \quad s \seteq  \eps_{(i,k)}(\cb), \quad y \seteq  \ep_{(i,k-1)}(\cb).
$$
Then  we have 
\bni
\item $\La( \dual^k \Rt_i, M) = 2 \max\{ x, r \}  + \sum_{ t \in \Z} (-1)^{\delta(t > k)}  (\alpha_i, \wt_t(\cb))$,
\item $\La( M, \dual^k \Rt_i) = 2 \max\{ y, s \}  + \sum_{t \in \Z} (-1)^{\delta(t < k)}  (\alpha_i, \wt_t(\cb))$,
\item $\de( \dual^k \Rt_i, M ) = \max\{ x,r \} + \max\{ y,s \} + (\alpha_i, \wt_k(\cb))  $.
\ee
\end{theorem}
\begin{proof}
Thanks to Proposition \ref{Prop: bijection BB} (iv), we may assume that $k=0$.
	
	We write $\cb = (b_k)_{k\in \Z}$ and set $M_k \seteq  \cLD (b_k) \in \catCD$ for $k\in \Z$.
Note that $ x = \eps_i(b_1) $, $ r = \ep_i(b_0) $, $ s = \eps_i(b_0) $, and $ y = \ep_i(b_{-1}) $ by the definition.
Set $ b_1' \seteq  \te_i^{* x } (b_1)$, $ b_{-1}' \seteq  \te_i^{ y } (b_{-1})$ and 
$$
 M_1'\seteq \cLD (  b_1')\quad \text{ and } \quad M_{-1}'\seteq\cLD ( b_{-1}').
$$ 
Note that $ M_1 \simeq M_1' \htens \Rt_i^{\tens x} $ and $ M_{-1} \simeq  
\Rt_i^{\tens  y}  \htens M_{-1}' $.
We define 
\begin{align*}
	X &\seteq  \hd(\cdots \tens \dual^{2} M_2 \tens \dual M_1'), \\
	N &\seteq  \hd( \dual \Rt_i^{\tens x} \tens M_0 \tens \dual^{-1} \Rt_i^{\tens y} ),\\
	Y &\seteq  \hd(  \dual^{-1} M_{-1}'  \tens \dual^{-2} M_{-2} \tens \cdots ).
\end{align*}
By Lemma \ref{Lem: simple head for Ns},  we have 
\begin{align*}
M &\simeq \hd(\cdots \tens \dual M_1 \tens M_0 \tens  \dual^{-1} M_{-1}\tens \cdots) \\	
&\simeq \hd (  \cdots \tens \dual M_1' \tens \dual \Rt_i^{\tens x}  \tens M_0 \tens \dual^{-1} ( \Rt_i)^{\tens y}  \tens \dual^{-1} M_{-1}' \tens \cdots ) \\
&\simeq \hd ( X \tens N \tens Y).
\end{align*} 
Moreover, the triple $(X,N,Y)$ is strongly unmixed and 
\begin{itemize}
\item $\de( \dual \Rt_i, Y) = \de( \Rt_i, Y)=0$ and $\de ( \Rt_i, X) = \de ( \dual^{-1}\Rt_i, X) = 0$,
\item $ \La(\Rt_i, Y) = \Li(\Rt_i, Y)$ and  $\La( \Rt_i, X) = - \La( X, \Rt_i) =  - \Li( X, \Rt_i) $,
\end{itemize}
 by \eqref{Eq: Li=La} and Proposition \ref{Prop: cusp}.
Thus,  it follows from \cite[Lemma 4.3, Corollary 4.4]{KKOP19C}, and Lemma \ref{Lem: N} that
\begin{align*}
	\La (\Rt_i, M) & = \La(\Rt_i, (X \htens N) \htens Y ) = \La(\Rt_i, X \htens N ) + \La(\Rt_i,  Y ) \\
	&= \La(\Rt_i, X  ) + \La(\Rt_i,   N ) + \La(\Rt_i,  Y ) \\
	&= - \Li(\Rt_i, X  ) + \La(\Rt_i,   N ) + \Li(\Rt_i,  Y ) \\
	&=  \Bigl(-\al_i, \sum_{t \ge1}  \wt_t(\cb) - x \al_i \Bigr) + 2 \max\{  x,r\}  +  (\alpha_i, \wt_0(\cb)) - 2x + 2y \\ 
	& \quad - \Bigl(- \al_i, \sum_{t \le -1} \wt_t (\cb) - y \al_i \Bigr) \\
	&= 2 \max\{ x, r \}  + \sum_{t\in \Z} (-1)^{\delta(t > 0)} (\al_i, \wt_t(\cb)).
\end{align*}

\snoi
(ii) can be proved in the same manner as above. 

\snoi
(iii) follows from (i) and (ii).
\end{proof}

\begin{remark}
Theorem \ref{Thm: La Li M}  can be understood as a quantum affine algebra analogue of \cite[Corollary 3.8]{KKOP18}.
\end{remark}

\vskip 2em

\section{Crystal description of $\sBD(\g)$ for affine type $A_n^{(1)}$ } \label{Sec: combi des}

In this section, we give a combinatorial description of the extended crystal $\sBD(\g)$ for affine type $A_n^{(1)}$ in terms of affine highest weights. 

Let $U_q'(\g)$ be the quantum affine algebra of affine type $A_n^{(1)}$, and let $ \rlQ_0$ be the root lattice of $\g_0$. 
 Note that $I=\st{0,1,\cdots,n}$, $I_0=I\setminus\st{0}$, and
$$(\al_i,\al_j)=2\,\delta(i=j)-\delta(i-j\equiv1 \bmod n+1)-\delta(i-j\equiv-1
\bmod n+1)
\qt{for $i,j\in I$.}$$ 
For $i\in I_0$ and $x \in \cor^\times$, we have 
\begin{align} \label{Eq: dual A}
	\dual( V(\varpi_i)_x ) \simeq V(\varpi_{n+1-i})_{x (-q)^{n+1}}.
\end{align}

We take as $\catCO$ the Hernandez-Leclerc category corresponding to 
$$
\sigZ \seteq  \{ (i, (-q)^{a}) \in I_0 \times \cor \mid   a-i \equiv 1 \bmod 2  \}.
$$

\subsection{Multisegments} \label{Sec: ms}\

A \emph{segment} is an interval $[a,b]$ for $  1 \le a \le b \le n $, and a \emph{multisegment} is a finite multiset of segments.
We set $[a,b]\seteq\emptyset$  if  $a > b$. 
When $a=b$, we simply write $[a] = [a,b]$. 
We set $\MS_n$ to be the set of multisegments. 
For a multisegment $\cm = \{ m_1, \ldots, m_k \}$, we sometimes write $\cm = m_1 + m_2 + \cdots + m_k$.

It is well-known that $\MS_n$ has an $A_n$-crystal structure and it is isomorphic to $B(\infty)$ as a crystal (see \cite{V01,CT15} for example).
Note that
the simple $R_{A_n}$-module corresponding to the segment $[a,b]$
is the one-dimensional$R(\al_{a,b})$-module
$$L[a, b]\seteq\dfrac{R(\al_{a,b})e(a,a+1,\ldots,b) }
{\sum_{k=1}^{b-a+1}R(\al_{a,b})x_ke(a,\ldots,b)+\sum_{k=1}^{b-a}R(\al_{a,b})\tau_ke(a,\ldots,b) }$$
where $\al_{a,b}\seteq\sum_{k=a}^b\al_k$.

We briefly review the crystal structure of $\MS_n$ following \cite{V01}. 
We set $\wt([a,b])\seteq - \al_{a,b}$ and define $ \wt (\cm) \seteq \sum_{t=1}^k  \wt(m_t)$ for $\cm = \sum_{t=1}^k m_t$.
We define total orders $<$ and $<'$ on segments by 
\begin{align}
[a, b]< [c, d]  \qquad  &\text{if either ($ b< d $) or ($b=d$ and $a >c$),}
 \label{Eq: left order} \\ 
	[a, b]<' [c, d]  \qquad  &\text{if either ($ a < c $) or ($a=c$ and $b > d$).} \label{Eq: right order}
\end{align}
for segments $[a,b]$ and $[c,d]$. Note that $<$ (resp.\ $<'$) is called the \emph{left order} (resp.\ \emph{right order}) in \cite{V01}. 
Note that for $\cm\in\MS_n$, the corresponding simple $R_{A_n}$-module is
$$\hd\bl L(m_1)\conv\cdots \conv L(m_k)\br
\simeq\hd\bl L(m'_1)\conv\cdots \conv L(m'_k)\br$$ where
$\cm=\sum_{t=1}^k m_t=\sum_{t=1}^k m'_t$ with
$m_k\le \cdots\le m_1$ and $m'_k\le' \cdots\le' m'_1$.

Let $i\in \{ 1,2, \ldots, n \}$ and $\cm \in \MS_n$.  Then the crystal operators $\tf_i$ and $\te_i$ are defined as follows.
\bna
\item[] We rearrange the segments of $\cm$ having the forms $[i,t]$ and $[i+1, t]$ from left to right by largest to smallest with respect to   
the order $<$ in \eqref{Eq: left order}. 
Then we put $-$ at each segment $[i,t]$ and $+$ at each segment $[i+1,t]$.  
This sequence is called the left $i$-signature sequence of $\cm$, which is denoted by $\rS^{<}_i(\cm)$. 
We then cancel out all $(+,-)$ pairs. If there is no +, then we define $\tf_i (\cm) \seteq  \cm \cup \{ [i] \}$. Otherwise, $\tf_i (\cm)$ is defined to be the multisegment obtained from $\cm$ by replacing $[i+1, t]$ placed at the left most + by $[i,t]$.  
If there is no $-$, then we define $\te_i (\cm) \seteq  0$. Otherwise, $\te_i (\cm)$ is defined to be the multisegment obtained from $\cm$ by replacing $[i, t]$ placed at the right most $-$ by $[i+ 1,t]$.  
Furthermore, $\ep_i(\cm)$ is the number of the remaining $-$'s. 

\ee
The crystal operators $\tfs_i$ and $\tes_i$ can be defined in a similar manner. 
\bna
\item [] 
We rearrange the segments of $\cm$ having the form $[t,i]$ and $[t, i-1]$ from left to right by largest to smallest with respect to 
the order $<'$ in \eqref{Eq: right order}.
Then we put $+$ at each segment $[t,i]$ and $-$ at each segment $[t, i-1]$. 
This sequence is called  the right  $i$-signature sequence of $\cm$, which is denoted by $\rS_i^{<'}(\cm)$. 
We then cancel out all $(+,-)$ pairs. If there is no $-$, then we define $\tfs_i (\cm) \seteq  \cm \cup \{ [i] \}$. Otherwise, $\tfs_i (\cm)$ is defined to be the multisegment obtained from $\cm$ by replacing $[t, i-1]$ placed at the right most $-$ by $[t, i]$.      
If there is no $+$, then we define $\tes_i (\cm) \seteq  0$. Otherwise, $\tes_i (\cm)$ is defined to be the multisegment obtained from $\cm$ by replacing $[t, i]$ placed at the leftmost most $+$ by $[t, i-1]$. 
In this case, $\eps_i(\cm)$ is the number of the remaining $+$'s.
\ee

\Rem
Note that $\te_i$ and $\tes_i$ in our paper correspond to
$\widehat{\tilde{\mathbf{e}}}_i$ and $\widetilde{\mathbf{e}}_i$ in
\cite[Section 2.3]{V01}, respectively.
We remark also that we swap $+$ and $-$ in the signature rule 
in  \cite[Rule 1 in Section 2.3]{V01} in order to match it  with the extended crystal signature rule given in Section \ref{Sec: crystal rule}.
\enrem

\subsection{Hernandez-Leclerc category $\catCQ$} \label{Sec: HL CQ} \

Let $\qQ$ be the Q-datum consisting of the Dynkin quiver 
$$
\xymatrix{
	1 \ar@{<-}[r]& 2 \ar@{<-}[r]   & \cdots \ar@{<-}[r]   &n-1 \ar@{<-}[r] & n
}
$$
with the height function $\xi(i) \seteq  i-1 $ 
for $i\in I_0 =\st{1,\ldots,n}$, and let $\ddD$ be the complete duality datum arising from $\qQ$:
\begin{align} \label{Eq: dd from Q}
\ddD = \{ \Rt_i\}_{i\in I_0},
\qt{where $\Rt_i \seteq   V(\varpi_1)_{(-q)^{2i-2}} $ for $i\in I_0$.}
\end{align}
We consider the subcategory $\catCD$.
Note that the category $\catCD$ coincides with the subcategory $\catCQ$ determined by $\qQ$ (see \cite[Section 6]{KKOP20B} for details).
For any $k\in \Z$, we denote by $\dual^k(\catCQ)$ the full subcategory of $\catCO$ whose objects are $\dual^k(M)$ for all $M \in \catCQ$, and set 
\eq
\dual^k (\sigQ) \seteq \{  (i,(-q)^a) \in \sigZ \mid V(\varpi_i)_{(-q)^a} \in \dual^k (\catCQ)   \}.
\label{def:Dsig}
\eneq
Note that
\begin{align} \label{Eq: fm in CQ}
	\sigQ = \{ (i,(-q)^a) \in \sigZ \mid  i-1 \le a \le 2n-1-i   \}.	
\end{align}

Let 
\begin{align} \label{Eq: rx w_0}
\rxw = (s_1) (s_2 s_1) (s_3 s_2s_1) \cdots (s_n s_{n-1} \cdots s_1)
\end{align}
be a $\qQ$-adapted reduced expression of the longest element $w_0$ in the symmetric group $\sg_{n+1}$. 
Let $\ell \seteq  n(n+1)/2$ and write $( i_1, i_2, \ldots, i_\ell) = (1,2,1, 3,2,1, \ldots, n ,n-1, \ldots, 1)$. For $k=1, \ldots, \ell$, we denote by $E^*(\beta_k)$ the dual PBW vector corresponding to $\rxw$ and $\beta_k \seteq s_{i_1}\cdots s_{i_{k-1}} (\alpha_{i_k})$.

The complexified Grothendieck ring $ \C \tens_\Z K(\catCQ)$ is isomorphic to the coordinate ring $\C[N]$ of the unipotent group $N$ associated with the simple Lie algebra of type $A_n$ and, under this isomorphism, the set of the isomorphism classes of simple modules in $\catCQ$ corresponds to the \emph{upper global basis} (or \emph{dual canonical basis}) of $\C[N]$ (\cite{HL15}).
Recall that
$\alpha_{a,b} $ is the positive root $ \sum_{k=a}^b \al_k$
for $1 \le a \le b \le n$. 
As dual PBW vectors are contained in the upper global basis, under this isomorphism, the PBW vector $E^*(\alpha_{a,b})$ corresponds to the fundamental module $V( \varpi_{ b-a+1} )_{(-q)^{ b+a-2}}$ in $\catCQ$. 
Since a segment $[a,b]$ is in 1-1 correspondence with a positive root $\alpha_{a,b}$, 
the correspondence 
\begin{align} \label{Eq: seg sig} 
[a,b] \mapsto V( \varpi_{ b-a+1} )_{(-q)^{ b+a-2}} 
\end{align}
gives a bijection between the set of segments and the set of fundamental modules in $\catCQ$. 
Thus the fundamental modules in $\catCQ$ give the dual PBW basis of $ \C \tens_\Z K(\catCQ)$ associated with $\rxw$, and 
 the set $B_\qQ$ of the isomorphism classes of simple modules in $\catCQ$ has a crystal structure which is isomorphic to $\MS_n \simeq B(\infty)$ (see Lemma \ref{Lem: B and catCD}).
For any multisegment $\cm = \sum_{k} [a_k, b_k]$, let $V_\cm$ be the simple  module with the affine highest weight $ \sum_{k} (b_k - a_k + 1, (-q)^{b_k + a_k -2} )$ (see Theorem \ref{Thm: basic properties} \ref{Thm: bp5}). Then, 
the correspondence $\cm \mapsto V_\cm$ gives a bijection 
\begin{align} \label{Eq: MS CQ}
 \phi_0 \col \MS_n \isoto B_\qQ,\qt{which is a crystal isomorphism.}
\end{align}

\subsection{Extended crystal realization} \label{Sec: crystal rule} \

We use the same notations in the previous subsections. Let  
$$
\crhI_n \seteq  \{ (i, a) \in I_0 \times \Z \mid   a-i \equiv 1 \bmod 2  \}  
$$
and  
\begin{align*}
	\crB \seteq (\Z_{\ge0})^{\oplus  \crhI_n}.
\end{align*}
For an element $ \la \in \crB$, we write  
$ \la = \sum_{(i,a) \in \crhI_n} c_{i,a}(\la)  (i,a)  $ with $c_{i,a}(\la) \in \Z_{\ge 0}$.
 We regard $ \crhI_n$ as a subset of $\crB$.
For $(i,a) \in \crhI_n$,  we define 
\begin{align} \label{Eq: cdual}
\cdual(i,a) \seteq  (n+1 - i, a+n+1) \quad \text{ and } \quad  \cdual^{-1}(i,a) \seteq  (n+1 - i, a-n-1).
\end{align}
We extend to $\cdual^k(i,a)$ for any $k\in \Z$ and, for a subset $A \subset \crhI_n$, define $ \cdual^k (A)\seteq \{ \cdual^k (i,a) \mid (i,a) \in A  \}$.
We set 
$$
 \crhI_n^0 \seteq  \{  (i,a) \in  \crhI_n  \mid  i-1 \le a \le 2n-1-i  \},
 $$
  and define
\begin{equation} \label{Eq: S P}
	\begin{aligned}
		\crhI_n^k &\seteq  \cdual^k( \crhI_n^0) \quad \text{and} \quad \crB^k \seteq  (\Z_{\ge0})^{\oplus  \crhI_n^k} \qquad \text{for any $k\in \Z$.}
	\end{aligned}
\end{equation}

Let us recall $ \awlP^+\seteq\Z_{\ge0}^{\oplus  \sig}$,
the set of affine highest weights defined in Theorem \ref{Thm: basic properties} \ref{Thm: bp5}.
The following lemma can be proved easily by using \eqref{Eq: dual A}, \eqref{Eq: fm in CQ},  \eqref{Eq: cdual} and  the definitions of $ \crhI_n$ and $\crhI_n^k$.
\begin{lemma} \label{Lem: bijection psi} \
\bni 
\item The map $\psi_n\col  \crhI_n  \isoto \sigZ$ defined by 
$$
\psi_n(i,a) = (i, (-q)^a) \qquad \text{ for any $(i,a) \in \crhI_n$}
$$
is bijective. We extend it to the bijection between $\crB$ and $\awlP^+$, which is denoted by the same notation $\psi_n\col \crB  \isoto \awlP^+$.
\item For any $k\in \Z$, the restriction of $\psi_n$ to the subset $ \crhI_n^k$ gives the bijection 
$$
 \psi_n \col  \crhI_n^k  \isoto \dual^k (\sigQ),
$$
where $\dual^k (\sigQ)$ is defined in \eqref{def:Dsig}. 
We extend it to the bijection between $\crB^k$ and $\awlP^+_{\qQ,k} \seteq \Z_{\ge0}^{\oplus  \dual^k (\sigQ)} $, which is denoted by the same notation $\psi_n\col \crB^k  
\isoto \awlP^+_{\qQ,k}$.

\item For any $k,l\in \Z$ with $k \ne l$, we have $\crhI_n^k \cap \crhI_n^l = \emptyset$ and 
\begin{align*}
   \crhI_n = \bigsqcup_{k\in \Z} \crhI_n^k \quad \text{ and } \quad \crB = \soplus_{ k\in \Z} \crB^k.
\end{align*}
\ee	
\end{lemma}

\begin{example} \label{Ex: S3k}
Let $n=3$. Then we have 
\begin{align*}
\crhI_3^{-1} &= \{ (3,-4), (3,-2), (3,0), (2,-3), (2,-1), (1,-2)  \},\\
\crhI_3^0 &= \{ (1,0), (1,2), (1,4), (2,1), (2,3), (3,2)  \},\\
\crhI_3^{1} &= \{ (3,4), (3,6), (3,8), (2,5), (2,7), (1,6)  \}.
\end{align*}	
Pictorially, the sets $ \crhI_3^{-1}$, $\crhI_3^{0}$ and $\crhI_3^{1}$ can be drawn as in the following figure
where the elements of $\crhI_3^{-1}$, $ \crhI_3^{0}$ and $\crhI_3^{1}$ are denoted 
by $\ast$, $\bullet$ and $\trg$, respectively. 
$$ 
\scalebox{0.8}{\xymatrix@C=1ex@R=  0.0ex{ 
		i \diagdown\, a     \ar@{-}[dddd]<3ex> \ar@{-}[rrrrrrrrrrrrrrrr]<-2.0ex>    & -5 & -4 & -3 & -2 & -1 & 0 &\  1 & \ 2& \ 3 &\  4 &\  5&\  6 & \ 7& \ 8 & & &   \\1       & &   & & \ast  &  & \bullet  & &  \bullet & & \bullet && \trg &&   \\
		2     &  & & \ast & & \ast &  &  \bullet  & &  \bullet  && \trg && \trg  \\
		3      & & \ast & & \ast & &  \ast  & &  \bullet && \trg && \trg && \trg \\
		& 
}}
$$	
\end{example}

\smallskip

Since Theorem \ref{Thm: basic properties} \ref{Thm: bp5} tells us that simple modules in $\catCO$ are parameterized by elements of $\awlP^+$ as affine highest weights,
 they are also parameterized by $\crB$ via the map $\psi_n$ defined in Lemma \ref{Lem: bijection psi}.  
For $\la \in \crB$, we denote by $V(\la)$ the corresponding simple module in $\catCO$ whose affine highest weight is $\psi_n(\la)$. 
Thus the map  
\begin{align} \label{Eq: Psi}
	\Psi_n\col   \crB  \buildrel \sim \over \longrightarrow   \sBD(\g), \qquad \la \mapsto V(\la) \quad \text{ for $\la \in \crB$}
\end{align}
is bijective. 
Note that 
$V(i,a) \seteq V(\varpi_{i})_{(-q)^a}$ and $ \dual^m \bl V( i,a)\br\simeq  V\bl \cdual^m ( i,a) \br $ for $m\in \Z$.

\medskip

We are now ready to define an extended crystal structure on $ \crB$.

First we define the weight $\hwt$ on $ \crB$ as follows.
For any $(i,a) \in \crhI_n^0 $,  define 
$$
\wt(i, a) \seteq  -\sum_{ k=A}^{B}   \al_k \in \rlQ_{0}, 
$$ 
where $ A =  \frac{a-i+3}{2}  $, $B = \frac{ a+i+1}{2}$. 
By Lemma \ref{Lem: bijection psi} (iii),
for any $(i,a) \in \crhI_n$, there exists a unique $k\in \Z$ such that $ \cdual^k(i,a) \in \crhI_n^0 $. Define  
\begin{align} \label{Eq: wt i a}
\wt(i,a) = (-1)^k \wt (\cdual^k(i,a) ).
\end{align}
We finally define $\hwt\col  \crB \rightarrow \rlQ_0 $ by 
$$
\hwt(\la) \seteq  \sum_{  (i,a) \in \crhI_n } c_{i,a}(\la) \wt(i,a)  \qquad \text{for $\la  = \sum_{(i,a) \in \crhI_n} c_{i,a}(\la)  (i,a) \in \crB $}.
$$

We now explain the extended crystal operators on $ \crB$.
For any $i\in I_0$, we define $ \cru_i \seteq  (1,  2(i-1)) \in  \crhI_n$ and set
\begin{align} \label{Eq: crD}
\crD \seteq  \{ \cru_i  \mid i \in I_0 \} \subset \crhI_n.
\end{align}
For any $(i,a) \in \crhI_n$, we define 
\eq
&&\ba{rcl}
	S (i,a) & \seteq & \{ (j,b)  \in \crhI_n \mid  \ j-b = i-a   \} \subset \crhI_n, \\	
	S' (i,a) & \seteq & \{ (j,b)  \in \crhI_n \mid  \ j+b = i+a   \} \subset \crhI_n,
\ea\eneq
and set $\crsh_t(i,a) \seteq  (i, a+t)$ for $t \in 2\Z$.

Let $\la \in \crB$ and $ (i,k) \in \cIz $. 
We shall define 
$ \tF_{i,k} (\la) $ and $ \tE_{i,k} (\la) $. 
They will correspond to the operators $\ttF_{i,k}$ and $\ttE_{i,k}$
on $\sBD(\g)$.

\smallskip

 Set $\la  = \sum_{(i,a) \in \crhI_n} \crc_{(i,a)}(\la)\; (i,a) \in \crB $
with $\crc_{(i,a)}(\la)\in\Z_{\ge0}$.

\mnoi
\textbf{(Step 1)}\   Set 
$$ \cru \seteq   \cdual^k(\cru_i)
=\bc \bl 1, 2(i-1)+k(n+1)\br&\text{if $k$ is even,}\\
 \bl n, 2(i-1)+k(n+1)\br&\text{if $k$ is odd.}\ec $$

We define $S_{i,k}^{-}$ and $S_{i,k}^{+}$ as follows:
\eqn
S_{i,k}^{-}&& \seteq \bc S(\cru)&\text{if $k$ is even,}\\
 S'(\cru)&\text{if $k$ is odd,}\\
\ec\\
S_{i,k}^{+}&&\seteq \bc S\bl\sh_2(\cru)\br&\text{if $k$ is even,}\\
 S'\bl\sh_2(\cru)\br&\text{if $k$ is odd,}\\
\ec
\eneqn
Set $ S_{ i,k }\seteq  S_{i,k}^{-}\cup S_{i,k}^{+}$.
 Let $\sck$ be the total order on $\crhI_n$ 
defined by 
\eqn
(j,a)\sck  (j',a')&\Longleftrightarrow &
\bc\text{($j>j'$) or ($j=j'$ and $a> a'$)}&\text{if $k$ is even,}\\[.5ex]
\text{($j<j'$) or ($j=j'$ and $a > a'$)}
&\text{if $k$ is odd}
\ec
\eneqn
Note that for $\mu=\sum_{t=1}^r(j_t,a_t)\in \crB $ with $(j_t, a_t ) \in S_{ i,k }$  ($t=1,\ldots, r$) and 
$(j_1,a_1)\scke \cdots \ \allowbreak  \scke(j_r,a_r)$, we have 
$$V(\mu)\simeq \hd\bl
V(\vp_{j_1})_{(-q)^{a_1}}\tens\cdots\tens V(\vp_{j_r})_{(-q)^{a_r}}\br.$$

We then rearrange elements of  $ S_{i,k}$ from left to right in decreasing order with respect to the order 
$\sck$:
\begin{align} \label{Eq: cr a_k}
	S_{i,k}= \{ \ca_{2n},\ca_{2n-1},\cdots,\ca_{1}  \}
\subset\crhI_n\qt{with  $\ca_{2n} \sck \ca_{2n-1} \sck\cdots \sck\ca_{1}$.}
\end{align}

\snoi
For simplicity, we set $\ca_t = 0$ unless $1\le t\le 2n$.
Note that $\ca_t\in S_{i,k}^-$ for any odd $t$ and $\ca_t\in S_{i,k}^+$ 
for any even $t$.

\mnoi
\textbf{(Step 2)}
\ For $t \in \Z_{> 0}$, let $-^t \seteq  \underbrace{- \cdots -}_{t}$ and $+^t \seteq  \underbrace{+ \cdots +}_{t}$, where $+$ and $-$ are symbols. 
We set $-^0 \seteq  \edot$ and $+^0 \seteq\edot$. 
For $k=1, 2, \ldots,  2n$, we define 
$$ 
\bfs_k \seteq  
\begin{cases}
	+^{ \crc_{ \ca_k }(\la) } & \text{ if }  \ca_k \in S_{i,k}^+,\\
	-^{ \crc_{ \ca_k }(\la) } & \text{ if }   \ca_k \in S_{i,k}^-,
\end{cases}
$$  
and define the sequence $\bfs$ as the concatenation $ \bfs_{2n} * \bfs_{2n-1}* \cdots *\bfs_{1} $ of $\bfs_k$'s. 
This sequence is called the \emph{$(i,k)$-signature sequence} of $\la$.  
We cancel out all the possible $(+,-)$ pairs  to obtain a sequence of $-$'s followed by $+$'s.
The resulting sequence is called the \emph{reduced $(i,k)$-signature sequence} of $\la$.

\mnoi
\textbf{(Step 3)}\ Let $\ca_t$ be the element of $S_{i,k}$ corresponding to the leftmost $+$ in the reduced $(i,k)$-signature sequence of $\la$.
 If such an $\ca_t$ exists, then we define 
$$
\tF_{i,k} (\la) \seteq  \la - \ca_t + \ca_{t+1}.
$$
Otherwise, we define 
$$
\tF_{i,k} (\la) \seteq  \la + \ca_1.
$$

Similarly, let $\ca_s$ be the element of $S$ corresponding to the rightmost $-$ in the reduced $(i,k)$-signature sequence of $\la$.
 If such an $\ca_s$ exists, then we define
$$
\tE_{i,k} (\la) \seteq  \la - \ca_s + \ca_{s-1}.
$$
Otherwise, we define 
$$
\tE_{i,k} (\la) \seteq  \la + \ca_{2n}.
$$

 Note that these operators $\tF_{i,k}$ and $\tE_{i,k} $ are compatible with the duality operator $D$
as seen in the next lemma.
\Lemma\label{lem:dualcr}
We have
\bnum
\item $D\bl S(\cru)\br=S'(D\cru)$
and $D\bl S'(\cru)\br=S(D\cru)$ for any $\cru\in \crhI_n$,
\item $D\bl S_{i,k}^{\pm}\br=S_{i,k+1}^{\pm}$,
and $D\bl S_{i,k}\br=S_{i,k+1}$,
\item for any $(j,a)$, $(j',a')\in\crhI_n$, we have
$D(j,a)  \ms{3mu}{\raisebox{-1.3ex}{$\scriptstyle{k+1}$}\hs{-2.5ex}{\succ}}\ms{7mu}
 D(j',a')$ if and only if $(j,a)\sck(j'a')$,
\item \label{it:DF} $D\circ \tF_{i,k}=\tF_{i,k+1}\circ D$.
\ee
\enlemma
\Proof
(i) and (iii) immediately follow from the definition \eqref{Eq: cdual}.

\snoi
(ii) follows from (i).

\snoi
(iv) easily follows from (ii), (iii).
\QED 

\medskip

We thus obtain the $\cIz$-colored graph structure on $ \crB$ induced by the operators $\tF_{i,k}$  as follows:
take $\crB $ as the set of vertices and define the $\cIz$-colored arrows on $\crB$ by
$$
\la \To[\;(i,k)\;]\la' \quad  \text{ if and only if} \quad  \la' = \tF_{i,k} (\la) \qquad (\;(i,k)\in \cIz).
$$

The following theorem is the main theorem of this section whose proof is postponed until the next subsection.  
\begin{theorem}  \label{Thm: CR}
Let $\ddD \seteq  \{ \Rt_i\}_{i\in I_0}$, where $\Rt_i \seteq   V(\varpi_1)_{(-q)^{2i-2}} $ for any $i\in I_0=\st{1,\ldots,n}$. 
\bni

\item There is a bijection $ \Upsilon_n \col  \cBg[\g_0]  \buildrel \sim \over \longrightarrow \crB$ such that 
\bna
\item $\Upsilon_n (\one) = 0$,
\item 
$\tF_{i,k} (  \Upsilon_n(    \bfb )) = \Upsilon_n(  \tF_{i,k}  ( \bfb) )$ and $\tE_{i,k} (  \Upsilon_n(    \bfb )) = \Upsilon_n(  \tE_{i,k}  ( \bfb) )$, 
\item  $\hwt (  \Upsilon_n(    \bfb )) = \hwt( \bfb)$. 
\ee

\item For $(i,k)\in \cIz $ and $\la \in  \crB $, we have 
\begin{align*}
	\ttF_{i,k} (  \Psi_n(    \la )) = \Psi_n(  \tF_{i,k}  ( \la) ) , \qquad  \ttE_{i,k} (  \Psi_n(    \la )) = \Psi_n(  \tE_{i,k}  ( \la) ),
\end{align*}
where $\Psi_n  \col   \crB  \buildrel \sim \over \longrightarrow   \sBD(\g) $ is the bijection defined in \eqref{Eq: Psi}.
Hence the following diagram 
$$
\xymatrix{
\cBg[\g_0]  \ar[rr]^{\Phi_\ddD}  \ar[rd]_{\Upsilon_n} &  & \sBD(\g) \\
 & \crB  \ar[ur]_{\Psi_n} &
}
$$
commutes, and the arrows are
$\cIz$-colored graph isomorphisms. 

\ee
\end{theorem}

\begin{example}

	Let $n=3$. In this case, we have $\crD = \{\cru_1, \cru_2, \cru_3\}$, where $ \cru_1 = (1,  0)$, $ \cru_2 = (1,  2)$, $ \cru_3 = (1, 4)$, and the corresponding duality datum is 
	$$ \ddD =  \{  V(\varpi_1),    V(\varpi_1)_{(-q)^2},   V(\varpi_1)_{(-q)^4}  \}. $$
	We choose $\la \in \crBB{3}$ as follows:
	\begin{equation*} 
		\begin{aligned} 
			\la =& \   (3,  {-4}) + (3,  {-2}) + 2(2,  {-1}) + (1,  {-2} )+ (1,  {2} ) + (2,  1  ) + (2,  {3} ) \\
			&  + 2 (3,  {4} ) + (2,  {5} ) + (2,  {7} ). 
		\end{aligned}
	\end{equation*}
	Pictorially, we write the coefficients of $\la$ as follows:
	\begin{align*} 
		\scalebox{0.8}{\xymatrix@C=1ex@R=  0.0ex{ 
				i \backslash a     \ar@{-}[ddddddd]<3ex> \ar@{-}[rrrrrrrrrrrrrrrr]<-2.0ex>    & -5 & -4 & -3 & -2 & -1 & 0 &\  1 & \ 2& \ 3 &\  4 &\  5&\  6 & \ 7& \ 8 & & &   \\ \\
				1       & &  \cdot & & 1  &  &\cdot  & &  1 & & \cdot && \cdot && \cdot  \\ \
				\\
				2     & \cdot & & \cdot & & 2 &  &  1  & &  1  && 1 && 1  \\ 
				\\
				3      & & 1 & & 1 & &  \cdot  & &   \cdot && 2 && \cdot && \cdot \\ 
				& 
		}}
	\end{align*}
	
	\bni
	\item Let $(i,k)=(1,0) \in \cIz$. Then, $\cru = \dual^0( \cru_1) = \cru_1$ and 
	\begin{align*}
		S_{1,0}^- &= S(\cru) = \{  (1,0), (2,1), (3,2) \}, \\
		S_{1,0}^+ &= S(\sh_2(\cru)) = \{  (1,2), (2,3), (3,4) \},
	\end{align*}
	and
	\begin{align*}
		S_{1,0} =  \{ (3,4) \sck[0](3,2) \sck[0] (2,3)  \sck[0] (2,1) 
  \sck[0](1,2)  \sck[0] (1,0)    \}. 	
	\end{align*}
	We use the notations $\ca_k$ for the element of $S_{1,0}$ as in \eqref{Eq: cr a_k}.
	For $k=1,2, \ldots, 6$, we obtain 
	\begin{align*}
		\bfs_6 = ++, \quad \bfs_5 = \edot, \quad \bfs_4 = +, \quad 		\bfs_3 = -, \quad 	\bfs_2  = +, \quad 		\bfs_1  = \edot, 
	\end{align*} 
	which gives the sequence $\bfs = \bfs_6 * \bfs_5 * \bfs_4 * \bfs_3 * \bfs_2 * \bfs_1$.
	Canceling out all $(+,-)$ pairs, we have the following reduced $(1,0)$-signature sequence of $\la$:
	$$
	\begin{array}{c|cccccc}
		& \ca_6 & \ca_5 & \ca_4 & \ca_3 & \ca_2 & \ca_1\\
		\hline
		\text{reduced $(1,0)$-signature} & ++ & \edot & \edot & \edot& + & \edot 
	\end{array}
	$$
	Since $\ca_6$ is located at the leftmost + and $\ca_t=0$ for $t > 6$, we have 
	\begin{align*}
		\tF_{1,0} (\la) &= \la - \ca_{6} + \ca_7 = \la - (3,4),
	\end{align*}	
	and by Theorem \ref{Thm: CR}, we have 
	$$
	\ttF_{1,0}( V(\la)) = V( \tF_{1,0} (\la)) = V( \la - (3,4) ).
	$$

	\item Let $(i,k)=(1,-1) \in \cIz$. In this case, we have $\cru = \dual^{-1}( \cru_1) = (3,-4)$ and 
	\begin{align*}
		S_{1,-1}^-&= S'(\cru) = \{  (1,-2), (2,-3), (3,-4) \}, \\
		S_{1,-1}^+ &= S'(\sh_2(\cru)) = \{  (1,0), (2,-1), (3,-2) \},
	\end{align*}
	and 
	\begin{align*}
		S_{1,-1} =  \{  (1,0) \sck[-1] (1,-2)  \sck[-1] (2,-1) \sck[-1]  (2,-3)  \sck[-1]  (3,-2)  \sck[-1]  (3,-4)    \}. 	
	\end{align*}
	We use the notations $\ca_k$ for the element of $S_{1,-1}$ as in \eqref{Eq: cr a_k}.
	For $k=1,2, \ldots, 6$, we obtain 
	\begin{align*}
		\bfs_6 = \edot, \quad 	\bfs_5 = -, \quad  	\bfs_4 = ++, \quad \bfs_3 = \edot, \quad 	  \bfs_2  = +, \quad \bfs_1 & = -,
	\end{align*} 
	which gives the sequence $\bfs = \bfs_6 * \bfs_5 * \bfs_4 * \bfs_3 * \bfs_2 * \bfs_1$.
	Canceling out all $(+,-)$ pairs, we have the following reduced $(1,-1)$-signature sequence of $\la$: 
	$$
	\begin{array}{c|cccccc}
		& \ca_6 & \ca_5 & \ca_4 & \ca_3 & \ca_2 & \ca_1\\
		\hline
		\text{reduced $(1,-1)$-signature} & \edot & - & ++ & \edot& \edot & \edot 
	\end{array}
	$$
	Since $\ca_4$ is located at the leftmost +, we have 
	\begin{align*}
		\tF_{1,-1} (\la) &= \la - \ca_{4} + \ca_5 = \la - (2,-1) + (1,-2).
	\end{align*}
	and by Theorem \ref{Thm: CR}, we have 
	$$
	\ttF_{1,-1}( V(\la)) =  V( \tF_{1,-1}(\la)) = V( \la - (2,-1) + (1,-2) ).
	$$
	\ee

\end{example}

\subsection{Proof of Theorem \ref{Thm: CR}} \

In this subsection, we will prove Theorem \ref{Thm: CR}.
We employ the same notations introduced in the previous subsections.

Recall that the set $\MS_n$ of multisegments defined in Section \ref{Sec: ms} has a crystal structure which is isomorphic to the crystal $B(\infty)$. 
For any $k\in \Z$, the correspondence $[a,b] \mapsto \cdual^k (b-a+1, b+a-2)$ gives a bijection between the set of segments and $\crhI_n^k$. Thus this induces a bijection 
$ \gamma_k\col \MS_n \buildrel \sim \over \longrightarrow \crB^k$. 

On the other hand, for any $k\in \Z$, let $B_\qQ^k$ be the set of the isomorphism classes of simple modules in $\dual^k (\catCQ)$. 
Since $B_\qQ^k$ is equal to
$B_{\dual^k\ddD}$,  $B_\qQ^k$ has a crystal structure which is isomorphic to $B(\infty)$ (see Lemma \ref{Lem: B and catCD}  and Lemma~\ref{lem:dual}). Since the elements of $B_\qQ^k$ are  parameterized by affine highest weights in $\awlP^+_{\qQ,k}$, the map
$$
\Psi_n^k\col\crB^k  \longrightarrow B_\qQ^k, \qquad \la \mapsto V(\la)
$$
is bijective. Note that $\Psi_n^k =  \Psi_n |_{ \crB^k }$, where $\Psi_n$ is defined in \eqref{Eq: Psi}. 
Thus we have the bijection $\phi_k\seteq \Psi_n^k \circ \gamma_k\col \MS_n  \buildrel \sim \over \longrightarrow B_\qQ^k$, i.e., the diagram
$$
\xymatrix{
 B(\infty)\simeq  \MS_n  \ar[rr]^{\phi_k} \ar[dr]_{\gamma_k} & & B_\qQ^k \simeq B(\infty) \\
 & \crB^k \ar[ur]_{\Psi_n^k} &
}
$$
commutes. 
By the definition \eqref{Eq: wt i a}, we have 
\begin{equation} \label{Eq: wt gamma k}
\begin{aligned}
\wt( \gamma_k ([a,b]) ) &= (-1)^{k} \wt( b-a+1, b+a-2)
 =\smash{ (-1)^{k+1} \sum_{k=a}^{b} \alpha_k}	\\
 & = (-1)^k \wt([a,b]).
\end{aligned}
\end{equation}

\begin{example} \label{Ex: conf with seg}
	We use the same notations given in Example \ref{Ex: S3k}.
In the table below, we write the segments $\gamma_{k}^{-1}(i,a)$ (where $(i,a)\in \crhI_3^{k}$ for $k=-1,0,1$) at the position $(i,a)$: 
	$$ 
	\scalebox{0.8}{\xymatrix@C=1ex@R=  0.0ex{ 
			i \backslash a     \ar@{-}[dddd]<3ex> \ar@{-}[rrrrrrrrrrrrrrrr]<-2.0ex>    & -5 & -4 & -3 & -2 & -1 & 0 &\  1 & \ 2& \ 3 &\  4 &\  5&\  6 & \ 7& \ 8 & & &   \\
			1       & &   & & \dul{[1,3]} & & [1]    & &  [2] & & [3]  && \ul{[1,3]}    &&   \\
			2     &  & & \dul{[1,2]} & & \dul{[2,3]} &  &  [1,2]  & &  [2,3]  && \ul{[1,2]} && \ul{[2,3]}   \\
			3      & & \dul{[1]} & & \dul{[2]} & & \dul{[3]}  & &  [1,3] && \ul{[1]} && \ul{[2]} && \ul{[3]} && \\
			& &&&&&&&&&&&&&&&
	}}
	$$		
Here, the underlined (resp.\ double underlined) segments are in the image of $\crhI_3^{1}$ (resp.\ $\crhI_3^{-1}$) under the map $\gamma_{1}^{-1}$ (resp.\ $\gamma_{-1}^{-1}$).	
\end{example}
\smallskip

From now on, we identify $\MS_n$ with $B(\infty)$ as a crystal.
 Similarly to the extended crystal $\cBg[\g_0]$, 
we define $\hMS_n $ by
\begin{align*}
\hMS_n \seteq  \Bigl\{  (\cm_k)_{k\in \Z } \in \prod_{k\in \Z} \MS_n
\biggm| \cm_k =\emptyset \text{ for all but finitely many $k$'s} \Bigr\} .
\end{align*} 
Define the map
\begin{align*}
& \gamma\col \hMS_n \longrightarrow  \crB, \qquad \cb = (\cm_k)_{k\in Z}  \mapsto \sum_{k\in \Z} \gamma_k(\cm_k).
\end{align*}
Thanks to Lemma \ref{Lem: bijection psi}, the map $\gamma$ is bijective. Let $\phi = \Psi_n \circ \gamma$. Then we have the following commutative diagram 
$$
\xymatrix@R=3ex{
	 \hMS_n  \ar[rr]^{\phi} \ar[dr]_{\gamma} & &  \sBD(\g) \\
	& \crB \ar[ur]_{\Psi_n} &
},
$$
in which all the arrows are bijective. 
For any $ \cb \in \hMS_n$, we have 
\begin{align} 
\gamma\circ \cd (\cb) &=  	\cdual \circ \gamma(\cb), \label{Eq: gamma D}  \\
 \hwt (\cb) &= \hwt \circ \gamma (\cb), \label{Eq: gamma hwt} \\
\gamma(\emptyset) &= 0, \label{Eq: gamma hw vector}
\end{align}
 where $ \cd$ is defined by 
$\cd\bl(\cm_k)_{k\in \Z }\br=(\cm_{k-1})_{k\in \Z }$
and $\cdual$ is defined by \eqref{Eq: cdual}. 
Note that \eqref{Eq: gamma hwt} 
follows  directly from \eqref{Eq: wt gamma k}.

\begin{lemma} \label{Lem: phi iso}
Let $\cb = (\cm_k)_{k\in \Z} \in \hMS_n$ and set $V_k\seteq \phi_k(\cm_k) \in B_\qQ^k$ for $k\in \Z$.	Then 
$$
\phi(\cb) \simeq \hd (\cdots \tens V_2 \tens V_1 \tens V_0 \tens V_{-1} \tens \cdots).
$$
Thus, if we identify $\hMS_n  $ with 
$\cBg[\g_0]$ by extending the isomorphism $\MS_n\simeq B(\infty)$,  then $\phi$ coincides with the map $\Phi_\ddD$ defined in Proposition \ref{Prop: bijection BB}. 
\end{lemma}

\begin{proof}
Let us recall the set $\crD =  \{ \cru_i  \mid i \in I_0 \} $ defined in \eqref{Eq: crD}.
Since the duality datum $\ddD \seteq \{ V(\cru_i)  \}_{i\in I_0}$ is equal to the complete duality datum defined in \eqref{Eq: dd from Q}, 
the fundamental modules in $\catCO$ forms the affine cuspidal modules corresponding to $\ddD$ and $\rxw$ defined in \eqref{Eq: rx w_0}.   

Let $\cb = (\cm_k)_{k\in \Z} \in \hMS_n$ and set $\la = \gamma(\cb)$. 
By writing $\la =\sum_{k\in  \Z} \la_k$ with $ \la_k \in \crB^k$ ($k\in \Z$),
we have
\begin{align*}
V(\la) &\simeq \hd ( \cdots \tens V( \la_2 ) \tens V( \la_1 ) \tens V( \la_0 ) \tens V( \la_{-1}) \tens \cdots ).
\end{align*}
Thus, by the definition of $\gamma$, we have $ \gamma_k(\cm_k) = \la_k$ for $k\in \Z$, which gives the assertion.
\end{proof}

By the isomorphism $\hMS_n\simeq \cBg[\g_0]$, the extended crystal operators
act on $\hMS_n$.
We will prove that $\gamma$ commutes with the extended crystal operator action
where the extended crystal operator action on $\crB$ is described in Section \ref{Sec: crystal rule}.

In order to show this, we need a couple of lemmas below. 

\smallskip

For $i\in I_0$, we define
\begin{align*}
	_iA &\seteq \{[i,t] \mid i \le t \le n  \} \subset \MS_n, \\
	A_i &\seteq \{[t,i] \mid 1 \le t \le i  \} \subset \MS_n.	
\end{align*}

\begin{lemma} \label{Lem: MS and hI}
Let $ i \in I_0 $. 
Let $S_{i,0}^-$, $S_{i,0}^+$ and $ S_{i,0} = \{\ca_{2n}, \ca_{2n-1}, \ldots, \ca_1 \}$ be the sets defined in {\rm\textbf{(Step 1)}} in \S\,{\rm\ref{Sec: crystal rule}}. 
\bni
\item 
For $1\le k\le n$, we have
\eqn
\ca_{2k-1}&&=\bl k, 2(i-1)+k-1\br.\\
\ca_{2k}&&=\bl k, 2(i-1)+k+1\br.
\eneqn
\item \label{it:ii} We have
\eqn
S_{i,0}\cap  \crhI_n^0&&=\st{\ca_k\mid 1\le k\le 2(n-i)+1},\\
S_{i,0}\cap  \crhI_n^1&&=\st{\ca_k\mid 2(n-i)+2\le k\le 2n}.
\eneqn

Moreover, we have, for any $\oep\in\st{0,1}$, 
\bna
\item for $1\le k\le n-i+\oep$,
$\ca_{2k-\oep}\in \crhI_n^0$ and
$\gamma_0^{-1}(\ca_{2k-\oep})=[i+1-\oep,k+i-\oep]$,

\item
for $n-i+\oep<k\le n$,
$\ca_{2k-\oep}\in \crhI_n^1$ and
$\gamma_1^{-1}(\ca_{2k-\oep})=[i+k-n-\oep,i-\oep]$.
\ee

\item \label{item:iv} We rearrange 
$$
_i A \cup { _{i+1} A} = \{ m_1, m_2, \ldots, m_{2n-2i+1} \}
$$
from left to right by largest to smallest in the order $<$ defined in \eqref{Eq: left order}, and rearrange
$$
A_{i-1} \cup {  A_i} = \{ m_1', m_2', \ldots, m_{2i-1}' \}
$$
from left to right by largest to smallest in the order $<'$ defined in \eqref{Eq: right order}. 
Then the sequence
$$
\left\{ \gamma_{1}(m_1'), \ldots, \gamma_{1}(m_{2i-1}'), \gamma_{0}(m_1),  \ldots, \gamma_{0}(m_{2n-2i+1}) \right\}
$$ 
is equal to the sequence $\{ \ca_{2n}, \ca_{2n-1}, \ldots, \ca_2, \ca_1\}$.

\ee
\end{lemma}	
\begin{proof}
Note that $S_{i,0} = S_{i,0}^- \cup S_{i,0}^+$ and  
$$
\gamma_0([a,b]) = (b-a+1, b+a-2)\quad \text{ and} \quad \gamma_1([a,b]) = (n-b+a, n+b+a-1)
$$ 
for any segment $[a,b]\in\MS_n$. 

\mnoi
(i) follows from the definitions of $S_{i,0}^\pm$ and $S_{i,0}$. 

\snoi
(ii) follows from (i) together with
\eqn \crhI_n^0&& = \st{ (i,a) \in \crhI_n \mid  1 \le i \le n,
\ i-1 \le a \le 2n-i-1  },\\
\crhI_n^1&& = \st{ (i,a) \in \crhI_n \mid  1 \le i \le n,
\ 2n-i + 1 \le a \le 2n+i-1  }.
\eneqn

\snoi
(iii) follows from (i).
\end{proof}

\begin{example}
We use the same notations given in Example \ref{Ex: conf with seg} and Lemma \ref{Lem: MS and hI}.  
\bni
\item Let $i=1$. In this case, $S_{1,0}^- = \{ (1,0), (2,1), (3,2) \}$, $S_{1,0}^+ = \{ (1,2), (2,3), (3,4) \}$ and 
$$
S_{1,0} =  \bst{ (3,4) \sck[0] (3,2) \sck[0]  (2,3)\sck[0](2,1)\sck[0](1,2) \sck[0] (1,0)   } \subset  \crhI_n^0 \cup \crhI_n^1.	
$$
Note that $\cru = (1,0)$ and $\cdual (\cru) = (3,4)$. Pictorially we write the segments $\gamma^{-1}(i,a)$ ($(i,a) \in S_{1,0}$) at the position $(i,a)$ as follows:
	$$ 
\scalebox{0.8}{\xymatrix@C=1ex@R=  0.0ex{ 
		i \backslash a     \ar@{-}[dddd]<3ex> \ar@{-}[rrrrrrrrrrrrrrrr]<-2.0ex>    & -5 & -4 & -3 & -2 & -1 & 0 &\  1 & \ 2& \ 3 &\  4 &\  5&\  6 & \ 7& \ 8 & & &   \\
		1       & &   & & \cdot  &  & [1]  & &  [2] & & \cdot && \cdot &&   \\
		2     &  & & \cdot & & \cdot &  &  [1,2]  & & [2,3]  && \cdot && \cdot  \\
		3      & & \cdot & & \cdot & &  \cdot  & &  [1,3] && \ul{[1]} && \cdot && \cdot \\
		& 
}}
$$	
Since $ {_1A} \cup {_{2}A} = \{ [1,3] > [2,3] >  [1,2]> [2]>[1]  \} $ and $ A_0 \cup A_1 = \{ [1]\}$, we have 
\begin{align*}
S_{1,0} &= \{ (3,4) ,(3,2) , (2,3) , (2,1) , (1,2) , (1,0)  \} \\
&= \{  \gamma_1( [1]),  \gamma_0([1,3]), \gamma_0( [2,3]),  \gamma_0([1,2]), \gamma_0([2]), \gamma_0([1]) \} .
\end{align*}	

\item Let $i=2$. In this case, $S_{2,0}^- = \{ (1,2), (2,3), (3,4) \}$, $S_{2,0}^+ = \{ (1,4), (2,5), (3,6) \}$ and 
$$
S_{2,0} =  \{ (3,6) \succ(3,4) \succ (2,5) \succ (2,3) \succ (1,4) \succ (1,2)    \} \subset  \crhI_n^0 \cup \crhI_n^1.	
$$
Note that $\cru = (1,2)$ and $\cdual (\cru) = (3,6)$. Pictorially we write the segments $\gamma^{-1}(i,a)$ ($(i,a) \in S_{2,0}$) at the position $(i,a)$ as follows:
$$ 
\scalebox{0.8}{\xymatrix@C=1ex@R=  0.0ex{ 
		i \backslash a     \ar@{-}[dddd]<3ex> \ar@{-}[rrrrrrrrrrrrrrrr]<-2.0ex>    & -5 & -4 & -3 & -2 & -1 & 0 &\  1 & \ 2& \ 3 &\  4 &\  5&\  6 & \ 7& \ 8 & & &   \\
		1       & &   & & \cdot  &  & \cdot  & &  [2] & & [3] && \cdot &&   \\
		2     &  & & \cdot & & \cdot &  &  \cdot  & & [2,3]  && \ul{[1,2]} && \cdot  \\
		3      & & \cdot & & \cdot & &  \cdot  & &  \cdot && \ul{[1]} && \ul{[2]} && \cdot \\
		& 
}}
$$	
Since $ {_2A} \cup {_{3}A} = \{ [2,3] > [3] >   [2]  \} $ and $ A_1 \cup A_2 = \{ [2] >' [1] >' [1,2]\}$, we have 
\begin{align*}
	S_{2,0} &= \{ (3,6) ,(3,4) , (2,5) , (2,3) , (1,4) , (1,2)  \} \\
	&= \{  \gamma_1( [2]),  \gamma_1([1]), \gamma_1( [1,2]),  \gamma_0([2,3]), \gamma_0([3]), \gamma_0([2]) \} .
\end{align*}	

\ee
	
\end{example}
\smallskip 

Recall that, for a multisegment $\cm$, $ \rS^{<}_i(\cm)$ and  $ \rS^{<'}_i(\cm)$ denote the  left and right  $i$-signature sequences of $\cm$ with respect to $<$ and $<'$ respectively, as described in Section \ref{Sec: ms}. 

\begin{lemma} \label{Lem: sig seq}
Let $\la \in \crB$ and write $\gamma^{-1}(\la) = (\cm_k)_{k\in \Z}$. For $i\in I_0$, 
the concatenation $ \rS^{<'}_i(\cm_{1}) *  \rS^{<}_i(\cm_0)$ is equal to the $(i,0)$-signature sequence of $\la$.
\end{lemma}
\begin{proof}
Let $S_{i,0} = \{\ca_{2n}, \ca_{2n-1}, \ldots, \ca_1 \}$ be the set defined in \textbf{(Step 1)} in Section \ref{Sec: crystal rule}. 	
 By  Lemma \ref{Lem: MS and hI}\;\ref{it:ii}, the signature of $\ca_k$ is equal to the $i$-signature of $\gamma^{-1}(\ca_k)$ for any $k$. 
Thus the assertion follows from Lemma \ref{Lem: MS and hI}\;\ref{item:iv}. 	
\end{proof}

\medskip 

We are now ready to prove Theorem \ref{Thm: CR}.

\smallskip 

\Proof[{Proof of Theorem \ref{Thm: CR}}]\quad
Thanks to \eqref{Eq: gamma hwt}, \eqref{Eq: gamma hw vector} and Lemma \ref{Lem: phi iso}, it suffices to show that the map $\gamma$ commutes with the crystal operators $ \tF_{i,k}$.

 Let $(i,k)\in \cIz$, $\la\in \crB$ and set $\cb  \seteq \gamma^{-1}(\la) $. 
We shall prove 
$$
\gamma (\tF_{i,k} (\cb)) = \tF_{i,k}  ( \la).
$$
Thanks to \eqref{Eq: gamma D}, Lemma~\ref{Lem: cD} and Lemma~\ref{lem:dualcr}\;\ref{it:DF},
 we may assume that $k=0$. 

\smallskip

We write $ \cb = (\cm_k)_{k\in \Z}$ for $\cm_k \in \MS_n$. 
Let $S_{i,0} = \{\ca_{2n}, \ca_{2n-1}, \ldots, \ca_1\}$ be the set defined in \textbf{(Step 1)} in Section \ref{Sec: crystal rule}. 
Let $\ca_t$ be the element of $S_{i,0}$ corresponding to the leftmost $+$ in the reduced $(i,0)$-signature sequence of $\la$.
If there is no such an $\ca_t$, then we set $\ca_t \seteq 0$.

On the other hand, we set 
$$
 \rS' \seteq \rS^{<'}_i(\cm_{1}) \quad \text{and}\quad \rS \seteq \rS^{<}_i(\cm_{0})
 $$
 and let $\brS $ (resp.\ $\brS'$) be the sequence obtained from $\rS$ (resp.\ $\rS'$) by canceling out all $(+,-)$ pairs.
 Let $\overline{\rS' * \rS}$  be the sequence obtained from the concatenation $\rS' * \rS$ by canceling out all $(+,-)$ pairs. 
By the crystal rule for $\MS_n$ described in Section \ref{Sec: ms},  $\ep_{i,0}(\cb) $ (resp.\ $\eps_{i,1}(\cb)$) is equal to the number of $-$'s (resp.\ $+$'s) in $\brS$ (resp.\ $\brS'$). 
Pictorially, the concatenation of $ \brS' * \brS$ can be written as follows:
\begin{align} \label{Eq: seq}
\underbrace{ - \cdots \cdots  - \  \overbrace{+ \cdots \cdots  +}^{\eps_{i,1}(\cb)} }_{ \brS'} \ \ 
\underbrace{ \overbrace{- \cdots \cdots  -}^{\ep_{i,0}(\cb)} \  + \cdots \cdots  +   }_{ \brS}
\end{align}

\mnoi
\textbf{(Case 1)} Suppose that $ \eps_{i,1}(\cb) > \ep_{i,0}(\cb)$. Let $[s,i]$ be the segment placed at the leftmost $+$ in the sequence $ \brS'$. By the definition of the extended crystal operator $\tF_{i,0}$, we have 
$$
\tF_{i,0}(\cb) = (\cm_k')_{k\in \Z},
$$
 where $\cm_k' = \cm_k$ for $k\ne 1$ and 
\begin{align} \label{Eq: tes cm1}
\cm_1' = \tes_i (\cm_1) = \cm_1 - [s,i] + [s,i-1].
\end{align}

Since $  \eps_{i,1}(\cb) > \ep_{i,0}(\cb)$, considering the configuration \eqref{Eq: seq}, 
$[s,i]$ is equal to the segment placed at the  the leftmost $+$ in $\overline{\rS' * \rS}$.
Thanks to Lemma \ref{Lem: MS and hI}\;\ref{item:iv} and Lemma \ref{Lem: sig seq}, $ \gamma_1([s,i])$ is equal to the element $\ca_t$ of $S_{i,0}$ corresponding to the leftmost $+$ in the reduced $(i,0)$-signature sequence of $\la$. 
Thus, by Lemma \ref{Lem: MS and hI} and \eqref{Eq: tes cm1}, we have 
\begin{align*}
\tF_{i,k}  ( \la) &= \la - \ca_t + \ca_{t+1} \\
&= \gamma(\cb) - \gamma_1([s,i]) + \gamma_1([s,i-1]) \\
&= \gamma ( \ldots, \cm_2, \cm_1 - [s,i] + [s,i-1], \cm_0, \ldots ) \\
&= \gamma ( \ldots, \cm_2, \tes_i( \cm_1), \cm_0, \ldots ) \\
&= \gamma (\tF_{i,k} (\cb)).	
\end{align*}

\mnoi
\textbf{(Case 2)} Suppose that $ \eps_{i,1}(\cb) \le \ep_{i,0}(\cb)$. Let $ \theta \seteq [i+1,s]$ be the segment placed at the leftmost $+$ in the sequence $ \brS$ if it exists. Otherwise, we set $\theta \seteq \emptyset$. By the definition of the extended crystal operator $\tF_{i,0}$, we have 
$$
\tF_{i,0}(\cb) = (\cm_k')_{k\in \Z},
$$
where $\cm_k' = \cm_k$ for $k\ne 0$ and 
\begin{align} \label{Eq: tes cm}
	\cm_0' = \tf_i (\cm_0) = 
	\begin{cases}
			\cm_0 - \theta + [i,s] & \text{ if } \theta \ne \emptyset,  \\
		\cm_0 + [i] & \text{ if } \theta = \emptyset.
	\end{cases}
\end{align}

Since $  \eps_{i,1}(\cb) \le \ep_{i,0}(\cb)$, when $\theta \ne \emptyset$, 
$\theta$ is equal to the segment placed at the  the leftmost $+$ in $\overline{\rS' * \rS}$  as seen by the figure \eqref{Eq: seq}.
By Lemma \ref{Lem: MS and hI} \ref{item:iv} and Lemma \ref{Lem: sig seq},  
$$
\theta = \emptyset \Longleftrightarrow \ca_t = 0,
$$
 and if $\theta \ne \emptyset$, then
$ \gamma_0([i+1, s])$ is equal to the element $\ca_t$ of $S_{i,0}$ corresponding to the leftmost $+$ in the reduced $(i,0)$-signature sequence of $\la$. 
Thus,  Lemma \ref{Lem: MS and hI} and \eqref{Eq: tes cm} imply the following: if $\theta \ne \emptyset$, then 
\begin{align*}
	\tF_{i,k}  ( \la) &= \la - \ca_t + \ca_{t+1} \\
	&= \gamma(\cb) - \gamma_0([i+1,s]) + \gamma_0([i,s]) \\
	&= \gamma ( \ldots, \cm_1, \cm_0 - [i+1,s] + [i,s], \cm_{-1}, \ldots ) \\
	&= \gamma ( \ldots, \cm_1,  \tf_i( \cm_0), \cm_{-1}, \ldots ) \\
	&= \gamma (\tF_{i,k} (\cb)),	
\end{align*}
and if $\theta = \emptyset,$ then 
\begin{align*}
	\tF_{i,k}  ( \la) &= \la + \ca_1 \\
	&= \gamma(\cb) + \gamma_0([i])  \\
	&= \gamma ( \ldots, \cm_1, \cm_0 + [i], \cm_{-1}, \ldots ) \\
	&= \gamma ( \ldots, \cm_1,  \tf_i( \cm_0), \cm_{-1}, \ldots ) \\
	&= \gamma (\tF_{i,k} (\cb)).	
\qh\end{align*}
\QED

\end{document}